\documentclass{article}
\usepackage{amsmath}
\usepackage[a4paper]{geometry}
\usepackage{amsthm}
\usepackage{amssymb}
\usepackage{graphicx}
\usepackage{hyperref}
\usepackage{enumitem}
\usepackage{color}
\usepackage{comment}
\usepackage{setspace}
\usepackage{overpic}
\usepackage[affil-it]{authblk}

\newcommand{\twopartdef}[4]
{
	\left\{
	\begin{array}{ll}
		#1 & \mbox{if } #2 \\
		#3 & \mbox{if } #4
	\end{array}
	\right.
}

\newcommand{\bN}{\mathbb{N}}
\newcommand{\bbN}{\mathbb{N}}

\newcommand{\bZ}{\mathbb{Z}}

\newcommand{\cH}{\mathcal{H}}
\newcommand{\cL}{\mathcal{L}}
\newcommand{\cN}{\mathcal{N}}
\newcommand{\cF}{\mathcal{F}}

\newcommand{\cP}{\mathcal{P}}
\newcommand{\cS}{\mathcal{S}}
\newcommand{\cT}{\mathcal{T}}
\newcommand{\cU}{\mathcal{U}}

\newcommand{\Aut}{\operatorname{Aut}}
\newcommand{\CAT}{\operatorname{CAT}}
\newcommand{\Sym}{\operatorname{Sym}}

\newcommand{\id}{\operatorname{id}}

\newcommand{\COS}{\operatorname{COS}}

\newcommand{\Int}{\operatorname{Int}}
\newcommand{\normal}{\trianglelefteq}

\newcommand{\axis}{\operatorname{axis}}

\newcounter{ClaimCounter}
\newtheorem{theorem}{Theorem}[section]
\newtheorem{proposition}[theorem]{Proposition}
\newtheorem{lemma}[theorem]{Lemma}
\newtheorem{corollary}[theorem]{Corollary}

\newtheorem{claim}[ClaimCounter]{Claim}
\theoremstyle{definition}
\newtheorem{definition}[theorem]{Definition}
\newtheorem{example}[theorem]{Example}
\theoremstyle{remark}
\newtheorem{rem}[theorem]{Remark}
\newenvironment{remark}{\begin{rem}}{\end{rem}}

\title{Directions and scale-multiplicative semigroups in restricted Burger-Mozes groups}
\author{Timothy P.~Bywaters\\
School of Mathematics and statistics\\
The University of Sydney\\
NSW 2006 Australia\\
t.bywaters@maths.usyd.edu.au}
\date{}

\begin{document}
\maketitle

\begin{abstract}
	We study the scale function, space of directions and scale-multiplicative semigroups for restricted Burger-Mozes groups. We relate these general notions to intrinsic properties of the group. Among other things, we give a formula for the scale function; relate the space of directions to both the action on the tree and an action on a $\CAT(0)$ cube complex; and construct maximal scale-multiplicative semigroups from the space of directions.
\end{abstract}
\tableofcontents
\section{Introduction}
The scale function and tidy subgroups for totally disconnected locally compact (t.d.l.c.) groups were originally introduced in \cite{Willis94} as a tool to resolve a conjecture in \cite{Hofmann81}. Since then, ideas surrounding theses concepts have grown and found applications in other areas of mathematics including random walks, ergodic theory and dynamical systems \cite{Dani06, Jaworski96, Previts03, Willis14}, arithmetic groups \cite{Shalom13} and Galois Theory \cite{Chat12}. These are applications not only of the scale function and tidy subgroups but of structural invariants derived from them. The space of directions and maximal scale-multiplicative semigroups are two such structural invariants.

The space of directions of a t.d.l.c. group is a notion of boundary built from the group action on its compact open subgroups. Originally defined in \cite{Baumgartner06} where it was shown that the boundary can be used to detect flat subgroups. These are analogues of geometric flats or apartments in buildings. On the other hand, hyperbolic groups, which are very far from having any geometric flats, are shown to have a discrete space of directions in \cite{Bywaters19}.

Scale-multiplicative semigroups are subsemigroups for which the scale function is a homomorphism. All such semigroups are contained within a maximal scale-multiplicative semigroup. These objects are defined and studied in \cite{Baumgartner15}. The authors also determine all maximal scale-multiplicative semigroups of the automorphism group of a regular tree. It is seen that these semigroups correspond to geometric properties of the tree. However, without further examples it is difficult to see how these results generalise to arbitrary t.d.l.c. groups, especially those which are not so closely linked to a geometric object.

We study the scale function, tidy subgroups, space of directions, and scale-multiplicative semigroups for restricted Burger-Mozes groups. We specifically aim to give meaning to these concepts in terms of the group itself. This is to increase our understanding of what information is encoded by these concepts. By studying restricted Burger-Mozes groups, we also gain insight into how results for the automorphism group of a regular tree, specifically those given in \cite{Baumgartner15}, may generalise to a larger class in t.d.l.c. groups. Restricted Burger-Mozes groups are a natural example to consider as they are still tree-like, they act as automorphisms on a regular tree, but the topology differs from the permutation topology. This topology is a result of a general construction which is also used to define families of almost automorphism groups, an example of which is Neretin's group. Consequently, studying restricted Burger-Mozes groups may also help to understand these other families of groups.

Many of our results require a good understanding of the scale function on restricted Burger-Mozes groups. This understanding is developed in Section \ref{sec:scale_on_G(F)}. Here we calculate the scale of a general element. Our arguments are based on the tidying procedure found in \cite{Willis01}. We then use our results to investigate the uniscalar elements of restricted Burger-Mozes groups. In particular, we characterise when the notions of uniscalar and elliptic (as a tree automorphism) coincide.

Section \ref{sec:leboudecDirections} contains results concerning the space of directions for restricted Burger-Mozes groups. We give two characterisations for when two group elements move in the same direction, that is they are asymptotic. The first involves the action of the group of a $\CAT(0)$-cube complex which was defined by \cite{Boudec15}. The other is in terms of a function which captures the asymptotic nature of the local action (on the tree) of a given element. The section concludes by using these characterisations to show that the topology on the space of directions is discrete.

The focus of Section \ref{sec:s_mult_leboudec} is on constructing maximal scale-multiplicative semigroups from asymptotic classes. To do so, the concepts from previous sections are built upon and it is shown that an asymptotic class union the uniscalar component of its stabiliser gives a maximal scale-multiplicative semigroup. In doing so we generalise some results specific of tree automorphism groups but using notions that are defined for any t.d.l.c. group. Thus, we gain insight into how maximal scale-multiplicative groups may be constructed from asymptotic classes in arbitrary t.d.l.c. groups.
\subsection*{Acknowledgements}
I am grateful to Udo Baumgartner, Jacqui Ramagge, Stephan Tornier and George Willis for fruitful discussions, suggestions and recommendations with respect to earlier drafts. This work was undertaken while the author was supported by the Australian Government Research Training Program.
\subsection*{Conventions}
Let $\bN = \{1,2,3,4,\ldots\}$ be the natural numbers and $\bN_0 = \{0\}\cup \bN$. For a group $G$ acting on a space $X$ and $V\subset X$, we let $G_{V} = \{g\in G\mid g(v) = v \hbox{ for all }v\in V\}$ and $G_{\{V\}} = \{g\in G\mid g(V) = V \}$. For $S\subset G$ and $g\in G$, we occasionally write $S^g := gSg^{-1}$.
\section{Preliminaries}
\subsection{Restricted Burger-Mozes groups}
After establishing out notation for graphs and their automorphisms, we define restricted Burger-Mozes groups and give some preliminary results. These are groups which act as automorphisms on a regular tree with the restriction that the local action of an element is specified almost everywhere by a fixed permutation group. They do not come equipped with the permutation topology but with a variation which is reminiscent of Neretin's group of almost automorphisms. These groups were studied in depth in \cite{Boudec15}. We adopt the terminology of restricted Burger-Mozes groups from \cite{Caprace17} which describes the relation between restricted Burger-Mozes groups and Burger-Mozes groups as analogous to the relation between restricted direct products and direct products. Justification for most of what is said here can be found in  \cite{Burger00}, \cite{Boudec15} and \cite{garrido2018}.
\subsubsection{Graphs and tree automorphisms}
\label{sec:graph_and_auto}
It is important that we establish our language for graphs. For us, a \emph{directed graph} $\Gamma$ is a disjoint union $V(\Gamma)\sqcup E(\Gamma)$ of a vertex set $V(\Gamma)$ and edge set $E(\Gamma)\subset V(\Gamma)\times V(\Gamma)$. We denote by $o,t:E(\Gamma)\to V(\Gamma)$ the projections onto the first and second components, the \emph{origin} and \emph{terminus}, of an edge. We require that our edge set is antisymmetric, that is, $(u,v)\in E(\Gamma)$ implies $u \neq v$. Geometrically, this means that $\Gamma$ is without loops.

Two vertices $v,w\in V(\Gamma)$ are \emph{adjacent} if either $(v,w)\in E(\Gamma)$ or $(w,v)\in E(\Gamma)$, equivalently the adjacency relation is the symmetric closure of the edge relation. A \emph{path} of length $k\in \bbN$ from $v\in V(\Gamma)$ to $v'\in V(\Gamma)$ is a sequence of vertices $(v=v_{0},\ldots,v_{k}=v')$ of $\Gamma$ such that $v_{i}$ and $v_{i+1}$ are adjacent for all $i\in\{0,\ldots,k-1\}$ and $v_{i}\neq v_{i+2}$ for all $i\in\{0,\ldots,k - 2\}$. We also allow \emph{infinite} and \emph{bi-infinite paths} which are paths indexed by $\bN$ and $\bZ$ respectively. A directed graph $\Gamma$ is \emph{connected} if for all $v,w\in V(\Gamma)$ there is a path from $v$ to $w$. It is a \emph{tree} if it is connected and the path joining any two vertices is unique. The \emph{distance} between two vertices $v_0,v_1\in V(\Gamma)$ is the length of the shortest path between them or $\infty$ if such a path does not exist. This defines a metric $d:V(\Gamma)\times V(\Gamma)\to \bN_0$ if $\Gamma$ is connected.

Two infinite paths in $\Gamma$ are \emph{equivalent} if they intersect in an infinite path. When $\Gamma$ is a tree, this is an equivalence relation on infinite paths and the \emph{boundary} $\partial\Gamma$ of $\Gamma$ is the set of these equivalence classes. 

A directed graph $\Gamma'$ is a \emph{subgraph} of $\Gamma$ if $V(\Gamma')\subset V(\Gamma)$ and $E(\Gamma')\subset E(\Gamma)$. For a subset $A\subseteq V(\Gamma)$, the \emph{subgraph of $\Gamma$ spanned by $A$} is the directed graph with vertex set $A$ and edge set $\{(v,w)\in E(\Gamma)\mid v,w\in A\}$.

A \emph{morphism} between directed graphs $\Gamma_{0}$ and $\Gamma_{1}$ is a map $\alpha_{V}:V(\Gamma_{0})\to V(\Gamma_1)$ such that for every $(v_0,v_1)\in E(\Gamma_0)$ we have $(\alpha(v_0),\alpha(v_1))\in E(\Gamma_{1})$ or $\alpha(v_0) = \alpha(v_1)$. We call $\alpha$ an \emph{isomorphism} if has an inverse which is also a morphism and an \emph{automorphism} if $\alpha$ is an isomorphism and $\Gamma_0 = \Gamma_1$. It is routine to show that the collection of automorphisms of a directed graph $\Gamma$ forms a group under composition which we denote by $\Aut(\Gamma)$.

An \emph{undirected graph} $\Gamma$, or simply \emph{graph}, is a directed graph together with a fixed-point-free involution of $E(\Gamma)$, denoted by $e\mapsto\overline{e}$, such that $o(\overline{e})=t(e)$ and $t(\overline{e})=o(e)$ for all $e\in E$. Alternatively, since our edges are determined by their origin and terminus, an undirected graph is a graph where the edge and adjacency relation coincide. Here, $\overline{(u,v)}$ is defined to be $(v,u)$ for all $(u,v)\in E(\Gamma)$. For $v\in V(\Gamma)$, we let $E(v):=o^{-1}(v)=\{e\in E(\Gamma)\mid o(e)=v\}$ be the set of edges issuing from $v$. The \emph{valency} of $x\in V(\Gamma)$ is $|E(x)|$. We say $\Gamma$ is \emph{regular} if $\deg(\Gamma):= |E(v)|$ does not depend on $v\in V(\Gamma)$. Note that a directed graph yields an undirected graph by passing to the symmetric closure of the edge set. Subgraphs of undirected graphs are required to be undirected. A subgraph $\cT$ of a regular undirected tree $T$ is \emph{complete} if $v\in V(\cT)\subset V(T)$ implies either all vertices adjacent to $v$ in $T$ are vertices in $\cT$ or the valency of $v$ in $\cT$ is $1$. The \emph{internal vertices} of $\cT$ are precisely those for which $E(v)\subset E(\cT)$. The set of internal vertices of $\cT$ is denoted by $\Int(\cT)$.

Occasionally, we will need control over the size of spheres in a tree. 
\begin{lemma}\label{lem:size_of_balls}
	Suppose $T$ is an infinite regular tree and $e\in E(T)$. Then 
	\[|\{v\in V(T)\mid d(v,e) = k\}|= 2(\deg(T)-1)^{k}.\]
\end{lemma}
\begin{proof}
	Note $|\{v\in V(T)\mid d(v,e) = 0\}|= 2$ as the vertices distance $0$ from $e$ are precisely $o(e)$ and $t(e)$. If $d(v,e) = k$, then there are exactly $\deg(T) - 1$ vertices adjacent to $v$ that are $k+1$ away from $e$. Thus,
	\[|\{v\in V(T)\mid d(v,e) = k+1\}| = (\deg(T) - 1)|\{v\in V(T)\mid d(v,e) = k\}|.\]
	The result follows via induction.
\end{proof}

Every automorphism $g$ of a locally finite regular tree $T$ can be classified into two types. First, $g$ is \emph{elliptic} if there exists $V\subset V(T)$ finite, such that $g(V) = V$. It can be seen that $V$ can be chosen as either a single vertex or a pair of adjacent vertices. Second, $g$ is \emph{hyperbolic} if $g$ is not elliptic and there exist two unique boundary points fixed by $g$. If $g$ is hyperbolic, then there exists a bi-infinite path $\axis(g) = (\ldots, v_{-1},v_0,v_1,\ldots)$, called the \emph{axis} of $g$, such that $g(\axis(g)) = \axis(g)$. For any $g\in\Aut(T)$ define
\begin{equation*}
\label{eq:length_function}
l(g): = \min_{p\in V(T)\cup E(T)}d_\cH(g(p),p),
\end{equation*}
where $d_{\cH}$ is the Hausdorff distance induced by the distance on the vertices of $T$. Equivalently $l(g) = 0$ if $g$ is elliptic and if $g$ is hyperbolic $l(g) = d(v,g(v))$ for some (hence any) $v\in\axis(g)$.
When $g$ is hyperbolic, the condition $d(v,g(v)) = l(g)$ characterises vertices $v\in\axis(g)$. By identifying $\axis(g)$ with $\bZ$, we can place a total order $\le_g$ on the vertices of $\axis(g)$. For $u,v\in \axis(g)$, we say that $u\le_g v$ if for some $k\in\bN_0$, $v$ is on the path from $u$ to $g^k(u)$. We say $u <_g v$ if $u\le_g v$ and $u\neq v$. Intuitively, vertices are larger with respect to $\le_g$ in the direction that $g$ translates. The two unique ends fixed by $g$ can be recovered in the following way: Choose $v\in \axis(g)$. Then there exists a unique minimal infinite path containing $g^{k}(v)$ for all $k\in\bN_0$. Set $\omega_{+}(g)$ to be the equivalence class of this path. This is one of the ends fixed by $g$. The other is $\omega_{-}(g) := \omega_{+}(g^{-1})$. Given $v\in V(T)$, let $\pi_g(v)$ denote the unique vertex in $\axis(g)$ which has minimal distance form $v$. Then $d(v, \axis(g)) = d(v,\pi_g(v))$. Since $g$ acts as a translation on $\axis(g)$, it can be shown that $\pi_gg(v) = g\pi_g(v)$.

The following lemma can be shown using the axis of translation of the individual hyperbolic automorphisms.

\begin{lemma}\label{lem:dist_to_axis_seq}
	Suppose $g$ and $h$ are automorphisms of a locally finite regular tree $T$ and suppose $v\in V(T)$. Set $d_n = d(h^{-n}(v), \axis(g))$. Then $d_n$ is eventually non-decreasing and is bounded if and only if $\omega_{-}(h)\in \{\omega_{\pm}(g)\}$.
\end{lemma}

\subsubsection{Restricting the local action of tree automorphisms}

Suppose $T$ is a locally finite regular tree. Let $\Omega$ be the set $\{0,1,\cdots, \deg(T)-1\}$. A \emph{legal colouring} of $T$ is a map $c:E(T)\to \Omega$ such that $c(e) = c(\bar{e})$ and  $c|_{E(v)}:E(v)\to \Omega$ is a bijection for each $v\in V(T)$. Suppose $g\in \Aut(T)$ and $v\in V(T)$. Then $g$ acts as a bijection from $E(v)$ to $E(g(v))$. This induces a permutation $\sigma(g,v)\in\Sym(\Omega)$ by
\[\sigma(g,v)(a) = cgc|_{E(v)}^{-1}(a).\]
We call $\sigma(g,v)$ the \emph{local action} of $g$ at $v$. Lemma \ref{lem:sig_basic} summarises the relation between $\sigma(\cdot,\cdot)$ and the group operations in $\Aut(T)$.

\begin{lemma}\label{lem:sig_basic}
	Suppose $g,h\in \Aut(T)$ and $v\in V(T)$. Then $\sigma(gh,v) = \sigma(g,h(v))\sigma(h,v)$. In particular $\sigma(g,v)^{-1} = \sigma(g^{-1},g(v))$ and
	\[\sigma(ghg^{-1},g(v)) = \sigma(g, h(v))\sigma(h,v)\sigma(g,v)^{-1}.\]
\end{lemma}

Fix a subgroup $F\le\Sym(\Omega)$. The \emph{Burger-Mozes} group associated to $F$, studied in depth in \cite{Burger00}, is given by
\[U(F):=\{g\in\Aut(T)\mid \sigma(g,v)\in F \hbox{ for all }v\in T\}\]
and is independent of choice of legal colouring. Furthermore, $U(F)$ is a closed subgroup of $\Aut(T)$ and is itself a t.d.l.c. group with the subspace (equivalently permutation) topology. It is discrete if and only if the action of $F$ on $\Omega$ is free. For $F$ fixed, we often need to refer to $U(F)_{\{V\}}$ and $U(F)_V$ for $V\subset T$. We make the abbreviations $U_{\{V\}}:=U(F)_{\{V\}}$ and $U_{V}:=U(F)_{V}$. Note that
\[\cU:= \{U_{\cF}\mid\cF\subset V(T) \hbox{ is finite}\}\]
is a basis of neighbourhoods at the identity for $U(F)$. Suppose we now have the following sequence of subgroups:
\[F\le F'\le \Sym(\Omega).\]
We define \emph{restricted Burger-Mozes groups} associated to $F$ and $F'$ as
\[G(F,F'): = \left\{g\in\Aut(T)\mid \sigma(g,v)\in F \hbox{ for all but finitely many }v\in V(T)\right\}\cap U(F').\]

Following from a general argument, see \cite[Section 3]{Bourbaki98} and \cite[Lemma 3.2]{Boudec15}, there exists a unique group topology on $G(F,F')$ such that the inclusion $U(F)\to G(F,F')$ is open and continuous. This topology has $\cU$ as a basis or neighbourhoods at the identity. Although they have the same neighbourhood basis, the topological properties of $G(F,F')$ and $U(F)$ are quite different. For example the action of $G(F,F')$ on $T$ is proper if and only if $G(F,F') = U(F)$. If the action is not proper, $G(F,F')$ does not have maximal compact open subgroups. In this case $G(F,F')$ cannot act properly and cocompactly on any simply connected metric space.

For a given $g\in G(F,F')$ it is useful  to identify the finite set of vertices for which the local action of $g$ is not in $F$. This is the purpose of Definition \ref{def:singularities}.
\begin{definition}\label{def:singularities}
	For $g\in G(F,F')$ we define the set of \emph{singularities} of $g$ to be
	\[S(g):=\{v\in V(T)\mid\sigma(g,v)\not\in F\}.\]
\end{definition}
\noindent It follows from Lemma \ref{lem:sig_basic} that $S(g^{-1}) = gS(g)$.

For $F\le \Sym(\Omega)$ define $\hat{F}\le\Sym(\Omega)$ to be the largest subgroup that preserves the orbits of $F$, known as a \emph{Young subgroup}. Then $\hat{F}$ is isomorphic to the direct product of symmetric groups over the orbits of $F$. It is shown in \cite{Boudec15} that $G(F,F')\le U(\hat{F})$ and $G(F,F') = U(F)$ if and only if $F = \hat{F}\cap F'$. 

For a subtree $A$ of $T$, a collection of permutations $\{\sigma_v\in F'|v\in V(A)\}$ is said to be consistent if $\sigma_v(c(v,u)) = \sigma_u(c(v,u))$ whenever $(v,u)$ forms an edge. Observe that for any $g\in G(F,F')$ and subtree $A$ of $T$, the set $\{\sigma_v = \sigma(g,v)\mid v\in V(A)\}$ is consistent. The following Lemma is a slight generalisation of \cite[Lemma 3.4]{Boudec15} and can be proved in an identical manner.

\begin{lemma}\label{lem:autExtension}
	Suppose $F\le F'\le \hat{F}$, $A$ is a subtree of $T$ and $\{\sigma_v\in F'\mid v\in V(A)\}$ is a consistent set of permutations such that  $\sigma_v\not\in  F$ for only finitely many $v\in V(A)$. Then for any $u\in V(A)$ and $u'\in V(T)$, there exists $g\in G(F,F')$ such that
	\begin{enumerate}[label = (\roman{*})]
		\item $g(u) = u'$;
		\item $\sigma(g,v) = \sigma_v$ for all $v\in V(A)$; and 
		\item $S(g)\subset A$. 
	\end{enumerate}
\end{lemma}
\begin{corollary}\label{cor:2-trans_translations}
	Suppose $F$ is $2$-transitive and $P$ is a bi-infinite path in $T$. Then there exists $h\in U(F)$ with $l(h) = 1$ and $\axis(h) = P$.
\end{corollary}
\begin{proof}
	Label $P = (\ldots, v_{-1}, v_0, v_1,\ldots )$. For each $i\in\bZ$, using $2$-transitivity of $F$, choose $\sigma_{v_i}\in F$ such that 
	\[\sigma_{v_i}(c(v_{i - 1}, v_i)) = c(v_i, v_{i+1})\vspace{0.2cm}\hbox{ and }\vspace{0.2cm}\sigma_{v_i}(c(v_i,v_{i+1})) = c(v_{i+1},v_{i+2}).\]
	Then $\{\sigma_v\mid v\in P\}$ is a consistent set of permutations. Lemma \ref{lem:autExtension} gives $h\in U(F)$ such that $h(v_0) = v_1$ and $\sigma(h,v_i) = \sigma_{v_i}$. We show that $h(v_i) = v_{i+1}$ for all $i\in\bZ$. Observe that we must have $h(v_1) =  v_2$ since $h(v_0) = v_1$ and $\sigma(h,v_1) = \sigma_{v_1}$. Similarly, $h(v_{-1}) = v_0$ since $h(v_0) = v_1$ and $\sigma(h,v_0) = \sigma_{v_0}$. Inductively, we see that $h(v_i) = v_{i+1}$ for all $i\in\bZ$.
\end{proof}
\subsection{Scale function and tidy subgroups}
The concepts of scale function and tidy subgroups were introduced in \cite{Willis94} but have been built up in \cite{Willis01} and \cite{Willis15}. We give the basic definitions and properties.

Results in \cite{vDa31} show that any t.d.l.c. group has a basis at the identity consisting of compact open subgroups. 
\begin{definition}
	Let $G$ be a t.d.l.c. group. Set 
	\[\COS(G) = \{U\le G \mid U \hbox{ is compact and open}\}.\]
\end{definition}
\begin{definition}
	Let $G$ be a t.d.l.c. group and $\alpha\in\Aut(G)$. We define the \emph{scale} of $\alpha$ to be the natural number
	\[s(\alpha):= \min\{[\alpha(U):\alpha(U)\cap U]\mid U\in\COS(G)\}.\]
	Any compact open subgroup for which $s(\alpha) = [\alpha(U):\alpha(U)\cap U]$ is called \emph{tidy} for $\alpha$. An automorphism $\alpha$ is said to be \emph{uniscalar} if $s(\alpha) = 1 = s(\alpha^{-1})$.
	
	The scale induces a function $s:G\to \bbN$, which we also call the \emph{scale}, by setting the scale of $g\in G$ to be the scale of the inner automorphism corresponding to $g$. 
\end{definition}
When a subgroup is tidy is classified by structural properties of the subgroup. These properties are defined in terms of subgroups derived from the original which depend on the given automorphism. For $G$ a t.d.l.c. group, $\alpha\in \Aut(G)$ and $U\in\COS(G)$, set 
\begin{alignat*}{3}
&U_+ = \bigcap_{n\ge 0} \alpha^n(U) \hbox{, } &U_{++} = \bigcup_{n\ge 0}\alpha^n(U_+),\\
&U_- = \bigcap_{n\le 0} \alpha^n(U)\quad \hbox{and} \quad&U_{--} = \bigcup_{n\le 0}\alpha^n(U_-).
\end{alignat*}
\begin{theorem}\label{thm:tid_min}
	Let $G$ be a t.d.l.c. group and $\alpha\in \Aut(G)$. A subgroup $U$ is tidy for $\alpha$ if and only if
	\begin{enumerate}
		\item $U = U_+U_-$. In this case we say $U$ is \emph{tidy above}.
		\item $U_{--}$ (or equivalently $U_{++}$ if $U$ is already tidy above) is closed. In this case we say $U$ is \emph{tidy below}.
	\end{enumerate}
\end{theorem}

The proof of Theorem \ref{thm:tid_min} relies on a process called a tidying procedure. This procedure takes in as input a compact open subgroup and automorphism and produces a subgroup a tidy subgroup for said automorphism. In Section \ref{sec:scale_on_G(F)}, we follow a tidying procedure to calculate the scale function on restricted Burger-Mozes groups.

Theorem \ref{thm:tid_min} and the associated tidying procedure are instrumental in proving the following properties of tidy subgroups and properties of the scale function.

\begin{proposition}[Properties of tidy subgroups]\label{prop:prop_of_tid}
	Suppose $G$ is a t.d.l.c. group with $U\in\COS(G)$ tidy for $\alpha\in\Aut(G)$. Then
	\begin{enumerate}[label = (\roman*)]
		\item $U$ is tidy for $\alpha^n$ for all $n\in\bZ$;
		\item If $V\in\COS(G)$ is tidy for $\alpha$, then $U\cap V$ is also tidy for $\alpha$;
		\item $\alpha^n(U)$ is tidy for $\alpha$ for all $n\in\bZ$.
	\end{enumerate}
\end{proposition}

\begin{proposition}[Properties of the scale function]\label{prop:scale_prop}
	Suppose $G$ is a t.d.l.c. group and $\alpha\in\Aut(G)$. Then
	\begin{enumerate}[label= (\roman*)]
		\item $s(\alpha) = 1$ if and only if there exists $U\in\COS(G)$ such that $\alpha(U)\le U$;
		\item $\alpha$ is uniscalar if and only if there exists $U\in\COS(G)$ such that $\alpha(U) = U = \alpha^{-1}(U)$;
		\item $s(\alpha^n) = s(\alpha)^n$ for all $n\in\bbN_0$;
		\item If $\beta\in\Aut(G)$, then $s(\alpha) = s(\beta\alpha\beta^{-1})$.	
	\end{enumerate}
	Furthermore, the scale function on $G$ is continuous and $\Delta(g) = s(g)/s(g^{-1})$, where $\Delta$ is the modular function on $G$.
\end{proposition}

\subsection{Scale-multiplicative semigroups}
\label{sec:s_mult_semigroups}

Scale-multiplicative semigroups are a recent attempt to extract geometric information from a t.d.l.c. group. They can be viewed as a refinement of flat groups, see \cite{Willis04}, a fact that was studied and exploited in \cite{Praeger17} to generalise the results from \cite{Moe02} to the higher rank case. The automorphism group of a regular tree as an enlightening example which we expand upon in Example \ref{examp:s_mult_tree}. 
\begin{definition}
	Suppose $G$ is a t.d.l.c. group. A semigroup $\cS\subset G$ is scale multiplicative if $s(gh) = s(g)s(h)$ for all $g,h\in \cS$. 
\end{definition}
\begin{example}
	It follows from Proposition \ref{prop:scale_prop} that the semigroup generated by a single $g\in G$ is scale-multiplicative.
\end{example}

\begin{definition}
	Suppose $G$ is a t.d.l.c. group, $H\le G$ and $\cS\subset H$ a scale-multiplicative semigroup. We say $\cS$ is \emph{maximal} in $H$ if for any other scale-multiplicative semigroup $\cS'\subset H$, we have $\cS\subset \cS'$ if and only if $\cS = \cS'$. If $H = G$, we say $\cS$ is \emph{maximal}.
\end{definition}
Proposition \ref{prop:prop_s_mult} is a collection of results from \cite{Baumgartner15}.
\begin{proposition}\label{prop:prop_s_mult}
	Suppose $G$ is a t.d.l.c. group and $\cS\subset G$ is scale-multiplicative semigroup. Then:
 	\begin{enumerate}[label = (\roman{*})]
 		\item $S$ is contained within a maximal scale-multiplicative semigroup of $G$;
 		\item If $S$ is a maximal, then $S$ is closed and contains the identity;
 		\item $S^{-1}$ is scale-multiplicative and $S\cap S^{-1}$ is uniscalar.
	\end{enumerate}
\end{proposition}

\begin{example}[{\cite[Section 4]{Baumgartner15}}]\label{examp:s_mult_tree}
	Suppose $T$ is a regular tree of degree at least $3$ and $G\le \Aut(T)$ a subgroup which acts $2$-transitively on $\partial T$.  We list the maximal scale-multiplicative semigroups of $G$. These are of four possible types.
	
	Choose $v\in V(T)$ and $I\subsetneq E(v)$ non-empty. Recall that $E(v) = o^{-1}(v)$. Set $G_{v,I}$ to be the collection of $g\in G$ which satisfy one of the following:
	\begin{enumerate}[label = (\roman*)]
		\item $g$ fixes $v$ and $g(I) = I$; or
		\item $g$ is hyperbolic with $v\in\axis(g)$. Further still, $g$ translates in through $I$ and out through $E(v)\setminus I$, that is, there exist edges $e\in E(v)\setminus I$ and $e'\in I$ such that $t(e),t(e')\in \axis(g)$ and $t(e)>_g v >_g t(e')$.
	\end{enumerate}
	
	Now choose $\varepsilon\in\partial T$, define $G_{\pm,\varepsilon}$ to be the collection of $g\in G$ such that $g$ is elliptic and fixes $\varepsilon$, or $g$ is hyperbolic and $\omega_{\pm}(g) = \varepsilon$.
	
	\begin{theorem}[{\cite[4.11]{Baumgartner15}}]\label{thm:max_s_mult_Aut_T}
		The maximal scale-multiplicative semigroups of $G$ are of the following types:
		\begin{enumerate}[label = (\roman*)]
			\item the stabiliser of vertex in $G$;
			\item the stabiliser of the mid point of an edge in $G$, that is $G_{\{e\}}$ for some $e\in E(T)$;
			\item $G_{v,I}$ for $v\in V(T)$ and $I\subsetneq E(v)$ non-empty; or
			\item $G_{\pm,\varepsilon}$ for some $\varepsilon\in \partial T$.
		\end{enumerate}
		Conversely, any of the above sets are maximal scale-multiplicative semigroups.
	\end{theorem} 
	
	The proof of Theorem \ref{thm:max_s_mult_Aut_T} is based on the result that $s(g) = (\deg(T) - 1)^{l(g)}$ and so scale-multiplicative semigroups are precisely the semigroups where the length function is additive. 
\end{example}
One of the major obstructions to further developing the theory of maximal scale-multiplicative semigroups and scale-multiplicative semigroups in general is the lack of examples that have been computed. This is the motivation for Section \ref{sec:s_mult_leboudec}.
\subsection{The direction of an automorphism}
In \cite{Baumgartner06}, the authors construct a metric space at infinity for an arbitrary t.d.l.c. group. The construction is based on the observation that any t.d.l.c. group acts by isometries on it's set of compact open subgroups with appropriate choice of metric. From there, construction is analogous other other notions of boundaries for group actions. 

For this section fix a t.d.l.c. group $G$. Note that for $\alpha\in\Aut(G)$, the map $U\in\COS(G)\mapsto \alpha(U)\in \COS(G)$ defines an action of $\Aut(G)$ on $\COS(G)$.

\begin{lemma}[{\cite[Section 2]{Baumgartner06}}]\label{lem:cos_g_met}
	The set $\COS(G)$ is a metric space with metric
	\[d(U,V) := \log([U:U\cap V][V:U\cap V])\]		
	on which $\Aut(G)$ acts by isometries.
\end{lemma}
The action of $\Aut(G)$ on $\COS(G)$ gives a different interpretation of the scale function and tidy subgroups. 
\begin{lemma}[{\cite[Lemma 3]{Baumgartner06}}]\label{lem:basic_cos_action}
	Suppose $\alpha\in\Aut(G)$ and $U\in\COS(G)$. Then
	\begin{enumerate}[label = (\roman*)]
		\item \label{itm:lem_dir_obv1} The set $\{d(\alpha^n(U),U)\mid n\in\bN_0\}$ is bounded if and only if $s(\alpha) = 1$;
		\item\label{itm:lem_dir_obv2} $U\in \COS(G)$ is tidy for $\alpha$ if and only if $d(\alpha(U), U) = \min_{V\in\COS(G)}d(\alpha(V), V)$. In this case we have $d(\alpha(U),U) = \log(s(\alpha))+\log(s(\alpha^{-1}))$. 
	\end{enumerate}	
\end{lemma}

We use $\COS(G)$ to build a space at infinity for $G$. Like many boundary constructions associated to metric spaces, our construction involves quotienting infinite rays by an asymptotic relation before defining a suitable topology. We start with our notion of rays.

\begin{definition}
	Suppose $\alpha\in \Aut(G)$. Then for any $U\in \COS(G)$, the \emph{ray generated by $\alpha$ based at $U$} is the sequence $(\alpha^{n}(U))_{n\in\bN_0}$.
\end{definition}

Observe that if $(\alpha^n(U))_{n\in\bN_0}$ and $(\alpha^n(V))_{n\in\bN_0}$ are two rays generated by $\alpha$ based at $U,V\in\COS(G)$, then $d(\alpha^n(U),\alpha^n(V)) = d(U,V)$ trivially is bounded. It is desirable to have a definition that is independent of base point. We expand on this notion of bounded. 

\begin{definition}\label{def:asymptotic_rays}
	Two sequences $(V_n)_{n\in\bN_0}$ and $(U_{n})_{n\in\bN_0}$ of compact open subgroups of $G$ are \emph{asymptotic} if $d(V_n,U_n)$ is bounded.
\end{definition}
We use the notion of asymptotic sequences of compact open subgroups to define a notion of asymptotic on automorphisms. We use rays generated by the automorphisms at a given base point. Naively, we could require that these two rays be asymptotic in the sense of  Definition \ref{def:asymptotic_rays}; however, this would not account for the speed at which an automorphism moves towards infinity. To accommodate for this, a scaling factor of the automorphism is incorporated. This is given by $k_{\alpha}$ and $k_{\beta}$ in Definition \ref{def:asymp}.

\begin{definition}\label{def:asymp}
	Let $\alpha,\beta\in \Aut(G)$. We say $\alpha$ and $\beta$ are \emph{asymptotic} and write $\alpha\asymp \beta$ if there exists $k_\alpha,k_\beta\in\bN$ and $U_{\alpha},U_{\beta}\in\COS(G)$ such that the rays generated by $\alpha^{k_\alpha}$ and $\beta^{k_{\beta}}$ based at $U_{\alpha}$ and $U_{\beta}$ respectively are asymptotic.
\end{definition}

\begin{proposition}[{\cite[Lemma 10]{Baumgartner06}}]
	The relation $\asymp$ is an equivalence relation and is independent of base point subgroups.
\end{proposition}

It is not hard to see from Lemma \ref{lem:basic_cos_action} that all automorphisms with scale $1$ form an asymptotic class. We distinguish these automorphisms from others.

\begin{definition} We say an automorphism $\alpha$ of $G$ \emph{moves towards infinity} if $s(\alpha)> 1$. 
\end{definition}

\begin{lemma}\label{lem:asymp_towards_inf}
	Suppose $\alpha$, $\beta\in \Aut(G)$ are asymptotic with $\alpha$ moving towards infinity. Then $\beta$ moves towards infinity.
\end{lemma}

\begin{definition}
	For $A\subset \Aut(G)$, let $A_>$ be the subset of automorphisms moving towards infinity. For $\alpha\in A_>$, let $\partial_A(\alpha)$ to be the asymptotic class of $\alpha$ in $A$.  By identifying $G$ with the subgroup of inner automorphisms, we have a definition of $\partial_G$ which we abbreviate to $\partial$. Finally, we let $\partial G = \partial_G(G_>)$ be the \emph{set of directions} for $G$.
\end{definition}

We use asymptotic classes to associate a metric space to $G$. This space is the completion of a quotient of $\partial G$ by a pseudometric which we now define.

\begin{definition}
	Suppose $\alpha$ and $\beta$ are automorphisms moving towards infinity and choose $U,V\in\cos(G)$. We define 
	\[\delta_{+,n}^{U,V}(\alpha,\beta) = \min\left\{\left.\dfrac{\log[\alpha^n(U):\alpha^n(U)\cap \beta^k(V)]}{n\log s(\alpha)} \hspace{2pt}\right| k\in\bN_0, s(\beta^k)\le s(\alpha^n)\right\}.\]
	Set 
	\[\delta_{+}(\alpha,\beta) = \limsup_{n\to\infty}\delta_{+,n}^{U,V}(\alpha,\beta).\]
\end{definition}
\begin{proposition}[{\cite[Corollary 16 and Lemma 17]{Baumgartner06}}]\label{intro:prop:pmetric}
	The map $\delta_{+}$ is independent of choice of $U$ and $V$, and takes values between $0$ and $1$. We have $\delta_{+}(\alpha,\beta) = 0$ whenever $\alpha\asymp \beta$. The map $\delta(\alpha,\beta):= \delta_+(\alpha,\beta)+\delta_{+}(\beta,\alpha)$ defines a pseudometric on $\Aut(G)$ where the maximum distance between classes is $2$.
\end{proposition}

\begin{lemma}[{\cite[Lemma 15]{Baumgartner06}}]\label{lem:dir_inv_distance}
	Suppose $\alpha\in \Aut(G)$ such that both $\alpha$ and $\alpha^{-1}$ move towards infinity. Then $\delta(\alpha,\alpha^{-1}) = 2$.
\end{lemma}

\begin{definition}
	The \emph{space of directions} of $G$ is the completion of the metric space $\partial G/\delta^{-1}(0)$.
\end{definition}
The pseudometric $\delta$ is designed to be a notion of angle between two asymptotic classes. This is analogous to the Tit's metric on the boundary of a $\CAT(0)$ space. Intuition suggests that an $n$-flat, which we roughly interpret as free abelian subgroup, should give an $(n-1)$-sphere inside our boundary. This is the case for appropriate notion of $n$-flat, see the calculations in \cite[Sections 3.4 and 4.2]{Baumgartner06}. On the other hand, hyperbolic groups, which are far from having any $2$-dimensional flats, have very large angles between rays. In \cite{Bywaters19} it is shown that the space of directions of a hyperbolic group can by identified with a subset of the hyperbolic boundary and that any two distinct asymptotic classes are distance 2 apart. This is generalisation of the fact that the space of directions of the automorphism group of a regular tree is isometric to the boundary of the tree with the discrete metric, see \cite{Baumgartner06}.

\section{The scale function via a tidy subgroup}\label{sec:scale_on_G(F)}

In this section we give an explicit calculation of the scale of $g\in G(F,F')$ via the tidying procedure found in \cite{Willis01}. The implementation of this tidying procedure is the main content of Section \ref{ssec:con_tidy_sub}. To ease the calculation, we make a particular choice subgroup to input into the algorithm. This choice uses the notion of pando which we define in Definition \ref{def:pando}. Pandos give considerable control over the singularities of a hyperbolic element and as a result the last part of the algorithm is simpler than the general case, see Theorem \ref{thm:tidysubgroup} for the precise simplification and tidy subgroup. We further exploit the structure of our tidy subgroup in Section \ref{ssec:leboudec_scale_formula} where we give an explicit formula for the scale in Proposition \ref{prop:Scale_calc_2}. This formula is then used in Section \ref{ssec:leboudec_uniscalar} to investigate uniscalar elements of restricted Burger-Mozes groups. In particular, Corollary \ref{cor:classification_of_uniscalar} characterises precisely when elliptic and uniscalar coincide in terms of properties of $F$.

\begin{proposition}\label{prop:elliptic_scale}
Suppose $g\in G(F,F')$ is elliptic. Then $s(g) = 1$.
\end{proposition}
\begin{proof}
By definition, some power of $g$ fixes a vertex. Since $s(g^n) = s(g)^n$, see Proposition \ref{prop:scale_prop}, and the scale takes values in the natural numbers, we may assume without loss of generality that $g$ fixes some vertex, say $v$.

Since the set $S(g)$ of singularities of $g$ is finite, see Definition \ref{def:singularities}, there exists a ball
\[B_n(v):=\{u\in V(T)\mid d(u,v)\le n\},\]
such that $S(g)\subset B_n(v)$. Define
\[K_{n}(v):=\{g'\in G(F,F')_{v}\mid S(g')\subset B(v,n)\}.\]
That $K_{n}(v)$ is a compact open subgroup of $G(F,F')$ follows as $U_{v}\le K_{n}(v)$ has finite index and is compact and open. Since $g\in K_{n}(v)$ we see that \[s(g) = [gK_n(v)g^{-1}:gK_n(v)g^{-1}\cap K_n(v)] = 1.\qedhere\]
\end{proof}

The rest of this section is devoted to the study of the scale function on hyperbolic elements of $G(F,F')$. Lemma \ref{lem:SingOfPowers} concerns the singularities of $g^n$ for some $g\in G(F,F')$ hyperbolic. Although simple, these observations are used often.

\begin{lemma}\label{lem:SingOfPowers}
Suppose $g\in G(F,F')$, $k\ge 1$  and $v\in S(g^k)$. There exists $0\le l< k$ such that $g^{l}(v)\in S(g)$.
\end{lemma}
\begin{proof}
We have 
\[\sigma(g^{k},v) = \sigma(g,g^{k-1}(v))\cdots\sigma(g,g(v))\sigma(g,v)\not\in F.\]
Hence $\sigma(g,g^{l}(v))\not\in F$ for some $0\le l < k$.
\end{proof}
It is useful to specify a bound on the distance between a singularity of a hyperbolic element and its axis. This is the purpose on Definition \ref{def:sing_depth}.

\begin{definition}\label{def:sing_depth}
	Suppose $g\in G(F,F')$ is hyperbolic. Set
	\[D_g := \max\{d(v,\axis(g))\mid v\in S(g) \}.\]
\end{definition}
Lemma \ref{lem:depth_power} is a consequence of Lemma \ref{lem:SingOfPowers}.
\begin{lemma}\label{lem:depth_power}
	For $g\in G(F,F')$ hyperbolic and $n\in\bN$, we have $D_{g^n}\le D_g$.
\end{lemma}
\subsection{A convenient tidy subgroup}
\label{ssec:con_tidy_sub}
To simplify calculations, we make a specific choice of compact open subgroup, the definition of which relies on Definition \ref{def:pando}. We see that our choice of subgroup is automatically tidy above in Proposition \ref{prop:tidy_above}. The control given by Definition \ref{def:pando} is highlighted in Lemma \ref{lem:ObstructionToTidySingularities} which is a major part of the proof used to give the desired tidy subgroup, see Theorem \ref{thm:tidysubgroup}.

\begin{definition}\label{def:pando}
Suppose $g\in G(F,F')$ is hyperbolic. A pando $\cP$ for $g$ is a finite complete subtree of $T$ which satisfies the following:
\begin{enumerate}
	\renewcommand{\labelenumi}{\textbf{\theenumi}}
	\renewcommand{\theenumi}{P\arabic{enumi}}
	\item $S(g)\subset \Int(\cP)$\label{pando1};
	\item there exists a vertex $v\in \axis(g)\cap \cP$ such that $g(v)\in\Int(\cP)$\label{pando2};
	\item if $v\in V(T)$ such that $\pi_{g}(v)\in\Int(\cP)$ and $d(u,\pi_{g}(u)) = d(v,\pi_{g}(v))$ for some $u\in V(\cP)$, then $v\in V(\cP)$\label{pando3}. 
\end{enumerate}
For a pando $\cP$ of $g$ we define the initial segment $\cP_{0}$ to be the smallest complete subtree of $T$ which contains $\cP\setminus g(\cP)$. 
\end{definition}

\begin{remark}
Given a hyperbolic element $g\in G(F,F')$, there exists infinitely many pandos for $g$. Indeed, choose any $D> D_g$ and $v_0,v_1\in\axis(g)$ with $g(v_0)<_g v_1$ and $v_0<_g\pi_g v<_g v_1$ for all $v\in S(g)$. Such a $v_0$ and $v_1$ exist since $S(g)$ is finite. Set $\cP$ to be the complete tree which contains $v_0$, $v_1$ and all vertices $v$ within distance $D$ of $\axis(g)$ such that $v_0<_g\pi_{g}(v)<_g v_1$. This is easily verified to be a pando for $g$.
\end{remark}
\begin{lemma}\label{lem:axis_fix_to_pando_stab}
	Suppose $g\in G(F,F')$ is hyperbolic and $\cT$ is a complete subtree of $T$ satisfying \ref{pando2} and \ref{pando3}. Then $G(F,F')_{\axis(g)}\le G(F,F')_{\{\cT\}}$.
\end{lemma}
\begin{proof}
	Suppose $x\in G(F,F')_{\axis(g)}$ and $v\in V(\cT)$. Since $\cT$ is a complete subtree containing a path in $\axis(g)$, this follows from \ref{pando2}, either $\pi_g(v) = v$ or $\pi_g(v)\in\Int(\cT)$. In the first case we have $x(v) = v$. In the second, since $x\in G(F,F')_{\axis(g)}$ we have $\pi_g(v) = \pi_g(x(v))$ and $d(v,\axis(v)) = d(x(v),\axis(g))$. It follows from \ref{pando3} that $v\in\cT$.
\end{proof}
\begin{proposition}\label{prop:tidy_above}
Suppose $\cP$ is a pando for $g\in G(F,F')$. Then $U_{\cP}$ is tidy above for $g$.
\end{proposition}
\begin{proof}
Suppose $a\in U_{\cP}$. We factorise $a$ as a product $a_{+}a_{-}$ where $a_{\pm}\in (U_{\cP})_{\pm}$. By \ref{pando2} in Definition \ref{def:pando}, we may choose an edge $e\in E(\cP)$ such that $o(e)\in\axis(g)$ is minimised with respect to $\le_g$. Since $\cP$ is a complete subtree of $T$, $o(e)$ must have valency $1$ in $\cP$. This implies that if $v\in V(\cP)$ with $\pi_{g}(v) = o(e)$, then $v = o(e)$. Since $S(g)\subset \Int(\cP)$, for all $v\in S(g)$ we have $\pi_g(v)>_g o(e)$.

Define $a_+\in \Aut(T)$ by
\[a_+(v) = \twopartdef{a(v)}{\pi_{g}(v)\le_g o(e)}{v}{\hbox{otherwise.}}\]
It follows that $a_+$ is an automorphism of $T$ since $a(e) = e$. Furthermore, it follows that $a_+\in U(F)$ since $\sigma(a_+,v)\in\{\sigma(a,v),\id\}\subset F$ for all $v\in V(T)$. Fix $k\ge 0$. We show $g^{-k}a_+g^{k}\in U_{\cP}$ which implies $a_+\in (U_{\cP})_{+}$ as out choice of $k$ is arbitrary. To see that $g^{-k}a_{+}g^{k}$ fixes $\cP$, note that if $v\in V(\cP)$, then $\pi_{g}g^{k}(v) \ge_{g} \pi_{g}(v) \ge_g o(e)$. The definition of $a_+$ shows that $a_+g^{k}(v) = g^{k}(v)$ and so $g^{-k}a_{+}g^{k}(v) = v$. It suffices to show $g^{-k}a_{+}g^{k}\in U(F)$. Suppose $v\in V(T)$. Then
\begin{equation}\label{eq:conj_calc}\sigma(g^{-k}a_+g^{k},v) = \sigma(g^{-k},a_+g^{k}(v))\sigma(a_+,g^{k}v)\sigma(g^{k},v).
\end{equation}
If $\sigma(g^{-k}a_+g^{k},v)\not\in F$, then since $a_+\in U(F)$, we have $v\in S(g^{k})$ or ${a_{+}g^{k}(v)\in S(g^{-k})}$. Lemma \ref{lem:SingOfPowers} gives $0\le l < k$ such that $g^{l}(v)\in S(g)$, or $g^{-l}a_{+}g^{k}(v) \in S(g^{-1}) = gS(g)$ and so $g^{-l - 1}a_{+}g^{k}(v)\in S(g)$. We  must have either $\pi_{g}a_{+}g^{k}(v) >_g o(e)$ or $\pi_{g}g^{k}(v)>_g o(e)$ by choice of $e$. This shows that $\sigma(a_+,g^{k}(v)) = \id$ for both cases. Substituting into equation \eqref{eq:conj_calc} we see that $\sigma(g^{-k}a_+g^{k},v) = \id \in F$.

Now consider $a_{-}:= a_{+}^{-1}a$. Observe that
\[a_-(v) = \twopartdef{v}{\pi_{g}(v)\le_g o(e)}{a(v)}{\hbox{otherwise.}}\]
An argument similar to that of the previous paragraph can be used to show $g^{k}a_{-}g^{-k}\in U_{\cP}$ for all $k\ge 0$. Thus $a_{-}\in (U_{\cP})_-$.
\end{proof}
For a given pando $\cP$ for $g\in G(F,F')$, it is not guaranteed that $U_{\cP}$ is tidy below for $g$. To construct a subgroup which is tidy for $g$ we set
\begin{equation}\label{eq:tidy_obstruction}
\cL: = \{l\in G(F,F')\mid g^{k}lg^{-k}\in U_{\cP}\hbox{ for all but finitely many }k\in\bZ\},
\end{equation}
and $L = \overline{\cL}$. If
\[U_{\cP}':= \{u\in U_{\cP}\mid  lul^{-1}\in U_{\cP} \hbox{ for all }l\in L\},\]
then the product $U_{\cP}'L$ is tidy subgroup for $g$ by \cite[Section 3]{Willis01}.

Lemma \ref{lem:ObstructionToTidySingularities} highlights the control over $\cL$ given by using a pando to define the subgroup used as input for the tidying procedure.
\begin{lemma}\label{lem:ObstructionToTidySingularities}
Suppose $g\in G(F,F')$ is hyperbolic and $\cP$ is a pando for $g$. Suppose also that $a\in G(F,F')_{\axis(g)}$ such that $g^{-k}ag^{k}\in U(F)$ for all but finitely many $k\in\bZ$. Then \[S(a)\subset \Int(g(\cP))\cap \Int(\cP).\]
\end{lemma}
\begin{proof}
Suppose $v\in S(a)$. For $k\in\bN$ sufficiently large, we have 
\[\sigma(g^{-k}ag^{k},g^{-k}(v)) = \sigma(g^{-k},a(v))\sigma(a,v)\sigma(g^{k},g^{-k}(v))\in F.\]
Since $\sigma(a,v)\not\in F$, this implies either
\[\sigma(g^{-k},a(v))\not\in F\] or 
\[\sigma(g^{k},g^{-k}(v)) = \sigma(g^{-k},v)^{-1}\not\in F.\] 
That is, $v\in S(g^{-k})$ or $a(v)\in S(g^{-k})$. Lemma \ref{lem:SingOfPowers} gives $0< l \le k$ such that \[g^{-l}(v)\in S(g^{-1}) = g(S(g))\subset \Int(g(\cP))\] or 
\[g^{-l}a(v)\in S(g^{-1}) = g(S(g))\subset \Int(g(\cP)).\]
In the first case, we have $g^{-l - 1}(v)\in \Int(\cP)$. In the latter, $g^{-l - 1}a(v)\in \Int(\cP)$, but since $a\in G(F,F')_{\axis(g)}$, we have 
\[\pi_gg^{-l - 1}a(v) = \pi_g g^{-l-1}(v) \hspace{0.2cm}\hbox{ and }\hspace{0.2cm} d(g^{-l - 1}a(v), \axis(g)) = d(g^{-l - 1}(v),\axis(g)).\]
It follows from \ref{pando3} that we again have $g^{-l - 1}(v)\in \Int(\cP)$.

Repeating the same calculation but with $-k\in\bN$ sufficiently large, we see that there exists $m \ge 0$ such that $g^{m}(v)\in \Int(\cP)$. Since $\pi_g g^{-l-1}(v), \pi_g g^m(v)\in \Int(\cP)$ and $\pi_g g^{-l-1}(v)<_g \pi_g(v)\le_g \pi_g g^m(v)$, we have $\pi_g(v)\in \Int(\cP)$. But $d(v,\axis(g)) = d(g^{-l-1}(v),\axis(g))$, and so $v\in \Int(\cP)$ by \ref{pando3}. Similarly, since 
\[\pi_gg^{-l}(v)\le_g \pi_g g^{-1}(v) <_g\pi_g g^m(v)\]
we have $g^{-1}(v)\in \Int(\cP)$.
\end{proof}
\begin{lemma}\label{lem:L_fixes_axis}
	Suppose $g\in G(F,F')$ is hyperbolic with pando $\cP$. Define $\cL$ as in equation \eqref{eq:tidy_obstruction}. Then $\cL\le G(F,F')_{\axis(g)}$.
\end{lemma}
\begin{proof}
	Suppose $l\in\cL$ and $v\in\axis(g)$. Choose any vertex $v_0\in\cP\cap \axis(g)$. There exists $k\in\bN$ sufficiently large such that $g^{-k}(v_0)<_gv<_gg^{k}(v_0)$ and $g^{\pm k}lg^{\mp k}\in U_{\cP}$. We then have $g^{\pm k}lg^{\mp k}(v_0) = v_0$ and so $lg^{\mp k}(v_0) = g^{\mp k}(v_0)$. Since $v$ is on the unique path between $g^{-k}(v_0)$ and $g^{k}(v_0)$, we have $l(v) = v$.
\end{proof}
\begin{theorem}\label{thm:tidysubgroup}
Suppose $g\in G(F,F')$ is hyperbolic with pando $\cP$. Define $\cL$ as in equation \eqref{eq:tidy_obstruction}. Then $U_{\cP}' = U_{\cP}$ and $U_{\cP}\cL$ is tidy for $g$.
\end{theorem}
\begin{proof}
Suppose $a\in U_{\cP}$ and $l\in \cL$. It follows from Lemma  \ref{lem:L_fixes_axis} and Lemma \ref{lem:axis_fix_to_pando_stab} that $l\in G(F,F')_{\{\cP\}}$. Since $a$ fixes $\cP$, the product $lal^{-1}$ also fixes $\cP$. Applying Lemma \ref{lem:ObstructionToTidySingularities}, we see that $S(l^{-1}) \cup S(l)\subset\Int(\cP)$. Now $\sigma(a,v) = \id$ for all $v\in\Int(\cP)$, and so $\sigma(lal^{-1},u)\in F$ for all $u\in V(T)$. Thus, $lal^{-1}\in U_{\cP}$ and so $U_{\cP}\cL = \cL U_{\cP}$ is a subgroup.

Now suppose $(l_{n})_{n\in\bN}\subset \cL$ is a sequence converging to $l'\in L$. Continuity of multiplication shows that $l_nal_n^{-1} \to l'a(l')^{-1}$. But this is a sequence contained in $U_{\cP}$ which is closed. Hence, $l'a(l')^{-1}\in U_{\cP}$. This shows that $U_{\cP}' = U_{\cP}$.  

Now $U_{\cP}\cL$ is an open and therefore closed subgroup. We have
\begin{equation}\label{eq:closure_calc}
U_{\cP}\cL\le U_{\cP}L\le \overline{U_{\cP}\cL} = U_{\cP}\cL.
\end{equation}
This shows that $U_{\cP}\cL= U_{\cP}'L$ is tidy for $g$.
\end{proof}
\subsection{A formula for the scale}
\label{ssec:leboudec_scale_formula}
We now have a tidy subgroup for a given hyperbolic $g\in G(F,F')$. We use this subgroup to give a formula for the scale which does not require calculating the actual subgroup. This formula is given in Proposition \ref{prop:Scale_calc_2}. Our formula is in terms of a pando $\cP$ and a special subset of automorphisms of $\cP_0$, see Definition \ref{def:pando} for a definitions of $\cP$ and $\cP_0$, which we define in Definition \ref{def:scale_aut}. These automorphisms are related to our tidy subgroup in Lemma \ref{lem:quot_map}. We need two results, with slight adaptations, found in \cite{WillisNeretin}. The statements and proofs, see Lemma \ref{lem:wilNeretin} and Proposition \ref{prop:Scale_calc_1}, are included for convenience.

\begin{lemma}\label{lem:wilNeretin}
Suppose $g\in G(F,F')$ is hyperbolic with pando $\cP$. Define $\cL$ as in equation \eqref{eq:tidy_obstruction}. Then
\[U_{\cP}\cL\cap gU_{\cP}g^{-1}\cL = (U_{\cP}\cap gU_{\cP}g^{-1})\cL.\] 
\end{lemma}
\begin{proof}
Suppose $a\in U_{\cP}\cL\cap gU_{\cP}g^{-1}\cL$. Lemma \ref{prop:tidy_above} shows that $U_{\cP}$ is tidy above, thus there exist $u_{\pm},w_\pm\in (U_{\cP})_{\pm}$ and $l_1,l_2\in \cL$ such that
\[a = u_+u_-l_1 = gw_+g^{-1}gw_{-}g^{-1}l_2.\]
Rearranging, we have
\begin{equation}\label{eq:ComplexEquality}
(gw_+g^{-1})^{-1}u_+ = gw_{-}g^{-1}l_2l_1^{-1}u_{-}^{-1}.
\end{equation}

Since $(U_{\cP})_+\le g(U_{\cP})_{+}g^{-1}$, we have $(gw_+g^{-1})^{-1}u_+\in g(U_{\cP})_{+}g^{-1}$. In particular
\begin{equation}\label{eq:in_+}
{g^{-1}(gw_+g^{-1})^{-1}u_+g\in (U_{\cP})_{+}}.
\end{equation} 
Looking now at the right-hand side of \eqref{eq:ComplexEquality}, we see that since $g(U_{\cP})_{-}g^{-1}\le (U_{\cP})_{-}$, we have \[{gw_{-}g^{-1}l_2l_1^{-1}u_{-}^{-1}\in (U_{\cP})_{-}\cL (U_{\cP})_{-}}.\]

It follows that for all $n\in\bN$
\[g^{n}(gw_{-}g^{-1}l_2l_1^{-1}u_{-}^{-1})g^{-n}\in g^{n}(U_{\cP})_{-}g^{-n}\cL g^{n}(U_{\cP})_{-}g^{-n}\le U_{\cP}\cL.\]
The sequence must have an accumulation point as $U_{\cP}\cL$ is compact. Applying {\cite[Lemma 3.2]{Willis01}}, equations \eqref{eq:ComplexEquality} and \eqref{eq:in_+} and noting that $L$ is invariant under conjugation by $g$, we see that $(gw_+g^{-1})^{-1}u_+ \in L$. It follows that 
\begin{align*}
a & = gw_+g^{-1}gw_{-}g^{-1}l_2 = u_+u_+^{-1}gw_+g^{-1}gw_{-}g^{-1}l_2\\
& = u_+((gw_+g^{-1})^{-1}u_+)^{-1}gw_-g^{-1}l_2
\in U_{+}L(g(U_{\cP})_{-}g^{-1})\cL.
\end{align*}
But $L$ is invariant under conjugation by $g$ and normalises $(U_{\cP})_{-}$ by Theorem \ref{thm:tidysubgroup} and \cite[Lemma 3.7]{Willis01}. Thus, $a\in (U_{\cP})_{+}(g(U_{\cP})_{-}g^{-1})L$. But \[(U_{\cP})_{+}g(U_{\cP})_{-}g^{-1}\le (U_{\cP})_{+}(U_{\cP})_{-}\] and \[(U_{\cP})_{+}g(U_{\cP})_{-}g^{-1}\le g(U_{\cP})_{+}g^{-1}g(U_{\cP})_{-}g^{-1}\le gU_{\cP}g^{-1},\]
this shows that $a\in (U_{\cP}\cap gU_{\cP}g^{-1})L$. That $a\in (U_{\cP}\cap gU_{\cP}g^{-1})\cL$ follows since $U_{\cP}\cap gU_{\cP}g^{-1}$ an open subgroup, see the calculation \eqref{eq:closure_calc}.  We have shown that 
\[U_{\cP}\cL\cap gU_{\cP}g^{-1}\cL\subset (U_{\cP}\cap gU_{\cP}g^{-1})\cL.\]
The reverse inclusions is a consequence of the general fact $(A\cap B)C \subset AC\cap BC$ for any subsets $A,B,C$ contained in any group. 
\end{proof}
Proposition \ref{prop:Scale_calc_1} shows that the scale of $g\in G(F,F')$ for $g$ hyperbolic can be calculated by only considering subgroups in $U(F)$.
\begin{proposition}\label{prop:Scale_calc_1}
Suppose $g\in G(F,F')$ is hyperbolic with pando $\cP$. Define $\cL$ from equation \eqref{eq:tidy_obstruction}. Then
\[s(g) = \dfrac{[U_{g(\cP)}:U_{\cP\cup g(\cP)}]}{[\cL\cap U_{g(\cP)}:\cL\cap U_{\cP\cup g(\cP)}]}.\]
\end{proposition}
\begin{proof}
Since $g$ normalises $\cL$, Lemma \ref{lem:wilNeretin} implies
\[U_{\cP}\cL\cap gU_{\cP}\cL g^{-1} = U_{\cP}\cL\cap gU_{\cP}gg^{-1}\cL g = (U_{\cP}\cap gU_{\cP}g^{-1})\cL.\] 
Since $S(g)\subset \Int(\cP)$, we have $\sigma(a,v) = \id$ for all $v\in S(g)$ and $a\in U_{\cP}$. It follows that $gU_{\cP}g^{-1} = U_{g(\cP)}$ and so
\[s(g) = [U_{g(\cP)}\cL:U_{\cP\cup g(\cP)}\cL].\] 
Consider the map
\[\lambda:\dfrac{U_{g(\cP)}}{U_{\cP\cup g(\cP)}}\to \dfrac{U_{g(\cP)}\cL}{U_{\cP\cup g(\cP)}\cL}: xU_{\cP\cup g(\cP)}\mapsto xU_{\cP\cup g(\cP)}\cL.\]
That $\lambda$ is well defined and surjective follows from noting that $\cL$ normalises $U_{\cP\cup g(\cP)}$. We calculate the size of $\lambda^{-1}\lambda(xU_{\cP\cup g(\cP)})$ for $x\in U_{g(\cP)}$. Since $U_{g(\cP)}\cL$ and $U_{\cP\cup g(\cP)}\cL$ are both subgroups, $\lambda(xU_{\cP\cup g(\cP)}) = \lambda(yU_{\cP\cup g(\cP)})$ if and only if $x^{-1}y\in U_{\cP \cup g(\cP)}\cL$. Equivalently, there exists $l\in\cL$ and $u\in U_{\cP\cup g(\cP)}$ such that $x^{-1}y = lu$. Rearranging, we see that $yu^{-1}\in x(\cL\cap U_{g(\cP)})$ and $yu^{-1}U_{\cP\cup g(\cP)} = yU_{\cP\cup g(\cP)}$. Alternatively, if $y\in x(\cL\cap U_{g(\cP)})$, then we have $\lambda(yU_{\cP\cup g(\cP)}) = \lambda(xU_{\cP\cup g(\cP)})$. This shows that
\[|\lambda^{-1}\lambda(x)| = |\{xyU_{\cP\cup g(\cP)}\mid y \in \cL\cap U_{g(\cP)}\}| = [\cL\cap U_{g(\cP)}:\cL\cap U_{\cP\cup g(\cP)}].\]
This is independent of $x$ and so
\[s(g) = \dfrac{[U_{g(\cP)}:U_{\cP\cup g(\cP)}]}{[\cL\cap U_{g(\cP)}:\cL\cap U_{\cP\cup g(\cP)}]}\]
as required.
\end{proof}
We now characterise $[\cL\cap U_{g(\cP)}:\cL\cap U_{\cP\cup g(\cP)}]$ in terms of the size of a set of automorphisms of $\cP_0$. To define this set, we define $\sigma(a,v)$ for $a\in \Aut(\cP_0)$ and $v\in \Int(\cP_0)$ by considering the colouring induced on $E(\cP_0)$ by $c$.
\begin{definition}\label{def:scale_aut}
Suppose $g\in G(F,F')$ is hyperbolic with pando $\cP$. Let $M_{g,\cP_0}$ denote the collections of automorphisms $a\in\Aut(\cP_0)$ such that: 
\begin{enumerate}[label = (\roman{*})]
	\item $a$ fixes $\axis(g)\cap \cP_{0}$; and
	\item for sufficiently large $k\in\bN$, we have $\sigma(a,v)\in \sigma(g^{k},a(v))^{-1}F\sigma(g^{k},v)\cap F$ for all $v\in\Int(\cP_{0})$.
\end{enumerate} 
\end{definition}
\begin{lemma}\label{lem:quot_map}
Suppose $g\in G(F,F')$ is hyperbolic with pando $\cP$ and $a\in M_{g,\cP_0}$. Then there exists $b\in \cL\cap U_{g(\cP)}$ such that $a(v) = b(v)$ for all $v\in V(\cP_{0})$.

Conversely, for $b\in \cL$, define $a\in \Aut(\cP_{0})$ by setting $a(v) = b(v)$ for all $v\in V(\cP_0)$. Then $a\in M_{g,\cP_0}$.
\end{lemma}
\begin{proof}
Suppose $a\in M_{g,\cP_0}$. By Lemma \ref{lem:autExtension}, there exists $\bar{a}\in U(F)$ such that $\bar{a}(v) = a(v)$ for all $v\in V(\cP_{0})$. We define $b\in U(F)$ by
\[b(v) = \twopartdef{\bar{a}(v)}{\pi_{g}(v)\in\Int(\cP_{0})}{v}{\hbox{ otherwise.}}\]
It is immediate that $b(v) = a(v)$ for all $v\in V(\cP_0)$ with $\pi_{g}(v)\in\Int(\cP_{0})$. If $v\in \cP_{0}$ with $\pi_{g}(v)\not\in\Int(\cP_{0})$, then $v\in\axis(g)$ as $\cP_{0}$ is a complete subtree. We have $a(v)= v = b(v)$. Overall, $a(v) = b(v)$ for all $v\in V(\cP_{0})$. If $v\in V(g(\cP))$, then the definition of $\cP_{0}$ implies  $\pi_{g}(v)\not\in\Int(\cP_{0})$. It follows that $b(v) = v$. This shows that $b\in U_{g(\cP)}$. For the first claim, we are left to show $b\in\cL$. Suppose $v\in V(T)$ and $k\in\bZ$. We show that for $|k|$ sufficiently large, we have $\sigma(g^{k}bg^{-k},g^{k}(v))\in F$. We split our argument into cases based on $v$.

\noindent{\bfseries{Case 1: }}Suppose $\pi_{g}(v)\not\in\Int(\cP_0)$. Then $\sigma(b, v) = \id$ by definition. Thus,

\[\sigma(g^{k}bg^{-k},g^{k}(v)) = \sigma(g^{k},v)\sigma(g^{-k},g^{k}(v)) = \id.\]

\noindent{\bfseries{Case 2: }}Suppose  $\pi_{g}(v)\in\Int(\cP_0)$ but $v\not\in\Int(\cP_{0})$. From \ref{pando3} we see that $d(v,\axis(g)) > D_g$. Thus, $d(g^{n}(v),\axis(g)) = d(v,\axis(g))> D_g$ for all $n\in\bZ$. In particular, $\sigma(g^{n},v)\in F$ for all $n\in\bZ$. Since $\sigma(b,v)\in F$ and $d(v,\axis(g)) = d(b(v),\axis(g))$, we have $\sigma(g^{k}bg^{-k},g^{k}(v))\in F$ by Lemma \ref{lem:sig_basic}.

\noindent{\bfseries{Case 3: }}Finally, suppose $v\in \Int(\cP_0)$. For all $k < 0$ we have $\sigma(g^{-k}, g^{k}(v))\in F$ since $S(g)\subset \cP$ and $g^{l}(v)\not\in \cP$ for all $l < 0$. The same reasoning shows $\sigma(g^{k}, b(v))\in F$. Since $\sigma(b,v) \in F$ by construction, $\sigma(g^{k}bg^{-k},g^{k}(v)) \in F$. Alternatively, if $k > 0$, we can without loss of generality assume that $k$ is chosen sufficiently large so that 
\[\sigma(a,v)\in \sigma(g^{k},a(v))^{-1}F\sigma(g^{k},v).\]
Since $\sigma(b,v) = \sigma(a,v)$ and
\[\sigma(g^{k}bg^{-k},g^{k}(v)) = \sigma(g^{k}b(v))\sigma(b,v)\sigma(g^{k},v)^{-1},\]
it follows that $\sigma(g^{k}bg^{-k},g^{k}(v))\in F$. 

We have covered all possible cases for $v$ and conclude $b\in\cL\cap U_{g(\cP)}$.

Conversely, suppose $b\in\cL$. Lemma \ref{lem:L_fixes_axis} and Lemma \ref{lem:axis_fix_to_pando_stab} show $b\in G(F,F')_{\{\cP_0\}}$. Thus, defining $a\in \Aut(\cP_0)$ as the restriction of $b$ to $\cP_0$ is well defined. Lemma \ref{lem:L_fixes_axis} shows that $a$ fixes $\axis(g)\cap \cP_{0}$. Lemma \ref{lem:ObstructionToTidySingularities} shows that $\sigma(a,v) = \sigma(b,v)\in F$ for all  $v\in\Int(\cP_0)$. That $\sigma(a,v)\in \sigma(g^{k},a(v))^{-1}F\sigma(g^{k},v)$ for sufficiently large $k\in\bN$ follows from the definition of $\cL$.
\end{proof}
\begin{proposition}\label{prop:Scale_calc_2}
Suppose $g\in G(F,F')$ is hyperbolic with pando $\cP$. Then
\[s(g) = \dfrac{[U_{g(\cP)}:U_{\cP\cup g(\cP)}]}{|M_{g,\cP_0}|}\]
\end{proposition}
\begin{proof}
Since $\cL\le G(F,F')_{\axis(g)}$ by Lemma \ref{lem:L_fixes_axis}, quick calculation can be used to show $\cL\cap U_{\cP\cup g(\cP)}\normal  \cL\cap U_{g(\cP)}$. Applying Proposition \ref{prop:Scale_calc_1}, it suffices to show 
\[(\cL\cap U_{g(\cP)})/(\cL\cap U_{\cP\cup g(\cP)})\cong M_{g,\cP_0}. \]
We do so by applying the First Isomorphism Theorem. Indeed, by Lemma \ref{lem:quot_map}, for each $b\in \cL\cap U_{g(\cP)}$, there exists unique $a\in M_{g,\cP_0}$ for which $b(v) = a(v)$ for all $v\in V(\cP_0)$. The map $b\mapsto a$ gives a surjective homomorphism from $(\cL\cap U_{g(\cP_0)})$ onto $M_{g,\cP_0}$ by Lemma \ref{lem:quot_map}. The kernel is precisely $\cL\cap U_{\cP_0\cup g(\cP)}$ which is equal to $\cL\cap U_{\cP\cup g(\cP)}$ as $\cP\cup g(\cP) = \cP_0\cup g(\cP)$. The First Isomorphism Theorem gives the required isomorphism.  
\end{proof}
For Corollary \ref{cor:same_scale}, we use the notation $s_G$ to denote the scale function on a group $G$. In general, for $g\in G\le G'$, it may not be that $s_G(g) = s_{G'}(g)$ but it is true that $s_G(g)\le s_{G'}(g)$, see \cite[Proposition 4.3 and Example 6.1]{Willis01}.
\begin{corollary}\label{cor:same_scale}
	Suppose $F\le F'$ and $F\le F''$ and $g\in G(F,F')\cap G(F,F'')$. Then $s_{G(F,F')}(g) = s_{G(F,F'')}(g)$.
\end{corollary}
\begin{proof}
	If $g$ is elliptic, then Proposition \ref{prop:elliptic_scale} shows $s_{G(F,F')}(g) = s_{G(F,F'')}(g) = 1$. Alternatively, if $g$ is hyperbolic, note that the formula for the scale given in Proposition \ref{prop:Scale_calc_2} depends only on $F$, $g$ and $\cP$ and not on $F'$ or $F''$.	
\end{proof}
\subsection{Uniscalar elements}\label{ssec:leboudec_uniscalar}
We now consider uniscalar elements of $G(F,F')$ and investigate the possibility of a converse to Proposition \ref{prop:elliptic_scale}. We show in Corollary \ref{cor:classification_of_uniscalar} that all uniscalar elements of $G(F,F')$ are elliptic if and only if $F$ satisfies the property given as Definition \ref{def:distinct_pt_stab}.

Throughout this section we refer to the scale of an element in $U(F)\le G(F,F')$. There is no confusion about which scale function we are discussing as Corollary \ref{cor:same_scale} shows the scale on $G(F,F')$ restricted to $U(F)$ agrees with the scale on $U(F) = G(F,F)$. 
\begin{definition}\label{def:distinct_pt_stab}
We say $F\le \Sym(\Omega)$ has \emph{distinct point stabilisers} if for all $a,b\in \Omega$, we have $F_a = F_b$ if and only if $a = b$.
\end{definition}
We investigate the scale of hyperbolic elements in $G(F,F')$ by comparing the scale of an element in $G(F,F')$ with the scale of an element in $U(F)$. Proposition \ref{prop:U(F)_scale_calc} is found in \cite{Reid13} and calculates the scale for hyperbolic $g\in U(F)$. It can be proved using Proposition \ref{prop:Scale_calc_2}, Lemma \ref{eq:coset_calc}, Corollary \ref{cor:same_scale} and choosing a pando $\cP$ such that $\Int(\cP)\subset\axis(g)$.
\begin{proposition}[{\cite[Proposition 4.3]{Reid13}}]\label{prop:U(F)_scale_calc}
Suppose $g\in U(F)$ is hyperbolic and \[(v_0,\ldots,v_{k} = g(v_0), v_{k+1} = g(v_1))\subset\axis(g)\]
is a path in $\axis(g)$. Write $c_i = c(v_{i},v_{i+1})\in \Omega$ for $i\in \{0,\ldots, k\}$. Then 
\[s(g) = \prod_{i = 0}^{k-1} \left|\dfrac{F_{c_{i}}}{F_{c_i}\cap F_{c_{i+1}}}\right|.\]
\end{proposition}

\begin{lemma}\label{lem:U(F)_hyp_scale}
Suppose $F$ has distinct point stabilisers and $g\in U(F)$ is hyperbolic. Then $s(g) > 1$.
\end{lemma}
\begin{proof}
We prove the contrapositive. Suppose $g\in U(F)$ is hyperbolic with $s(g) = 1$. Choose a path $(v_0,\ldots, v_k = g(v_0),v_{k+1} = g(v_1))\subset\axis(g)$ and let $c_i = c(v_{i},v_{i+1})\in \Omega$ for $i\in \{0,\ldots, k\}$. Proposition \ref{prop:U(F)_scale_calc} shows 
\[1 = s(g) = \prod_{i = 0}^{k-1} \left|\dfrac{F_{c_{i}}}{F_{c_i}\cap F_{c_{i+1}}}\right|.\]
This implies $F_{c_i}\cap F_{c_{i+1}} = F_{c_{i}}$ 
for $0\le i \le k-1$. Thus, $ F_{c_i}\le F_{c_{i+1}}$ for $1\le i \le k$, which is a nested chain of inclusions
\begin{equation}
\label{eq:nesting}
F_{c_0}\le F_{c_1}\le \cdots\le F_{c_{k}}.
\end{equation} However, $\sigma(g,v_0)(c_0) = c_{k}$ since $g(v_0) = v_{k}$ and $g(v_1) = v_{k+1}$. Thus, 
\[|F_{c_0}| = |\sigma(g,v_0)^{-1}F_{c_{k}}\sigma(g,v_0)| = |F_{c_{k}}|.\] 
But $F$ is finite and so \eqref{eq:nesting} shows $F_{c_i}= F_{c_{i+1}}$ for $i\in\{0,\ldots,k-1\}$. This completes the proof since $c_0$ and $c_1$ are the colours of distinct edges in $E(v_1)$ and are therefore distinct.
\end{proof}
To extend the conclusion of Lemma \ref{lem:U(F)_hyp_scale} to hyperbolic $g\in G(F,F')$, we use the following Lemma.
\begin{lemma}\label{lem:coset_from_found}
Suppose $g\in G(F,F')$ is hyperbolic and $\cP$ a pando for $g$. Then
\[[U_{g(\cP)}:U_{g(\cP)\cup\cP}] = [U_{g(\cP_0)}:U_{g(\cP_0)\cup\cP_0}].\]
\end{lemma}

\begin{proof}
Define $\varphi:U_{g(\cP)}/U_{g(\cP)\cup\cP}	\to U_{g(\cP_0)}/U_{g(\cP_0)\cup\cP_0}$ by setting \[{\varphi(aU_{g(\cP)\cup\cP}) = aU_{g(\cP_0)\cup\cP_0}}.\]
That $\varphi$ well defined follows since $ U_{g(\cP)\cup\cP}\le U_{g(\cP_0)\cup\cP_0}$. To see $\varphi$ is injective, first observe that $\cP_{0}\cup g(\cP) = g(\cP)\cup \cP$. It follows that if $a,b\in U_{g(\cP)}$ and $a^{-1}b\in U_{\cP_0\cup g(\cP_0)}$, then $a^{-1}b U_{\cP_{0}\cup g(\cP)} = U_{\cP\cup g(\cP)}$. 

To see $\varphi$ is surjective, suppose $a\in U_{g(\cP_0)}$. Pick $v\in \Int(g(\cP_0))\cap \axis(g)$ and define $b\in U_{g(\cP_0)\cup\cP_0}$ by
\[b(u) = \twopartdef{u}{\pi_{g}(u)<_g v}{a(u)}{\pi_g(u)\ge_g v}.\]
Calculation shows $ab^{-1}\in U_{g(\cP)}$ and $\varphi(ab^{-1}U_{g(\cP)\cup\cP}) = aU_{g(\cP_0)\cup\cP_0}$.
\end{proof}

\begin{proposition}\label{prop:U(F)_small_scale}
	Suppose $F$ has distinct point stabilisers and $g\in G(F,F')$ is hyperbolic. Then there exists $h\in U(F)$ with $l(h) = l(g)$ such that $s(h)\le s(g)$.
\end{proposition}
\begin{proof}
	Since $S(g)$ is finite, there exists a pando $\cP$ for $g$ such that $V(\cP_0)\cap S(g) = \varnothing$. Applying Lemma \ref{lem:autExtension}, there exists $h\in U(F)$ such that $h(v) = g(v)$ for all $v\in V(\cP_0)$. Then $h(\cP_0) = g(\cP_0)$. It follows from \cite[Lemma 4.1]{Baumgartner15} that $h$ is hyperbolic and $\cP_0$ is a pando for $h$. Applying Lemma \ref{lem:coset_from_found} gives
	\[[U_{g(\cP)}:U_{g(\cP)}\cap U_{\cP}] = [U_{g(\cP_0)}: U_{g(\cP_0)}\cap U_{\cP_0}] = [U_{h(\cP_0)} : U_{h(\cP_0)}\cap U_{\cP_0}].\]
	Definition \ref{def:scale_aut} and the fact that $S(h) = \varnothing$ gives $M_{g,\cP_0}\le M_{h,\cP_0}$. Proposition \ref{prop:Scale_calc_2} shows
	\[s(h) = \dfrac{[U_{h(\cP_0)}:U_{h(\cP_0)}\cap U_{\cP_0}]}{|M_{h,\cP_0}|}\le \dfrac{[U_{g(\cP)}:U_{g(\cP)}\cap U_{\cP}]}{|M_{g,\cP_0}|}= s(g).\qedhere\]
\end{proof}
\begin{corollary}\label{cor:classification_of_uniscalar}
Suppose $F\le F'$. Then $F$ has distinct point stabilisers if and only if the set of uniscalar elements in $G(F,F')$ is equal to the set of elliptic elements in $G(F,F')$.
\end{corollary}
\begin{proof}
Suppose $F$ has distinct point stabilisers and $g\in G(F,F')$. If $g$ is elliptic, then $g$ is uniscalar by Proposition \ref{prop:elliptic_scale}. Suppose instead that $g$ is hyperbolic. Then Proposition \ref{prop:U(F)_small_scale} gives $h\in U(F)$ with $s(h) \le s(g)$. Lemma \ref{lem:U(F)_hyp_scale} gives $1<s(h)$.
Thus, $g$ is not uniscalar.

Suppose $F$ does not have distinct point stabilisers. Choose distinct $a,b\in\Omega$ such that $F_a = F_b$. There exists an infinite path $P := (\ldots,v_{-1},v_0,v_1,\ldots)\subset T$ such that $c(v_{2k},v_{2k+1}) = a$ and $c(v_{2k+1},v_{2k}) = b$ for all $k\in\bZ$. Using Lemma \ref{lem:autExtension}, define $g\in U(F)$ such that $g(v_k) = v_{k+2}$ and $\sigma(g,v_k) = \id$. Then $g$ is hyperbolic with $\axis(g) = P$. Proposition \ref{prop:U(F)_scale_calc} shows $s(g) = 1$.
\end{proof}

\section{Asymptotic classes and directions}
\label{sec:leboudecDirections}
In this section we study the space of directions of $G(F,F')$ when $F\le \Sym(\Omega)$ is assumed to be $2$-transitive. The outcomes of this sections can be summarised into three main results. Our first major result, Theorem \ref{thm:asymptotic_and_length_classification}, relates the asymptotic relation on $G(F,F')_>$ to a length function defined in \cite{Boudec15}. Indirectly, this compares the asymptotic classes to the action of $G(F,F')$ on a $\CAT(0)$ cube complex. The majority of Section \ref{sec:comm_and_asymp} is devoted to proving this result. Proposition \ref{prop:classification_in_terms_of_sing_traj} is another characterisation of the asymptotic relation but this time in terms of a function which captures the information on singularities of high powers, see Definition \ref{def:asymp_func}. The basic properties of this function are given in Lemma \ref{lem:basic_asym_sing_prop} and control over this function on products is given by the technical Lemma \ref{lem:special_set}. The final major result of this section is Theorem \ref{thm:directions} which shows that the topology on the space of directions of $G(F,F')$ is discrete by giving a lower bound on the distance between distinct asymptotic classes. Corollary \ref{cor:dist_bound_asym} gives the precise lower bound.

Lemma \ref{lem:coset_calc} can be shown using repeated applications of the Orbit-Stabiliser Theorem.
\begin{lemma}\label{lem:coset_calc}
Suppose $(v_0,v_1,\ldots,v_k,v_{k+1})$ is a path in $T$. Then, setting $e_i = (v_i,v_{i+1})$ for $i\in\{0,\ldots,k\}$, we have  
\begin{equation*}\label{eq:coset_calc}
[U_{e_0}:U_{e_0}\cap U_{e_k}] = \prod_{i = 0}^{k-1}\left|\dfrac{F_{c(e_i)}}{F_{c({e_i})}\cap F_{c(e_{i+1})}}\right|.
\end{equation*}
\end{lemma}
For $g,h\in \Aut(T)$ hyperbolic, it is shown that $g\asymp h$ (in $\Aut(T)$) if and only if $\omega_{+}(g) = \omega_{+}(h)$, see \cite[Section 5.1]{Baumgartner06}. This does not hold in $G(F,F')$ for arbitrary $F\le F'\le \Sym(\Omega)$. Lemma \ref{lem:equal_ends} shows one of the implications does hold.
\begin{lemma}\label{lem:equal_ends}
Suppose $F$ has distinct point stabilisers and $g,h\in G(F,F')$  are hyperbolic such that
\[\{[U_{e}: U_{e}\cap U_{e}^{g^{-n}h^{n}}]\mid n\in\bN\}\]
is bounded for some $e\in E(T)$. Then $\{d(e, g^{-n}h^{n}(e))\mid n\in\bN\}$ is bounded. In particular $\omega_{+}(g) = \omega_{+}(h)$. 
\end{lemma}
\begin{proof}
Since $U_{e}\cap U_{e}^{g^{-n}h^{n}}\le U_{e}\cap U_{g^{-n}h^{n}(e)}$, we have
\[[U_{e}:U_{e}\cap U_{e}^{g^{-n}h^{n}}]\ge [U_{e}:U_{e}\cap U_{g^{-n}h^{n}(e)}].\]
Noting that $F$ has distinct point stabilisers and applying Lemma \ref{lem:coset_calc}, we see that
\[[U_{e}:U_{e}\cap U_{g^{-n}h^{n}(e)}] \ge d(e,g^{-n}h^{n}(e)).\]
This proves our first claim. For the second choose $v\in \axis(g)$ and $u\in\axis(h)$. Observe that
\begin{align*}
d(h^{n}(u),g^{n}(v)) &= d(v,g^{-n}h^{n}(v))\\
&\le d(v,o(e))  + d(o(e),g^{-n}h^{n}(o(e)))\\
& \hspace{20.52mm} + d(g^{-n}h^{n}(o(e)),g^{-n}h^{n}(u))\\
&\le d(v,e)+ d(e,g^{-n}h^{n}(e))+d(g^{-n}h^{n}(e),g^{-n}h^{n}(u)) + 4\\
& = d(v,e) + d(e, g^{-n}h^{n}(e)) + d(e,u) + 4
\end{align*}
Since $T$ is a tree, any pair of infinite paths containing $(g^{n}(v))_{n\in\bN}$ and $(h^{n}(u))_{n\in\bN}$ must eventually agree. This shows $\omega_{+}(g) = \omega_{+}(h)$ completing the proof.
\end{proof}
\subsection{Asymptotic classes and a length function}\label{sec:comm_and_asymp}

In this section we relate asymptotic classes of $G(F,F')$ with the action on a $\CAT(0)$ cube complex. The action we consider is given in \cite[Section 6]{Boudec15}. We describe the construction for convenience and to establish notation.

Suppose $e$ is an edge in $T$ and $g\in G(F,F')$. We denote by $\cT_{e}(g)$ the unique minimal complete subtree of $T$ such that $e$ and $g^{-1}(e)$ are edges in $\cT_{e}(g)$ and $v\in S(g)$ implies $v\in\Int(\cT_{e}(g))$.
Note that $\cT_{e}(g^{-1}) = g\cT_{e}(g)$. Let $\cN_{e}(g) = |\Int(\cT_{e}(g))|$. It is shown in \cite[Section 6]{Boudec15} that $\cN_e$ is a length function, that is, for all $g,h\in G(F,F')$
\begin{enumerate}[label = (\roman{*})]
	\item $\cN_e(\id) = 0$; 
	\item $\cN_{e}(g) = \cN_{e} (g^{-1})$; and
	\item $0\le \cN_e(gh)\le \cN_e(g)+ \cN_e(h)$.
\end{enumerate} 
Furthermore, if $F$ is transitive and $e'$ is another edge then there exists $K_1,K_2\in\bN$ such that 
\begin{equation}\label{eq:comparable_card_def}
K_1^{-1}\cN_{e'}(g) - K_2\le \cN_e(g)\le K_1\cN_{e'}(g)+K_2
\end{equation}
for all $g\in G(F,F')$. It is shown in  \cite[Proposition 6.11]{Boudec15} that $\cN_{e}$ is a cardinal definite function, that is, there exists a set $S$ on which $G(F,F')$ acts and a subset $S'\subset S$ such that $|g(S')\Delta S'| = 2\cN_e(g)$, here $\Delta$ is the symmetric difference of sets. A general argument, see \cite{corn13} and references therein, gives an action of $G(F,F')$ on the $1$-skeleton of a $\CAT(0)$ cube complex with distinguished vertex $m_0$ such that $d(m_0,g(m_0)) = \cN_{e}(g)$. We show that if $F$ is $2$-transitive, then  $g\asymp h$ if and only if there exists $k_1,k_2\in \bN$ such that $\cN_{e}(g^{-k_1n}h^{k_2n})$, equivalently $ d(g^{k_1n}(m_0),h^{k_2n}(m_0))$, is bounded for $n\in\bN$.

\begin{theorem}\label{thm:asymptotic_and_length_classification}
	Suppose $F$ is $2$-transitive, $e\in E(T)$ and $g,h\in G(F,F')$ are hyperbolic. Then $g\asymp h$ if and only if there exists $p,q\in \bN$ such that $\{\cN_{e}(g^{-pn}h^{qn})\mid n\in\bN\}$ is bounded.
\end{theorem}
The proof of Theorem \ref{thm:asymptotic_and_length_classification} can be split into two parts. Each part finds either an upper or lower bounds, roughly in terms of $\cN_{e}$, for the indices that feature in the asymptotic relation. The majority of the first part is Lemma \ref{lem:U_stab_rel} and Lemma \ref{lem:length_bound_coset_above} which gives an upper bound in terms of $\cN_e$. 

The second part is more complicated but can be summarised as follows: Lemma \ref{lem:U_stab_rel} that if $g\asymp h$, then $\cT_{e}(h^{-n}g^n)$ is contained within a finite distance of $\axis(g)$ and that this distance is independent of $n$. Thus $|\cT_{e}(h^{-n}g^n)|$ is unbounded as $n\to\infty$, then the number distinct distances of between leaves in $\cT_{e}(h^{-n}g^n)$ and $e$ is also unbounded as $n\to\infty$. This is the case since paths between leaves and $e$ must stay within bounded distance of $\axis(g)$. Lemma \ref{lem:paths_bound_cosets} gives a lower bound on the relevant indices in terms of the number of these distinct lengths.

\begin{lemma}\label{lem:U_stab_rel}
Suppose $g\in G(F,F')$ and $e\in E(T)$. Then \[{U_{\cT_{e}(g^{-1})}\le U_{\{e\}}^{g}\cap U_{\{e\}}\le U_{\{\cT_{e}(g^{-1})\}}}.\]
\end{lemma}
\begin{proof}
Suppose $x\in U_{\cT_{e}(g^{-1})}$. Then $x\in U_{\{e\}}\cap U_{\{g(e)\}}$ since both $e$ and $g(e)$ are edges in $\cT_{e}(g^{-1})$. It suffices to show $g^{-1}xg\in U(F)$. If $v\in S(g^{-1})$, then $v\in\Int(\cT_{e}(g^{-1}))$ and so $\sigma(x,v) = \id$. This and Lemma \ref{lem:sig_basic} shows $g^{-1}xg\in U(F)$.

Now suppose $x\not\in U_{\{\cT_{e}(g^{-1})\}}$. From the definition of $\cT_{e}(g^{-1})$, it follows that either $x(e)\not\in E(\cT_{e}(g^{-1}))$, $x(g(e))\not\in E(\cT_{e}(g^{-1}))$ or that there exists $v\in S(g^{-1})$ such that $x(v)\not\in V(\cT_{e}(g^{-1}))$. The first two cases imply $x\not\in U_{\{g(e)\}}\cap U_{\{e\}}$ which contains $ U_{\{e\}}^{g}\cap U_{\{e\}}$. We may suppose $v\in S(g^{-1})$ with $x(v)\not\in \cT_{e}(g^{-1})$. This implies $x(v)\not\in S(g^{-1})$ and so $\sigma(g^{-1},x(v))\in F$. However, $\sigma(g,g^{-1}(v))\not\in F$ since $g^{-1}(v)\in S(g)$. Lemma \ref{lem:sig_basic} shows $\sigma(g^{-1}xg,g^{-1}(v))\
\not\in F$. Thus $x\not\in U_{\{e\}}^{g}\cap U_{\{e\}}$.
\end{proof}
\begin{lemma}\label{lem:length_bound_coset_above}
Suppose $g\in G(F,F')$ and $e\in E(T)$, then \[{[U_{\{e\}}:U_{\{e\}}\cap U_{\{e\}}^{g}]\le (2(\deg(T)-1)^{\cN_{e}(g)})!}.\]
\end{lemma}
\begin{proof}
Choose $k\in\bN$ minimal such that if $v\in V(\cT_{e}(g^{-1}))$, then $d(v,e)\le k$. Note that $\cN_{e}(g)\ge k$ as $\cT_{e}(g^{-1})$ is connected and so must contain a path of length $k+1$. Set $B := \{v\in V(T)\mid  d(v,e)\le k\}$. Then $U_{\{e\}}$ acts by permutations on $B$. Also if $u_1,u_2\in U_{\{e\}}$ such that $u_{1}^{-1}u_{2}$ fixes $B$, then $u_{1}^{-1}u_{2}\in U_{\cT_{e}(g^{-1})}$. Lemma \ref{lem:U_stab_rel} shows $u_1$ and $u_2$ are in the same coset of $gU_{\{e\}}g^{-1}\cap U_{\{e\}}$. Applying Lemma \ref{lem:size_of_balls} shows
\[[U_{\{e\}}:U_{\{e\}}^{g}\cap U_{\{e\}}]\le |\Sym(B)| = (2(\deg(T) - 1)^{k})! \le (2(\deg(T) - 1)^{\cN_{e}(g)})!, \]
as required.
\end{proof}
\begin{lemma}\label{lem:cylinder_of_singularities}
Suppose $g,h\in G(F,F')$ are hyperbolic such that $\omega_{+}(g) = \omega_{+}(h)$. There exists $M\in\bN$ such that
\[v\in \bigcup_{m,n\in\bN}S(h^{-m}g^{n})\hbox{ implies }d(v,\axis(g))\le M.\]

Suppose further that $F$ has distinct point stabilisers, $e\in E(T)$ and 
\[\left\{[U_{\{e\}}:U_{\{e\}}\cap U_{\{e\}}^{g^{-n}h^{n}}]\mid n\in\bN\right\}\]
is bounded.  Then there exists $M'\in \bN$ such that $n\in \bN$ and $v\in V(\cT_e(h^{-n}g^{n}))$ implies  $d(v,\axis(g))\le M'$.
\end{lemma}
\begin{proof}
Fix $m,n\in\bN$ and suppose $v\in S(h^{-m}g^{n})$. We must have $v\in S(g^{n})\cup g^{-n}S(h^{-m})$. We separate these two cases.

If $v\in S(g^{n})$, applying Lemma \ref{lem:SingOfPowers} shows $g^{k}(v)\in S(g)$ for some $0\le k < n$. But $d(v,\axis(g)) = d(g^{k}(v),\axis(g))$, and so choosing $M_1$ such that $d(u,\axis(g))\le M_1$ for all $u\in S(g)$ gives $d(v,\axis(g))\le M_1$.

If $g^{n}(v)\in S(h^{-m})$, Lemma \ref{lem:SingOfPowers} shows $h^{-k}g^{n}(v)\in S(h^{-1})$ for some $k\ge 0$. Since $\omega_{+}(g) = \omega_{+}(h)$ and $S(h^{-1})$ is finite, Lemma \ref{lem:dist_to_axis_seq} applied to $h$ and $g$ shows that $\{d(h^n(u),\axis(g))\mid u\in S(h^{-1}), n\in\bN\}$ is bounded, say by $M_2\in\bN$. Thus $d(g^n(v),\axis(g)) = d(v,\axis(g))\le M_2$. Setting $M = \max\{M_1,M_2\}$ gives the required bound.

Suppose $F$ has distinct point stabilisers and $\{[U_{\{e\}}:U_{\{e\}}\cap U_{\{e\}}^{g^{-n}h^{n}}]\mid n\in\bN\}$ is bounded. Applying \cite[Lemma 7]{Baumgartner06} shows $\{[U_{e}:U_{e}\cap U_{e}^{g^{-n}h^{n}}]\mid n\in\bN\}$ is also bounded. Lemma \ref{lem:equal_ends} shows $\{d(e,g^{-n}h^{n}(e))\mid n\in\bN\}$ is bounded. We must have $\{d(\axis(g),g^{-n}h^{n}(e))\mid n\in\bN\}$ also bounded. Applying the first assertion and completes the result as $\cT_e(h^{-n}g^{n})$ is the minimal subtree containing $e$, $g^{-n}h^{n}(e)$ and $S(h^{-n}g^{n})$.
\end{proof}
\begin{definition}
Suppose $g\in G(F,F')$ and $e\in E(T)$. Let $P_e(g)$ denote the collection of paths $(v_0,\ldots,v_k)$ such that $(v_0,v_1)\in \{e,\overline{e}\}$ and $v_k$ is a vertex in $\cT_{e}(g)$ which is not internal. Set 
\[p_e(g) := |\{d(v,v')\mid  \hbox{ there exists }\gamma\in P_e(g)\hbox{ from } v \hbox{ to } v'\}|.\]
\end{definition}
\begin{lemma}\label{lem:paths_bound_cosets}
Suppose $F$ is $2$-transitive and $g\in G(F,F')$ is hyperbolic. Then \[[U_{\{e\}}:U_{\{e\}}\cap U_{\{e\}}^{g}]\ge p_e(g).\]
\end{lemma}
\begin{proof}
Choose $\{\gamma_{i}\mid 1\le i \le p_e(g)\}\subset P_e(g)$ such that $i < j$ implies the length of $\gamma_i$ is strictly less that $\gamma_j$. Since $F$ is $2$-transitive, $U(F)$ acts transitively on paths with the same length. This implies there exists $x_i\in U_{\{e\}}$ such that $x_i(\gamma_i)\subset \gamma_{p_e(g)}$. Then for $i < j$, we have $\gamma_i\subsetneq x_i^{-1}x_j(\gamma_j)$ as these are two paths with  different lengths. Since $\gamma_i$ ends in a leaf of $\cT_e(g^{-1})$, we must have $x_{i}^{-1}x_{j}\not\in U_{\{\cT_e(g^{-1})\}}$. Lemma \ref{lem:U_stab_rel} shows $x_i^{-1}x_j\not\in U_{\{e\}}\cap U_{\{e\}}^{g}$, and so $[U_{\{e\}}:U_{\{e\}}\cap U_{\{e\}}^{g}]\ge p_e(g)$. 
\end{proof}

\begin{proof}[Proof of Theorem \ref{thm:asymptotic_and_length_classification}]
Suppose there exist $p,q, M\in \bN$ such that $M$ bounds $\{\cN_{e}(g^{-pn}h^{qn})\mid n\in\bN\}$. Recall that from Lemma \ref{lem:cos_g_met}
\[d(U_{\{e\}}^{g^{pn}}, U_{\{e\}}^{h^{qn}}) = \log([U_{\{e\}}^{g^{pn}}:U_{\{e\}}^{g^{pn}}\cap U_{\{e\}}^{h^{qn}}][U_{\{e\}}^{h^{qn}}:U_{\{e\}}^{g^{pn}}\cap U_{\{e\}}^{h^{qn}}]).\]
Applying Lemma \ref{lem:length_bound_coset_above} shows
\begin{align*}
[U_{\{e\}}^{g^{pn}}:U_{\{e\}}^{g^{pn}}\cap U_{\{e\}}^{h^{qn}}] &= [U_{\{e\}}:U_{\{e\}}\cap U_{\{e\}}^{g^{-pn}h^{qn}}]\\
&\le (2(\deg(T)-1)^{\cN_{e}(g^{-pn}h^{qn})})!\\
&\le (2(\deg(T)-1)^M)!.
\end{align*}
We also have 
\[[U_{\{e\}}^{h^{qn}}:U_{\{e\}}^{g^{pn}}\cap U_{\{e\}}^{h^{qn}}]\le (2(\deg(T)-1)^M)!,\]
by the same argument but noting that $\cN_{e}(h^{-qn}g^{pn}) = \cN_{e}((h^{-qn}g^{pn})^{-1})$. Thus,
\[d(U_{\{e\}}^{g^{pn}}, U_{\{e\}}^{h^{qn}})\le 2\log((2(\deg(T)-1)^M)!),\]
which shows $g$ and $h$ are asymptotic.

Now suppose that for all $p,q\in\bN$, $\{\cN_{e}(g^{-pn}h^{qn})\mid n\in\bN\}$ is unbounded. Lemma \ref{lem:equal_ends} shows we can assume $\omega_{+}(h) = \omega_{+}(g)$ as otherwise we are done. This implies there exists an edge with endpoints in $\axis(g)\cap \axis(h)$. Equation \eqref{eq:comparable_card_def} implies, without loss of generality, we can suppose $e$ is such an edge.

Choose any $p,q\in\bN$ and fix a constant $M_1\in\bN$. Lemma \ref{lem:cylinder_of_singularities} allows us to suppose there that exists $M_{2}\in\bN$ such that $v\in \bigcup_{n\in\bN}\cT_e(h^{-qn}g^{pn})$ implies $d(v,\axis(g))< M_2$. Since $\{\cN_{e}(g^{-pn}h^{qn})\mid n\in\bN\}$ is unbounded and $\cN_e$ is a length function, $\{\cN_{e}(h^{qn}g^{-pn})\mid n\in\bN\}$ is also unbounded. Hence, the lengths of paths in $\cT_e(h^{-qn}g^{pn})$ is unbounded as $n\to\infty$. Choose $n\in\bN$ sufficiently large such that $P_e(h^{-qn}g^{-pn})$ contains a path 
\[(v_0,\ldots, v_{M_1M_2},\ldots,v_{(M_1+1)M_2},\ldots,  v_l)\]
where $l > (M_1+1)M_2$. By definition of $P_e(g^{-pn}h^{qn})$, we have $(v_0,v_1)\in\{e,\overline{e}\}$ and $d(v_l,\axis(g))\le M_2$. Since \[d(v_l,v_{M_1M_2}) = l - M_1M_2 > (M_1+1)M_2 - M_1 = M_2\] 
and $v_0,v_1\in\axis(g)$ our choice of $M_2$ shows we must have $(v_0,\ldots, v_{M_1M_2})\subset \axis(g)$. For each $0 < k \le M_1$, we must have $v_{kM_2}\in\Int(\cT_e(h^{-qn}g^{pn}))$ since $\cT_e(h^{-qn}g^{pn})$ is a complete subtree. It follows that there exist paths $p_k\in P_e(h^{-qn}g^{pn})$ such that $p_k$ contains $(v_0,\ldots,v_{kM_2})$ but not $v_{kM_2+1}$. Since our choice of $M_2$ implies the end of $p_k$ is distance at most $M_2 - 1$ from $\axis(g)$, the length of $p_k$ is at least $kM_2$ but strictly less than $(k+1)M_2$. This gives $M_1$ paths of different lengths in $P(h^{-qn}g^{pn})$. Applying Lemma \ref{lem:paths_bound_cosets}, we have
\[[U_{\{e\}}:U_{\{e\}}\cap U_{\{e\}}^{h^{-qn}g^{pn}}]\ge p_e(h^{-qn}g^{pn})\ge M_1.\]
But,
\begin{align*}
d(U_{\{e\}}^{g^{pn}},U_{\{e\}}^{h^{qn}}) &= \log([U_{\{e\}}^{g^{pn}}:U_{\{e\}}^{g^{pn}}\cap U_{\{e\}}^{h^{qn}}][U_{\{e\}}^{h^{qn}}:U_{\{e\}}^{g^{pn}}\cap U_{\{e\}}^{h^{qn}}])\\
& \ge \log([U_{\{e\}}:U_{\{e\}}\cap U_{\{e\}}^{h^{-qn}g^{pn}}])\\
&\ge \log(M_1).
\end{align*}
Since our choice of $M_1$ was arbitrary, we have $g\not\asymp h$.
\end{proof}

\subsection{Asymptotic classes and the local action at a vertex}
\label{ssec:asympt_func_and_dir}
In this section we give another equivalent condition for being asymptotic which is complimentary to that given in Theorem \ref{thm:asymptotic_and_length_classification}.  For Definition \ref{def:asymp_func}, note that for $g\in G(F,F')$ hyperbolic and $v\in V(T)$, there exists $N\in\bN$ such that $g^{-n}(v)\not\in S(g)$ for all $n\ge N$. Thus, for $n\ge N$ we have 
\[\sigma(g^{n},g^{-n}(v)) = \sigma(g^N, g^{-N}(v))\sigma(g^{n - N}, g^{-n}(v))\in\sigma(g^N, g^{-N}(v))F.\] 
In particular, the sequence $\sigma(g^{n},g^{-n}(v))F$ is eventually constant.
\begin{definition}\label{def:asymp_func}
For $g\in G(F,F')$ hyperbolic and $v\in V(T)$ define
\[\lambda_g(v) = \lim_{n\to\infty}\sigma(g^{n},g^{-n}(v))F,\]
and set $H_g(v)$ to be the smallest natural number such that $\sigma(g^n,g^{-n}(v))F = \lambda_g(v)$ for all $n\ge H_g(v)$.
\end{definition}

\begin{lemma} \label{lem:basic_asym_sing_prop}
Suppose $g\in G(F,F')$ is hyperbolic and $v\in V(T)$. Then:
\begin{enumerate}[label = (\roman{*})]
	\item\label{itm:traj_1} for all $k\ge 0$ we have $\lambda_g(v) = \sigma(g^{k},g^{-k}(v))\lambda_g(g^{-k}(v))$.
	\item\label{itm:traj_3} if $g^{-k}(v)\not\in S(g)$ for all $k\ge 1$, then $\lambda_g(g^{n}(v)) = \sigma(g^n,v)F$ for all $n\ge 0$. In particular, if $d(v,\axis(g))> D_g$, then $\lambda_g(v) = F$.
	\item\label{itm:traj_4} $\lambda_g(v) = \lambda_{g^k}(v)$ for all $k> 0$.
\end{enumerate}
\end{lemma}
\begin{proof}
For \ref{itm:traj_1}, note that for $n>\max\{H_g(v),  H_g(g^{-k}(v)) + k\}$
\begin{align*}\lambda_g(v) &= \sigma(g^n,g^{-n}(v))F\\ &= \sigma(g^k,g^{-k}(v))\sigma(g^{n - k}, g^{-n}(v))F\\ &= \sigma(g^k,g^{-k}(v))\sigma(g^{n - k}, g^{k - n}g^{-k}(v))F.
\end{align*}
But $n-k \ge  H_g(g^{-k}(v))$, hence $\sigma(g^{n - k}, g^{k - n}g^{-k}(v))F = \lambda_{g}(g^{-k}(v))$.

For \ref{itm:traj_3}, note that by \ref{itm:traj_1} we have 
\[\lambda_{g}(g^n(v)) = \sigma(g^n,g^{-n}g^{n}(v))\lambda_g(g^{-n}g^{n}(v)) = \sigma(g^n,v)\lambda_g(v)\].
But $g^{-k}(v)\not\in S(g)$ for all $k\in\bN$, hence $\lambda_g(v) = \lim_{n\to\infty}\sigma(g^k,g^{-k}(v))F = F$ .

For \ref{itm:traj_4}, fix $k >1$. Observe that for $n\ge H_g(v)$, we have $\lambda_g(v) = \sigma(g^{kn},g^{-kn}(v))F$. Taking the limit as $n\to\infty$ gives $\lambda_{g^k}(v) = \lambda_g(v)$.
\end{proof}
\begin{lemma}\label{lem:special_set}
Suppose $g,h\in G(F,F')$ are hyperbolic such that $\omega_{+}(g) = \omega_{+}(h)$ and $l(g) = l(h)$. Let $\cP_{g}$ and $\cP_{h}$ be pandos for $g$ and $h$ respectively. There exists a finite set $V\subset V(T)$ such that if $v\not\in V$, then both of the following hold: 
\begin{enumerate}[label = (\roman{*})]
	\item\label{itm:tech_lem_1} $h^{k}(v)\in \cP_{h}$ for some $k\ge 0$ implies  $g^{-j}g^{-n}h^{n}(v)\not\in \cP_{g}$ for all $j,n\in \bN$; and
	\item\label{itm:tech_lem_2}$h^{-k}(v)\in\cP_{h}$ for some $k\ge 0$ implies $g^{-j}h^{n}(v)\not\in \cP_{g}$ for all $n\ge j\ge 0$.
\end{enumerate}
\end{lemma}
\begin{proof}
Set $v_{\min}$ and $v_{\max}$ to be the minimum and maximum elements respectively of $\pi_{g}V(\cP_{g})$ with the order $\le_{g}$. Our proof splits into two cases.

\noindent{\bfseries{Case 1: }} $\omega_{-}(g) = \omega_{-}(h)$. 

In this case $\axis(h) = \axis(g)$, we have $\pi_{g} = \pi_{h}$ and the relations $<_g$ and $<_h$ agree. The additional assumption that $l(g) = l(h)$ implies $h(v) = g(v)$ for all $v\in\axis(g)$. This, and that $\cP_{h}$ is finite, implies there exists $k_0\in \bN$ such that for all $v\in V(\cP_{h})$ we have $\pi_{h}h^{-k_0}(v)<_h v_{\min}$ and $\pi_{h}h^{k_0}(v)>_h v_{\max}$. Set $V = \bigcup_{k = -k_0}^{k_0}h^{k}(V(\cP_{h}))$. 

Suppose $v\not\in V$. We start by showing \ref{itm:tech_lem_1}. Assume $h^{k}(v)\in \cP_{h}$ for some $k \ge 0$. Then $\pi_{h}h^{-k_0}h^{k}(v)<_h v_{\min}$. Furthermore, since $v\not\in V$ we have $k> k_0$ and so $\pi_hh^{-k+k_0}(u)<_h\pi_hu$ for all $u\in V(T)$. Thus 
\[\pi_{h}(v) = \pi_{h}h^{-k + k_0}h^{-k_0}h^{k}(v)<_h \pi_hh^{-k_0}h^{k}(v)<_h v_{\min}.\] 
However, $g(u) = h(u)$ for all $u\in\axis(g)$. Thus, $\pi_{h}(v) = \pi_{g}g^{-n}h^{n}(v)$ for all $n\in\bN$. It follows that \[\pi_{g}g^{-j}g^{-n}h^{n}(v) = \pi_{g}g^{-j}(v) \le_g \pi_{g}(v)<_gv_{\min}\] 
for all $j\in\bN$. The definition of $v_{\min}$ shows $g^{-j}g^{-n}h^{n}(v)\not\in\cP_{g}$.

For \ref{itm:tech_lem_2}, if $h^{-k}(v)\in \cP_{h}$ for some $k > 0$, then again we must have $k > k_0$ as $v\not\in V$ by assumption. Using a similar argument to that of the previous paragraph, we see that
\[\pi_{h}(v) = \pi_{h}h^{k-k_0}h^{k_0}h^{-k}(v)>_h \pi_hh^{k_0}h^{-k}(v)>_h v_{\max}.\] 
Again, since $g(u) = h(u)$ for all $u\in\axis(g)$, for all $j,n\in\bN$ with $n \ge j$ we have
\[\pi_{g}g^{-j}h^{n}(v) \ge_g \pi_g(v) >_g v_{\max}.\]
In particular, $g^{-j}h^{n}(v)\not\in\cP_{g}$. This completes the proof in the case when $\omega_{-}(g) = \omega_{-}(h)$.

\noindent{\bfseries{Case 2: }} $\omega_{-}(h) \neq \omega_{-}(g)$. 

In this case, for all $v\in V(T)$, the sequence $d(h^{-k}(v),\axis(g))_{k\in\bN}$ is unbounded and eventually non-decreasing by Lemma \ref{lem:dist_to_axis_seq}. Set $B = \max\{d(v,\axis(g))\mid v\in V(\cP_g)\}$. Since $\omega_{+}(h) = \omega_{+}(g)$ by assumption, $\axis(h)$ and $\axis(g)$ have infinite intersection. Hence, for any $v\in V(T)$, there exists $K\in\bN$ such that for all $k\ge K$ we have $\pi_{g}h^k(v) = \pi_hh^k(v)$. Since $\cP_h$ and $\cP_g$ are finite, there exists $k_0,k_1,k_2\in \bN$ such that for all $v\in V(\cP_{h})$, we have
\begin{enumerate}
	\item $\pi_gg^{-k_0}(v)<_g v_{\min}$. 
	\item $k>k_1$ implies $d(h^{-k}(v),\axis(g))> B$. \label{itm:dist_assump}
	\item $\pi_{g}h^{k_2}(v) = \pi_{h}k^{k_2}(v)>_g v_{\max}$.
\end{enumerate}
Set 
\[V = \bigcup_{i = -k_0-k_1}^{k_2}h^{i}(\cP_{h}).\]

Suppose $v\not\in V$. For \ref{itm:tech_lem_1}, suppose $h^{k}(v)\in \cP_{h}$ for some $k\ge 0$ and fix $n\in\bN$. We must have $k > k_0+k_1$ as $v\not\in V$. We assume that $d(h^{n-k}h^{k}(v),\axis(g))\le B$ as otherwise, since the action of $g$ preserves distance from $\axis(g)$, choice of $B$ implies $g^{-n}h^{n-k}h^{k}(v)\not\in \cP_g$. From \ref{itm:dist_assump} we have $-k_1\le n - k$ and so $n \ge k - k_1 > k_0$. If $n - k < 0$, then since $\omega_{+}(g) = \omega_{+}(h)$, we have $\pi_{g} h^{n-k}h^{k}(v)\le_g \pi_{g}h^{k}(v)$. By choice of $k_0$ we have
\[\pi_gg^{-n}h^n(v) = \pi_{g}g^{-n}h^{n-k}h^{k}(v)\le_g\pi_gg^{-n}h^{k}(v)\le_g \pi_gg^{-k_0}h^{k}(v)<_g v_{\min}.\]
This shows $g^{-j}g^{-n}h^{n}(v)\not\in\cP_{g}$ for all $j \ge 0$. Alternatively, if $n - k \ge 0$, then because $l(g) = l(h)$ and $\omega_{+}(g) = \omega_{+}(h)$,  
\[\pi_{g}h^{n-k}h^{k}(v)\le_g\pi_{g}g^{n-k}h^{k}(v). \]
It follows that 
\begin{align*}
\pi_{g}g^{-n}h^{n}(v)& = g^{-n}\pi_{g}h^{n}h^{-k}h^{k}(v) \le_{g}g^{-n}\pi_{g}g^{n - k}h^{k}(v)\\
& = \pi_{g}g^{-k+k_0}g^{-k_0}h^{k}(v).
\end{align*}
But $h^{k}(v)\in \cP_{h}$ and $-k +k_0 < 0$. Choice of $k_0$ gives $\pi_{g}g^{-k+k_0}g^{-k_0}h^{k}(v)<_{g} v_{\min}$. In particular $\pi_{g}g^{-n}h^{n}(v)<_{g}v_{\min}$ and so $g^{-j}g^{-n}h^{n}(v)\not\in\cP_{g}$.

For \ref{itm:tech_lem_2}, if $h^{-k}(v)\in \cP_{h}$ for some $k\ge 0$, then
we must have $k > k_2$ as $v\not\in V$. Thus, we may write $v = h^{k}h^{-k}(v)$, where $k > k_2$ and $h^{-k}(v)\in V(\cP_{h})$. Choice of $k_2$ gives $\pi_{g}(v) = \pi_{h}(v) >_g v_{\max}$. Since $l(g) = l(h)$ and $\omega_{+}(g) = \omega_{+}(h)$, for all $n,j\in\bN$ with $n \ge j$ we have
\[\pi_{g}g^{-j}h^{n}(v) = \pi_{g}g^{n-j}(v)\ge_g \pi_g(v)>_g v_{\max}.\]
In particular $g^{-j}h^{n}(v)\not\in \cP_{g}$.
\end{proof}
\begin{proposition}\label{prop:classification_in_terms_of_sing_traj}
Suppose $F$ is $2$-transitive and $g,h\in G(F,F')$ are hyperbolic. Then the following are equivalent:
\begin{enumerate}
	\item We have $\omega_{+}(g) = \omega_{+}(h)$ and for all $v\in V(T)$, we have $\lambda_g(v) = \lambda_h(v)$;
	\item $g$ and $h$ are asymptotic.
\end{enumerate}
\end{proposition}
\begin{proof}
First suppose $\omega_{+}(g) = \omega_{+}(h)$ and $\lambda_g(v) = \lambda_h(v)$ for all $v\in V(T)$. It suffices to show $g^{n_1}\asymp h^{n_2}$ for some $n_1,n_2\in\bN$ as $\asymp$ is an equivalence relation which is closed under taking powers. Part \ref{itm:traj_4} of Lemma \ref{lem:basic_asym_sing_prop} shows that, by taking powers of $g$ and $h$ if necessary, we may assume $l(g) = l(h)$. Since $\axis(g)\cap \axis(g)$ is an infinite path, choose an edge $e$ with $o(e),t(e)\in\axis(g)\cap\axis(h)$. Then $g^{n}(e) = h^{n}(e)$ for all $n\in\bN$. To show $g\asymp h$ it suffices to show $\{\cS(g^{-n}h^{n})\mid n\in\bN\}$ is bounded as this would imply that $\{\cH_{e}(g^{-n}h^{n})\mid n\in\bN\}$ is bounded which, by Theorem \ref{thm:asymptotic_and_length_classification}, shows $g\asymp h$. 

Choose a pando $\cP_{g}$ for $g$. Since $\cP_{g}$ is finite, there exists $B\in\bN$ such that $v\in V(\cP)$ implies $d(v,\axis(g)\cap \axis(h))<B$. Choose any pando $\cP_{h}$ for $h$ containing a vertex $u$ with $d(u,\axis(h))>B$.

Applying Lemma \ref{lem:special_set} to $\cP_{g}$ and $\cP_{h}$, there exists a finite set $V_1\subset V(T)$ such that for any $v\not\in V_{1}$
\begin{enumerate}
	\item $h^{k}(v)\in \cP_{h}$ for some $k\ge 0$ implies $g^{-j}g^{-n}h^{n}\not\in \cP_{g}$ for all $j,n\in \bN$; and
	\item $h^{-k}(v)\in\cP_{h}$ for some $k\ge 0$ implies $g^{-j}h^{n}(v)\not\in \cP_{g}$ for all $n\ge j\ge 0$.
\end{enumerate}
Set $V = V_{1}\cup V(\cP_{h})$ and fix $n\in\bN$. We show  $\cS(g^{-n}h^{n})\subset V$. Since $V$ is finite, this proves the result. To this end suppose $v\not\in V$. We consider three possibilities: $h^{k}(v)\in \cP_{h}$ for some $k\ge 0$; $h^{-k}(v)\in\cP_{h}$ for some $k\ge 0$; and $h^{k}(v)\not\in\cP_{h}$ for all $k\in\bZ$. 

Suppose first that $h^{k}(v)\in \cP_{h}$ for some $k\ge 0$. Since $v\not\in\cP_{h}$ and $\cP_h$ is a pando for $h$, we must have $k > 0$ and $h^{-j}(v)\not\in \cP_{h}$ for all $j \ge 0$. Then $\lambda_{h}(h^n(v)) = \sigma(h^n,v)F$ by part \ref{itm:traj_3} of Lemma \ref{lem:basic_asym_sing_prop} and \ref{pando1} in Definition \ref{def:pando}. Since $v\not\in V_1$, we have $g^{-j}g^{-n}h^{n}(v)\not\in \cP_{g}$ for all $j\ge 0$. Thus, the sequence $\sigma(g^{i},g^{-i}h^{n}(v))F$ is constant for $i\ge n$. This shows $\lambda_{g}(h^n(v)) = \sigma(g^{n}, g^{-n}h^{n}(v))F$. By assumption $\lambda_{g}(h^n(v)) = \lambda_h(h^n(v))$ and so 
\[\sigma(g^{n}, g^{-n}h^{n}(v))^{-1}\sigma(h^{n},v)\in F.\]
Hence
$\sigma(g^{-n}h^{n},v) = \sigma(g^n,g^{-n}h^{n}(v))^{-1}\sigma(h^n,v)\in F$. 

Now suppose $h^{-k}(v)\in\cP_{h}$ for some $k\ge 0$. Again we must have $k > 0$ and $h^{j}(v)\not\in\cP_{h}$ for all $j\ge 0$. Since $v\not\in V_{1}$, we have $g^{-j}h^{n}(v)\not\in \cP_{g}$ for all $n \ge j\ge 0$. Lemma \ref{lem:SingOfPowers} shows $\sigma(g^{-n}h^{n},v) = \sigma(g^{-n},h^{n}(v))\sigma(h^{n},v)\in F$. Combining all previous cases shows $S(g^{-n}h^n)\subset V$ as required.	

Finally, suppose $h^{k}(v)\not\in\cP_{h}$ for all $k\in\bZ$. Since $\cP_{h}$ is a pando for $h$, \ref{pando2} and \ref{pando3} show $d(h^{k}(v),\axis(h))> B$ for all $k\in\bZ$. Noting that $\omega_{+}(g) = \omega_{+}(h)$ we have
\[d(g^{-j}h^{n}(v),\axis(h)\cap \axis(g)) \ge d(h^{n}(v),\axis(h)\cap \axis(g)) > B.\]
Choice of $B$ shows $g^{-j}h^{n}(v)\not\in \cP_{g}$ for all $j\ge 0$. Lemma \ref{lem:SingOfPowers} shows $\sigma(g^{-n}h^{n},v)\in F$.

Now suppose $g\asymp h$. Lemma \ref{lem:equal_ends} shows $\omega_+(g) = \omega_{+}(h)$. Theorem \ref{thm:asymptotic_and_length_classification} gives $p,q\in\bN$ and $e\in E(T)$ such that $\{\cH_{e}(g^{-pn}h^{qn})\mid n\in\bN\}$ is bounded. We can assume without loss of generality that $g = g^{p}$ and $h = h^{q}$ by noting that $g^p\asymp g\asymp h \asymp h^q$ and applying Lemma \ref{lem:basic_asym_sing_prop}. 

Suppose $v\in V(T)$. Since $\{\cH_{e}(g^{-j}h^{j})\mid j\in\bN\}$ is bounded, it follows from the definition of $\cH_e$ and $\cT_e(g^{-j}h^j)$ that vertices $S(g^{-j}h^{j})$ are contained within a bounded distance of $e$. There exists $N\in\bN$ such that $n\ge N$ implies $h^{-n}(v),g^{-n}(v)\not\in S(g^{-j}h^{j})$ for all $j\in\bN$. Choosing $n\ge N$ sufficiently large such that $\lambda_h(v) = \sigma(h^{n},h^{-n}(v))F$ and $\lambda_{g}(v) = \sigma(g^{n},g^{-n}(v))F$, we have
\[\lambda_g(v)^{-1}\lambda_h(v) = F\sigma(g^{n},g^{-n}(v))^{-1}\sigma(h^{n},h^{-n}(v))F = F\sigma(g^{-n}h^{n},h^{-n}(v))F = F.\] 
Thus $\lambda_g(v) = \lambda_h(v)$.
\end{proof}
\subsection{The distance between asymptotic classes}
In this section we consider the pseudometric defined on asymptotic classes. We show this pseudometric induces the discrete topology on the set of equivalence classes and is therefore a metric.
\begin{lemma}\label{lem:translation_of_edge_calc}
Suppose $g,h\in G(F,F')$ are hyperbolic and $e\in E(T)$. Then there exists a constant $K\in\bN$ such that \[|nl(g) - ml(h)|\le d(e,g^{-n}h^{m}(e)) + K\]
for all $n,m\in\bN$.
\end{lemma}
\begin{proof}
Since $g$ and $h$ act as translations along $\axis(g)$ and $\axis(h)$ respectively, we have \[d(g^{n}(o(e)),o(e)) = nl(g)+2d(o(e),\axis(g))\] and \[d(h^{m}(o(e)),o(e)) = ml(h)+2d(o(e),\axis(h)).\]
Set $K = 2d(o(e),\axis(h)) - 2d(o(e),\axis(g))$. Then
\begin{align*}
|nl(g) - ml(h)|	& = |d(g^{n}(o(e)),o(e)) - d(h^{m}(o(e)),o(e)) + K|\\
&\le |d(g^{n}(o(e)),o(e)) - d(h^{m}(o(e)),o(e))| +|K|.
\end{align*}
Applying the reverse triangle inequality, and noting that 
\[d(g^{n}(o(e)),h^{m}(o(e))) \le d(e,g^{-n}h^{m}(e))+2\]
shows $|K|+2$ is the required constant.
\end{proof}

\begin{remark}\label{rem:constant_in_2_tran}
	For the proof of Theorem \ref{thm:directions} and other results in Section \ref{sec:s_mult_leboudec}, we often assume that $F$ is $2$-transitive and make use of the constant $|F_a/(F_a\cap F_b)|$, where $a,b\in\Omega$ are distinct. That this does not depend on $a$ and $b$ is a consequence of $2$-transitivity. The Orbit-Stabiliser Theorem shows the constant is precisely $|\Omega| - 1 = \deg(T) - 1$.
\end{remark}

We are ready to show that the space of directions of $G(F,F')$ is discrete when $F$ is $2$-transitive. To do so we assume that $g,h\in G(F,F')$ are hyperbolic but not asymptotic. We then consider four cases depending on $\omega_+(g)$, $\omega_+(h)$ and $\delta_{+n}^{U,U}(g,h)$, here $U$ is a fixed compact open subgroup. The aim of each case is similar. Roughly, up to exchanging $n\in\bN$ for a subsequence $(n_i)\subset\bN$, we build a sequence of edges $(e_i)_{i\in\bN}$ such that $U$ fixes $e_0$, $U\cap U^{g^{-n}h^{n'}}$ fixes $e_n$ ($n'\in\bN$ depends on $n$), and $d(e_0, e_n)$ grows proportional to $nl(g)$ as $n\to\infty$. In the simpler cases this sequence of edges will be $g^{-n}h^{n'}(e_0)$ but in the more complicated case, specifically when $\omega_+(g) = \omega_{+}(h)$, Proposition \ref{prop:classification_in_terms_of_sing_traj} and Lemma \ref{lem:cylinder_of_singularities} are required to build the sequence. This case is the most difficult because as elements of $\Aut(T)$, $g$ and $h$ are asymptotic thus examination of the singularities of $g$ and $h$ is required. Following the construction of $(e_i)$, Lemma \ref{lem:coset_calc} is used to complete each case.

\begin{theorem}\label{thm:directions}
Suppose $F$ is $2$-transitive. The pseudometric on the asymptotic classes of $G(F,F')$ induces the discrete topology.
\end{theorem}
\begin{proof}
Fix $g\in G(F,F')$ moving towards infinity. We find a constant $B > 0$ such that for all $h\in G(F,F')$ moving towards infinity with $h\not\asymp g$ we have  $\delta_{+}(g,h)\ge B$. To this end, suppose $h\in G(F,F')$ moves towards infinity and $h\not\asymp g$.

Fix an edge $e\in E(T)$ such that $o(e), t(e)\in \axis(g)$ and set $U = U_{e}$. Define a function $f:\bN\to\bN_0$ by letting $f(n)$ be a natural number such that
\[\delta_{+n}^{U,U}(g,h) = \dfrac{\log([U^{g^{n}}: U^{h^{f(n)}}\cap U^{g^{n}}])}{n\log s(g)}.\] 
We split our argument into four cases depending on $f$ and $\omega_{+}(h)$. 

\noindent{\bfseries{Case 1}:} There exists $M\in\bN$ such that $f(n)\le M$ for all $n\in\bN$.\\
It follows that $\delta_+(g, h)\ge 1$ by \cite[Lemma 15]{Baumgartner06}. Propositions \ref{prop:U(F)_small_scale} and \ref{prop:U(F)_scale_calc} show that 
\[\delta_+(g,h)\ge 1 \ge \dfrac{l(g)\log(|\Omega| - 1)}{\log s(g)}.\]

\noindent{\bfseries{Case 2}:} $\omega_{-}(g) = \omega_{+}(h)$.\\
If $x\in U\cap U^{g^{-n}h^{f(n)}}$, then $x$ fixes $e$ and $g^{-n}h^{f(n)}(e)$. Lemma \ref{lem:coset_calc} shows
\[[U:U^{g^{-n}h^{f(n)}}\cap U] \ge [U: U\cap U_{g^{-n}h^{f(n)}}]\ge (|\Omega| - 1)^{d(e, g^{-n}h^{n}(e))}.\]
Since $\omega_-(g) = \omega_{+}(h)$, we have $\omega_+(g^{-1}) = \omega_+(h)$. It follows from \cite[Lemma 4.7]{Baumgartner15} that $l(g^{-n}h^{f(n)}) = nl(g)+f(n)l(h)$. Thus, \[d(g^{-n}h^{f(n)}(e), e)\ge l(g^{-n}h^{f(n)}) - 1\ge nl(g) - 1,\] 
and so
\[[U:U^{g^{-n}h^{f(n)}}\cap U] \ge  (|\Omega| - 1)^{nl(g) - 1}.\]
This implies 
\[\delta_{+n}^{U,U}(g,h)\ge (n l(g) - 1)\dfrac{\log(|\Omega| - 1)}{n\log s(g)}.\]
Taking the limit as $n\to\infty$ gives
\[\delta_+(g,h)\ge \dfrac{l(g)\log(|\Omega| - 1)}{\log s(g)}.\]
\\[.3cm]
\noindent{\bfseries{Case 3}:} $f(n)\to \infty$ as $n\to\infty$ and $\omega_{+}(g) = \omega_{+}(h)$.\\
Lemma \ref{lem:cylinder_of_singularities} gives $M_1\in\bN$ such that $v\in \cS(h^{-f(n)}g^{n})$ implies $d(v,\axis(g))\le M_1$. Proposition \ref{prop:classification_in_terms_of_sing_traj} shows there exists $v\in V(T)$ such that $\lambda_g(v)\neq\lambda_h(v)$. Observe that $n\ge H_g(v)$ and $f(n)\ge H_h(v)$ implies $\lambda_g(v) = \sigma(g^{n},g^{-n}(v))F$ and $\lambda_h(v) = \sigma(h^{f(n)},h^{-f(n)}(v))F$. Since these cosets are distinct, we have 
\[\sigma(h^{-f(n)}g^{n},g^{-n}(v)) = \sigma(h^{f(n)},h^{-f(n)}(v))^{-1}\sigma(g^{n},g^{-n}(v))\not\in F.\]
In particular, $g^{-n}(v)\in S(h^{-f(n)}g^{n})$.

As $g$ acts by translation along $\axis(g)$, there exists $M_{3}$ such that for $n\ge M_{3}$ we have $d(\pi_{g}g^{-n}(v),e)> M_1$ and $\pi_{g}g^{-n}(v)<_go(e)$. Since $f(n)\to\infty$ as $n\to\infty$, we can choose a sequence $(n_i)_{i\in\bN}\subset\bN$ such that:
\begin{enumerate}
	\item $n_{i+1}>n_i\ge H_g(v)$ and $f(n_{i+1})>f(n_i)$; and
	\item $n_0\ge M_3$ and $f(n_0)\ge H_h(v)$.
\end{enumerate}
For each $i\in\bN$, we have $d(\pi_{g}g^{-n_i}(v),e)> M_1$ since $n_i \ge n_0 \ge M_3$. Since $n_i \ge H_g(v)$ and $f(n_i)\ge H_h(v)$, we have $g^{- n_i}(v)\in S(h^{-f(n_i)}g^{n_i})$. Hence $d(g^{-n_i}(v),\axis(g))\le M_1$. There exists $v_0\in\axis(g)$ on the path between $e$ and $g^{-n_0}(v)$ with $d(v_0, g^{-n_0}(v)) = M_1$. Set $v_i = g^{n_0-n_i}(v_0)$. Then $v_i\in\axis(g)$ and $\pi_{g}g^{-n_i}(v)\le_g v_i \le_g o(e)$. Furthermore, we have $d(v_i, g^{-n_i}(v)) = M_1$ and $d(v_0,v_i) = (n_0 - n_i)l(g)$.

Now suppose $x\in U$ with $x(v_i)\neq v_i$. Since $x(v_i)$ is on the path between $xg^{-n_i}(v)$ and $e$, $x(v_i)$ is on the path between $xg^{-n_i}(v)$ and $\axis(g)$. But $d(x(v_i), xg^{-n_i}(v)) = M_1$ and so $d(xg^{- n_i}(v),\axis(g))> M_1$. Choice of $M_1$ shows $xg^{ - n_i}(v)\not\in \cS(h^{-f(n_i)}g^{n_i})$. This shows $x\not\in U^{g^{-n_i}h^{f(n_i)}}$. We have shown that $U\cap U^{g^{-n_i}h^{f(n_i)}}\le U\cap U_{v_i}$. Thus,
\[[U:U\cap  U^{g^{-n_i}h^{f(n_i)}}]\ge [U:U\cap U_{v_i}].\] 
Let $v_i'$ be the unique vertex distance one away from $v_i$ on the path between $v_i$ and $e$. Set $e_i = (v_i,v_i')$. Then $U\cap U_{v_i} = U\cap U_{e_i}$. Also, since $v_0$ is on the path between $v_i$ and $e$, we know $d(e_i,e)\ge d(v_i', v_0) = (n_i - n_0)l(g) - 1$. Lemma \ref{lem:coset_calc} shows
\[[U:U\cap U_{v_i}]\ge (|\Omega| - 1)^{(n_i - n_0)l(g) - 1}.\]
Hence,
\[\delta_{+n_i}^{U,U}(g,h)\ge \dfrac{(n_il(g) - n_0l(g) - 1)\log (|\Omega| - 1)}{n_i\log s(g)}\]
and so 
\[\delta_{+}(g,h)\ge \dfrac{l(g)\log (|\Omega| - 1)}{\log s(g)}.\]
\\[.3cm]
\noindent{\bfseries{Case 4}:} $f(n)\to \infty$ as $n\to\infty$ and $ \omega_{+}(h)\not\in\{\omega_{+}(g),\omega_{-}(g)\}$.\\
As in case 2,
\[[U:U\cap U^{g^{-n}h^{f(n)}}]\ge[U:U\cap U_{g^{-n}h^{f(n)}(e)}].\]
Since $\omega_{+}(g)\neq\omega_{+}(h)$ and $\omega_+(h)\neq \omega_-(g)$, there exists $M_1\in\bN$ such that $n\ge M_1$ implies $h^{n}(e)\not\subset\axis(g)$ and $\pi_{g}h^{n}(o(e)) = \pi_{g}h^{M_1}(o(e))$. There exists $M_2$ such that $n\ge M_2$ implies $\pi_{g}g^{-n}h^{M_1}(o(e)) <_g o(e)$. It follows that for $k_0,k_1 \ge M_1$ and $n_1\ge n_0\ge M_2$ we have 
\begin{align*}
d(g^{-n_1}\pi_gh^{k_1}o(e), g^{-n_0}\pi_gh^{k_0}o(e))& = d(g^{-n_1}\pi_gh^{M_1}o(e), g^{-n_0}\pi_gh^{M_1}o(e))\\
& = (n_1 - n_0)l(g)
\end{align*}
Since $f(n)\to \infty$ as $n\to \infty$, choose a sequence $(n_i)_{i\in\bN}\subset \bN$ such that:
\begin{enumerate}
	\item $n_{i+1}> n_i$ and $f(n_{i+1})> f(n_i)$; and
	\item $n_0 > M_2$ and $f(n_0) > M_1$.
\end{enumerate}
Since $n_i> n_0 > M_2$ and $f(n_i) > f(n_0) > M_1$ we have  \[\pi_gg^{-n_i}h^{f(n_i)}o(e)<_g \pi_g g^{-n_0}h^{f(n_0)}o(e) <_g o(e)\]
for all $i\in\bN$. This shows
\[d(e,g^{-n_i}h^{f(n_i)}(e))\ge d(\pi_g g^{-n_0}h^{f(n_0)}o(e),\pi_gg^{-n_i}h^{f(n_i)}o(e)) - 2.\] 
Applying Lemma \ref{lem:coset_calc}, we see that
\begin{align*}
[U:U\cap U_{g^{-n_{i}}h^{f(n_{i})}(e)}]\ge (|\Omega| - 1)^{d(e,g^{-n_i}h^{f(n_i)}(e))} \ge  (|\Omega| - 1)^{l(g)(n_i - n_0) - 2}.
\end{align*} 
Hence,
\[\delta_{+n_i}^{U,U}(g,h) \ge \dfrac{(n_il(g) - n_0l(g)-2)\log (|\Omega| - 1)}{n_i\log s(g)},\]
which implies
\[\delta_{+}(g,h)\ge \dfrac{l(g)\log (|\Omega| - 1)}{\log s(g)}.\qedhere\]
\end{proof}
Corollary \ref{cor:dist_bound_asym} follows directly from the proof of Theorem \ref{thm:directions}
\begin{corollary}\label{cor:dist_bound_asym}
	Suppose $F$ is $2$-transitive and $g,h\in G(F,F')$ are hyperbolic but not asymptotic. Then 
	\[\delta(g,h)\ge \log(|\Omega| - 1)\left(\dfrac{l(g)}{\log s(g)}+\dfrac{l(h)}{\log s(h)}\right).\]
\end{corollary}
Corollary \ref{cor:flt_rank} follows from Theorem \ref{thm:directions} and \cite[Theorem 3]{Baumgartner12}.
\begin{corollary}\label{cor:flt_rank}
	Suppose $F$ is $2$-transitive. Then $G(F,F')$ has flat rank $1$.
\end{corollary}

\section{Scale-multiplicative semigroups at infinity}\label{sec:s_mult_leboudec}

In this section we consider scale-multiplicative semigroups of restricted Burger-Mozes groups associated to asymptotic classes. We start by studying a weaker relation than asymptotic, see Definition \ref{def:wk_asymp_tran_c}. In Section \ref{ssec:asymp_is_semigroup} we show that this relation is stable under multiplication. This section builds on results from Section \ref{ssec:asympt_func_and_dir} utilises the functions $\lambda_g$, $g\in G$, defined in Definition \ref{def:asymp_func} which characterise the asymptotic relation, see Proposition \ref{prop:classification_in_terms_of_sing_traj}.  
Section \ref{ssec:asymp_scalemult} focuses on showing that the semigroups constructed are scale multiplicative. Our proofs rely on calculations given in Section \ref{ssec:leboudec_scale_formula}. It is an immediate corollary that these results also apply to asymptotic classes. Thus asymptotic classes are scale-multiplicative semigroups of non-uniscalar elements.

Observe that any scale-multiplicative semigroup can be decomposed into a disjoint union of a two scale-multiplicative semigroups, each containing only uniscalar or non-uniscalar elements respectively. To have any hope of building a maximal scale-multiplicative semigroup from asymptotic classes, the correct uniscalar elements need to be identified. Section \ref{ssec:asymp_uniscalar} focuses on establishing a candidate uniscalar component and showing that its addition still gives a scale-multiplicative semigroup. Section \ref{ssec:asymp_maximal} then considers maximality with the final result given as Theorem \ref{thm:maximal}. The proof of this theorem is split into two major parts. The first part, given as Lemma \ref{lem:end_semigroup_large}, gives a selection of many elements in the asymptotic class and places restrictions on the singularities of these elements. The proof then proceeds by cases on elements not in the candidate semigroup. Each case chooses an element given by Lemma \ref{lem:end_semigroup_large} and shows that the semigroup generated by it and the external element is not scale multiplicative. The most complicated of these cases is covered by Lemma \ref{lem:same_attracting_end_not_smult}. The proof is technical but can be summarised as establishing results about the singularities of the product and calculating a lower bound on the scale from these results. 

\begin{definition}\label{def:wk_asymp_tran_c}
	Suppose $g,h\in G(F,F')$ are hyperbolic. 
	\begin{enumerate}[label = (\roman{*})]
		\item We say $g$ and $h$ are \emph{translation compatible} for any pair $u,v\in\axis(g)\cap \axis(h)$, we have $u\le_g v$ if and only if $u\le_h v$. 
		\item We say $g$ and $h$ are \emph{weakly asymptotic} if $\lambda_g(v) = \lambda_h(v)$ for all $v\in V(T)$.
	\end{enumerate}
\end{definition}
\begin{remark}
	Results from \cite{Baumgartner15} show that $g$ and $h$ are translation compatible if and only if $l(gh)\ge l(g)+l(h)$.
\end{remark}
\begin{lemma}
	Weakly asymptotic is an equivalence relation on the set of hyperbolic elements of $G(F,F')$.
\end{lemma}

Hyperbolic elements $g,h\in G(F,F')$ where $\axis(g)\cap \axis(h)=\varnothing$ are trivially translation compatible. Hyperbolic $g,h\in G(F,F')$ where $\omega_+(g) = \omega_+(h)$ are also translation compatible. 
\begin{remark}\label{rem:asymp_in_terms_of_wk_asymp}
	Using Lemma \ref{lem:equal_ends}, we may rephrase Proposition \ref{prop:classification_in_terms_of_sing_traj} as follows:
	
	\emph{Suppose $g,h\in G(F,F')$ are hyperbolic. Then $g\asymp h$ if and only if $g$ and $h$ are weakly asymptotic and $\omega_+(g) = \omega_+(h)$.}
\end{remark}
\subsection{Closed under multiplication}
\label{ssec:asymp_is_semigroup}
We build semigroups from the relations given in Definition \ref{def:wk_asymp_tran_c}. Recall the definition $D_g = \max\{d(v,\axis(g))\mid v \in S(g)\}$ from Definition \ref{def:sing_depth}.
\begin{lemma}\label{lem:trans_compatible_basics}
	Suppose $g,h\in G(F,F')$ are translation compatible. Then:
	\begin{enumerate}[label = (\roman{*})]
		\item \label{itm:tr_compatibe_1} The product $gh$ is hyperbolic and translation compatible with both $g$ and $h$.
		\item \label{itm:tr_compatibe_2}$\omega_{-}(g)\neq\omega_{+}(h)$.
		\item \label{itm:tr_compatibe_3}$g^{-1}$ and $h^{-1}$ are also translation compatible. 
	\end{enumerate}	
\end{lemma}
\begin{proof}
	Noting $\axis(g) = \axis(g^{-1})$ and that the relation $\le_g$ on $\axis(g)$ agrees with $\ge_{g^{-1}}$, we see that \ref{itm:tr_compatibe_3} holds. If $\omega_{+}(g) = \omega_{-}(h)$, then there exists an infinite path $(v_0,v_1,\ldots)$ contained in $\axis(g)\cap \axis(h)$ such that $v_i <_g v_j$ and $v_i >_h v_j$ for all $i < j$. Thus, $g$ and $h$ cannot be translation compatible. This shows \ref{itm:tr_compatibe_2}. Finally, \ref{itm:tr_compatibe_1} follows from \cite[Lemma 4.6 and Lemma 4.7]{Baumgartner15}.
\end{proof}
\begin{lemma}\label{lem:random_natural}
	Suppose $g,h\in G(F,F')$ are hyperbolic and translation compatible. Then for any $v\in V(T)$, there exists $N\in\bN$ such that $n\ge N$ implies $g^{-k}h^{-n}(v)\not\in S(g)$ for all $k\ge 0$.
\end{lemma}
\begin{proof}
	Since $h$ and $g$ are translation compatible $\omega_{-}(h)\neq\omega_{+}(g)$. If $\omega_{-}(h) \neq \omega_{-}(g)$, then the sequence $d(h^{-n}(v),\axis(g))$ is unbounded and eventually non-decreasing by Lemma \ref{lem:dist_to_axis_seq}. There exists $N\in\bN$ such that $n\ge N$ implies
	\[d(h^{-n}(v),\axis(g))>D_g.\]
	Then for all $k\in\bN$ and $n\ge N$, we have $d(g^{-k}h^{-n}(v),\axis(g))> D_g$ and so $g^{-k}h^{-n}(v)\not\in S(g)$.
	
	Alternatively if $\omega_{-}(h) = \omega_{-}(g)$, let $v_{\min} = \min_{\le g}\pi_gS(g)$. Since $h$ and $g$ are translation compatible, there exists $N\in\bN$ such that $n\ge N$ implies $\pi_gh^{-n}(v) <_g v_{\min}$. Then for all $k\in\bN$ and $n\ge N$, $\pi_gg^{-k}h^{-n}(v)< v_{\min}$ and so $g^{-k}h^{-n}(v)\not\in S(g)$.
\end{proof}
\begin{proposition}\label{prop:wk_asymp_semigroup}
	Suppose $g,h\in G(F,F')$ are hyperbolic, translation compatible and weakly asymptotic. Then $h$ and $hg$ are weakly asymptotic.
\end{proposition}
\begin{proof}
	Choose any $v\in V(T)$. For $n\ge H_{hg}(v)$ we have $\lambda_{hg}(v) = \sigma((hg)^n,(hg)^{-n}(v))F$. Choosing $n$ larger if necessary, Lemma \ref{lem:random_natural} and Lemma \ref{lem:trans_compatible_basics} allows us to suppose $g^{-k}(hg)^{-n}(v)\not\in S(g)$ for $k\ge 0$. In particular, Lemma \ref{lem:basic_asym_sing_prop} shows that $\lambda_g((hg)^{-n}(v)) = F$. By assumption, $g$ and $h$ are weakly asymptotic and so $\lambda_h((hg)^{-n}(v)) = F$. We have
	\[\lambda_{hg}(v) = \sigma((hg)^n,(hg)^{-n}(v))F = \sigma((hg)^n,(hg)^{-n}(v))\lambda_h((hg)^{-n}(v)).\]

	We now show that for $k > 0$ 
	\[\sigma((hg)^{k},(hg)^{-k}(v))\lambda_h((hg)^{-k}(v)) = \sigma((hg)^{k - 1},(hg)^{-k+1}(v))\lambda_h((hg)^{-k+1}(v)).\]
	This completes the proof as inductive applications of this equality to $\lambda_{hg}(v)$ gives
	\begin{align*}
	\lambda_{hg}(v) &= \sigma((hg)^{n},(hg)^{-n}(v))\lambda_h((hg)^{-n}(v))\\
	&= \sigma((hg)^{n-1},(hg)^{-n+1}(v))\lambda_h((hg)^{-n+1}(v))\\
	&\hspace{0.2cm}\vdots\\
	& = \lambda_h(v).
	\end{align*}
	Now
	\[\sigma((hg)^{k},(hg)^{-k}(v)) = \sigma((hg)^{k - 1}h,g(hg)^{-k}(v))\sigma(g, (hg)^{-k}(v)),\]
	and, applying Lemma \ref{lem:basic_asym_sing_prop} after recalling that $g$ and $h$ are weakly asymptotic, we have 
	\[\sigma(g, (hg)^{-k}(v))\lambda_h((hg)^{-k}(v)) = \sigma(g, (hg)^{-k}(v))\lambda_g((hg)^{-k}(v)) = \lambda_g(g(hg)^{-k}(v)).\]
	Hence, 
	\begin{align*}
	\sigma((hg)^{k},(hg)^{-k}(v))\lambda_h((hg)^{-k}(v)) &=\sigma((hg)^{k - 1}h,g(hg)^{-k}(v))\lambda_g(g(hg)^{-k}(v))\\
	& = \sigma((hg)^{k - 1}h,g(hg)^{-k}(v))\lambda_h(g(hg)^{-k}(v)).
	\end{align*}
	But
	\[\sigma((hg)^{k - 1}h,g(hg)^{-k}(v)) = \sigma((hg)^{k - 1},(hg)^{-k+1 }(v))\sigma(h, g(hg)^{-k}(v))\]
	and
	\[\sigma(h, g(hg)^{-k}(v)) \lambda_h(g(hg)^{-k}(v)) = \lambda_h((hg)^{-k+1}(v)),\]
	be Lemma \ref{lem:basic_asym_sing_prop}.
	Hence, 
	\[\sigma((hg)^{k},(hg)^{-k}(v))\lambda_h((hg)^{-k}(v)) =  \sigma((hg)^{k - 1},(hg)^{-k+1 }(v))\lambda_h((hg)^{-k+1}(v))\]
	as required.
\end{proof}
\begin{corollary}
	Suppose $F$ is $2$ transitive and $g,h\in G(F,F')$ are both hyperbolic and asymptotic. Then $gh\asymp g$.
\end{corollary}
\begin{proof}
	Observe that $gh$ and $h$ are weakly asymptotic and translation compatible by Lemma \ref{lem:trans_compatible_basics} and Proposition \ref{prop:wk_asymp_semigroup}. The result follows from Remark \ref{rem:asymp_in_terms_of_wk_asymp} since $\omega_+(g) = \omega_+(gh)$.
\end{proof}
\subsection{Scale multiplicative}
\label{ssec:asymp_scalemult}
We investigate scale multiplicativity of the semigroups generated by elements that are weakly asymptotic and translation compatible. The simpler case is when the generators have differing attracting ends. This case is restrictive on elements as shown in Lemma \ref{lem:wk_asym_not_asymp} and Lemma \ref{lem:scale_weak_asyp}. The case when the generators have the same attracting ends, which occurs when looking at asymptotic classes, is more complicated and relies on a results from Section \ref{ssec:leboudec_scale_formula}.

\begin{lemma}\label{lem:wk_asym_not_asymp}
	Suppose $g,h\in G(F,F')$ are translation compatible with $\omega_+(g)\neq \omega_+(h)$. If $g$ and $h$ are weakly asymptotic, then for each $v\in V(T)$ there exists $N\in \bN$ such that $n\ge N$ implies $\lambda_{g}(g^n(v)) = F$.
\end{lemma}
\begin{proof}
	Since $g$ and $h$ are translation compatible $\omega_+(g)\neq \omega_{-}(h)$. This and the assumption that $\omega_{+}(g)\neq \omega_+(h)$ gives $N\in\bN$ such that $n\ge N$ implies $d(g^n(v),\axis(h))> D_h$. We must have $\lambda_h(g^n(v)) = F$, since $h^{-k}g^n(v)\not\in S(h)$ for all $k\ge 0$. Since $g$ and $h$ are weakly asymptotic, we have $\lambda_g(g^n(v)) = F$. 
\end{proof}

\begin{lemma}\label{lem:scale_weak_asyp}
	Suppose $F$ is $2$-transitive and $g\in G(F,F')$ is hyperbolic such that for each $v\in V(T)$, there exists $N\in \bN$ such that $n\ge N$ implies $\lambda_{g}(g^n(v)) = F$. Then 
	\[s(g) = (|\Omega|-1)^{l(g)}.\]
\end{lemma}
\begin{proof}
	Since $F$ is $2$-transitive, Corollary \ref{cor:2-trans_translations} gives $g'\in U(F)$ such that $l(g') = l(g)$ and $\omega_{\pm}(g) = \omega_{\pm}(g')$. We show that $s(g) = s(g')$. Observe that $\axis(g) = \axis(g')$. Choose a pando $\cP$ for $g$ with initial segment $\cP_0$. Then $\cP$ is a pando for $g'$, again with initial segment $\cP_0$. We claim that $M_{g,\cP_0} = M_{g',\cP_0}$. It is clear that $M_{g,\cP_0}\le M_{g',\cP_0}$ since $S(g') = \varnothing$. Suppose $\varphi\in M_{g',\cP_0}$. We have $\sigma(\varphi,u)\in F$ for all $u\in\Int(\cP_0)$. Also, if $u\in\Int(\cP_0)$, then $g^{-k}(u)\not\in S(g)$ for all $k\ge 1$ as $S(g)\subset \cP$. Lemma \ref{lem:basic_asym_sing_prop} shows that $k\in\bN$ we have $\lambda_{g}(g^k(u)) = \sigma(g^k,u)F$. Hence, for $k$ sufficiently large we have $\sigma(g^k,u)\in F$. Therefore, for any $u\in\Int(\cP_0)$ and $k\in \bN$ sufficiently large we have
	\[\sigma(g^k,\varphi(u))\sigma(\varphi,u)\sigma(g^{k},u)^{-1} \in F.\]
	Hence $\varphi\in M_{g,\cP_0}$.
	
	We have shown that $s(g) = s(g')$. Lemma \ref{lem:U(F)_hyp_scale} and Remark \ref{rem:constant_in_2_tran} complete the result. 
\end{proof}
To compare the scale of various $g,h\in G(F,F')$, it is useful to have a more flexible version of Proposition \ref{prop:Scale_calc_2}.
\begin{lemma}\label{lem:scale_flex}
	Suppose $F$ is $2$-transitive $g\in G(F,F')$ is hyperbolic, $v_0\in\axis(g)$. Choose $D > D_g$ and let $\cT$ the minimal complete subtree of $T$ satisfying:
	\begin{enumerate}[label = (\roman{*})]
		\item $v_0\in\Int(\cT)$ and $g(v_0)\in \cT$; and
		\item if $\pi_g(v)\in\Int(\cT)$ and $d(v,\axis(g))\le D $, then $v\in \cT$.
	\end{enumerate} 
	Define $M_{g,\cT}$ to be the set of automorphisms $\varphi$ of $\cT$ fixing $\axis(g)\cap \cT$ such that for $k\in\bZ$ with $|k|$ sufficiently large and $v\in\Int(\cT)$ \[\sigma(g^k,\varphi(v))\sigma(\varphi,v)\sigma(g^{-k},g^{k}(v))\in F.\]
	Then 
	\[s(g) = \dfrac{(|\Omega| - 1)^{|\Int(\cT)|}}{|M_{g,\cT}|}.\]
	
	Furthermore, if $v_0>_g\pi_g(u)$ for all $u\in S(g)$, then $M_{g,\cT}$ is the set of automorphisms $\varphi$ of $\cT$ fixing $\axis(g)\cap \cT$ such that $\sigma(\varphi,u)\in F\cap \lambda_g(\varphi(u))\lambda_g(u)^{-1}$ for all $u\in\Int(\cT)$.
\end{lemma}
\begin{proof}
	Choice of $\cT$ implies that there exists $n\in\bN$ such that $g^{-n}(\cT)$ is the initial segment of some pando $\cP$ for $g$. Proposition \ref{prop:Scale_calc_2} gives
	\[s(g) = \dfrac{[U_{g^{-n+1}(\cT)}:U_{g^{-n}(\cT)}\cap U_{g^{-n+1}(\cT)}]}{|M_{g,\cP}|}.\]
	Repeated applications of the Orbit-Stabiliser Theorem and the assumption that $F$ is $2$-transitive shows that the numerator is precisely  $(|\Omega|-1)^{|\Int(g^{-n}(\cT))|} = (|\Omega|-1)^{|\Int(\cT)|}$. To complete the proof of the first claim, note that by conjugating by $g^{-n_1}$ gives bijection between $M_{g,\cP_0}$ and $M_{g,\cT}$.
	
	For the final claim, suppose $v_0>_g\pi_g(u)$ for all $u\in S(g)$. Since $\pi_g(v)\ge v_0$ for all $v\in \Int(\cT)$, given $v\in\Int(\cT)$ we have $\sigma(g^k,v)\in F$. Also, for all $k\in\bN$ and for $k$ sufficiently large we have $\sigma(g^{k},g^{-k}(v))F = \lambda_g(v)$. Using these identities gives the required characterisation of $M_{g,\cT}$.
\end{proof}

\begin{proposition}\label{prop:asymp_is_s_mult}
	Suppose $F$ is $2$-transitive and $g,h\in G(F,F')$ are hyperbolic, translation compatible and weakly asymptotic such that $\axis(g)\cap \axis(h)$ is non-empty. Then the semigroup generated by $g$ and $h$ is scale-multiplicative. 
\end{proposition}
\begin{proof}
	First suppose $\omega_+(g)\neq\omega_+(h)$. If $h_0$ is in the semigroup generated by $g$ and $h$, then $h_0$, $g$ and $h$ are all weakly asymptotic by Proposition \ref{prop:wk_asymp_semigroup}. We must have $\omega_+(h_0)\neq \omega_+(g)$ or $\omega_+(h_0)\neq \omega_+(h)$ and so applying Lemma \ref{lem:scale_weak_asyp}, we see that $s(h_0) = (|\Omega| - 1)^{l(h_0)}$. Thus, to show that the semigroup is scale-multiplicative, it suffices to show that it the length function is additive. Since $\axis(g)\cap \axis(h)$ is in the axis of every element of the semigroup by \cite[Lemma 4.7]{Baumgartner15} and any two elements of the semigroup are translation compatible by Lemma \ref{lem:trans_compatible_basics}, the result follows from \cite[Lemma 4.7]{Baumgartner15}.
	
	Suppose now that $\omega_+(g) = \omega_+(h)$. Choose $v_0\in \axis(g)\cap \axis(h)$ such that:
	\begin{enumerate}[label = (\roman{*})]
		\item $v_0 >_g \pi_g(u)$ for all $u\in S(g)$;
		\item $v_0 >_g \pi_h(u)$ for all $u\in S(h)$;
		\item $v_0 >_g{gh}\pi_{gh}(u)$ for all $u\in S(gh)$.
	\end{enumerate}
	Note that $\{v_0, h(v_0), gh(v_0)\}\subset \axis(gh)\cap \axis(h)\cap \axis(g)$.
	Choose
	$D > \max\{D_g,D_h,D_{gh}\}$
	and let $\cT_1$ be the minimal complete subtree of $T$ such that:
	\begin{enumerate}[resume, label = (\roman{*})]
		\item $v_0\in\Int(\cT_{1})$ and $gh(v_0)\in\cT_1$; and
		\item if $\pi_{gh}(v)\in \cT_1$ and $d(v,\pi_{gh}(v))< D$, we have $v\in\cT_1$.
	\end{enumerate}
	Let $\cT_2$ be the minimal complete subtree of $\cT_1$ containing every vertex $v\in V(\cT_1)$ such that $v_0\le_{gh}\pi_{gh}(v)<_{gh}h(v_0)$. Similarly, let $\cT_{3}$ be the minimal complete subtree of $\cT_1$ containing every vertex $v\in V(\cT_1)$ such that $h(v_0)\le_{gh}\pi_{gh}(v)<_{gh}gh(v_0)$. Then $\cT_2$ and $\cT_3$ union to give $\cT_1$ and intersect to give and edge. Thus, $|\Int(\cT_1)| = |\Int(\cT_2)|+|\Int(\cT_3)|$. Applying Lemma \ref{lem:scale_flex} to the pairs $gh$ and $\cT_1$, $h$ and $\cT_2$, and $g$ and $\cT_3$, we see that ${s(gh) = s(g)s(h)}$ if and only if $|M_{gh,\cT_{1}}| = |M_{h,\cT_2}||M_{g,\cT_3}|$. Suppose $\varphi\in M_{gh,\cT_1}$. Equivalently, we have $\varphi\in \Aut(\cT_1)$ such that $\varphi$ fixes $\axis(gh)\cap \cT_1$ and ${\sigma(\varphi,v) \in F\cap \lambda_{gh}(\varphi(v))\lambda_{gh}(v)^{-1}}$ for all $v\in \Int(\cT_1)$. Let $\varphi_2$ and $\varphi_3$ be the restrictions of $\varphi$ to $\cT_2$ and $\cT_3$ respectively. Then $\varphi\in M_{gh,\cT_1}$ if and only if both of the following hold:
	\begin{enumerate}[label = (\alph{*})]
		\item\label{itm:equiv_1} $\varphi_2\in \Aut(\cT_2)$ fixes $\axis(gh)\cap\cT_2$ and $\sigma(\varphi_2,v)\in F\cap \lambda_{gh}(\varphi(v))\lambda_{gh}(v)^{-1}$ for all and $v\in \Int(\cT_2)$; and
		\item\label{itm:equiv_2} $\varphi_3\in \Aut(\cT_3)$ fixes $\axis(gh)\cap\cT_3$ and $\sigma(\varphi_3,v)\in F\cap \lambda_{gh}(\varphi_3(v))\lambda_{gh}(v)^{-1}$ for all and $v\in \Int(\cT_3)$.
	\end{enumerate}
	By assumption $\lambda_{gh}(v) = \lambda_{g}(v) = \lambda_{h}(v)$, and so \ref{itm:equiv_1} and \ref{itm:equiv_2} are equivalent to $\varphi_2\in M_{h,\cT_2}$ and $\varphi_3\in M_{g,\cT_{3}}$. We have given a bijection $M_{gh,\cT_{1}}\to M_{h,\cT_{2}}\times M_{g,\cT_3}$ as required.
\end{proof}
\subsection{Uniscalar elements}
\label{ssec:asymp_uniscalar}
To build maximal scale-multiplicative semigroups from asymptotic classes, the correct uniscalar elements to add need to be identified. To define these elements, observe that and t.d.l.c. group $G$ acts on $\partial G$ by conjugation, that is, set $g\partial h = \partial ghg^{-1}$. Let $G_{\partial g}$ denote the stabiliser of $\partial g$ under the action of $G$. Lemma \ref{lem:uni_and_asymp} shows that adding the uniscalar elements of $G(F,F')_{\partial g}$ to $\partial g$ gives a scale-multiplicative semigroup. The proof involves an alternate description of $G(F,F')_{\partial g}$ which is given in Lemma \ref{lem:stab_comp}. This description and Lemma \ref{lem:length_to_scale} is used to show that we still have a scale-multiplicative semigroup in Lemma \ref{lem:uni_and_asymp}. 

\begin{lemma}\label{lem:length_to_scale}
	Suppose $F$ is $2$-transitive and $g,h\in G(F,F')$ are hyperbolic with $g\asymp h$ such that $l(g) = l(h)$. Then $s(g) = s(h)$.
\end{lemma}
\begin{proof}
	Since $g\asymp h$, we have $\omega_+(g) = \omega_+(h)$ by Lemma \ref{lem:equal_ends}. Therefore, we may choose $v\in\axis(h)\cap \axis(g)$ such that: 
	\begin{enumerate}[label = (\roman{*})]
		\item $v>_g\pi_g(u)$ for all $u\in S(g)$; and
		\item $v>_h\pi_h(u)$ for all $u\in S(h)$.
	\end{enumerate}
	Choose $D>\max\{D_h,D_g\}$
	and $\cT$ be the unique minimal subtree of $T$ such that:
	\begin{enumerate}[resume, label = (\roman{*})]
		\item $v\in \Int(\cT)$ and $g(v) = h(v)\in \cT$; and
		\item If $u\in V(T)$ with $\pi_g(u) = \pi_h(u)\in \Int(\cT)$ and $d(u,\pi_g(u))\le D$, then $u\in \cT$.
	\end{enumerate} 
	To show $s(g) = s(h)$ it suffices to show $M_{g,\cT} = M_{h,\cT}$ by Lemma \ref{lem:scale_flex}. However, this is immediate as $\lambda_{g}(v) = \lambda_h(v)$ for all $v\in V(T)$.
\end{proof}

\begin{lemma}\label{lem:stab_comp}
	Suppose $F$ is $2$-transitive and $g\in G(F,F')$ is hyperbolic. Then 
	\begin{align*}
	G(F,F')_{\partial g} &= \{x\in G(F,F')_{\omega_+(g)}\mid \sigma(x,v)\lambda_{g}(v) = \lambda_g(x(v))\}\\
	& =  \{x\in G(F,F')_{\omega_+(g)}\mid   \sigma(x,x^{-1}(v))\lambda_g(x^{-1}(v)) = \lambda_g(v)\}.
	\end{align*}
\end{lemma}
\begin{proof}
	Suppose $x\in G(F,F')$.  Observe that $\omega_+(xgx^{-1}) = x\omega_+(g)$ since $xgx^{-1}$ acts by translation along $x\axis(g)$. Hence, if $xgx^{-1}\asymp g$, if follows from Lemma \ref{lem:equal_ends} that $x\in G(F,F')_{\omega_+(g)}$. Suppose this is the case. Then $x\in G(F,F')_{\partial g}$ if and only if $\lambda_{xgx^{-1}}(v) = \lambda_g(v)$ for all $v\in V(T)$. For a given $v\in V(T)$, if $n\ge \max\{H_{xgx^{-1}}(v), H_g(x^{-1}(v))\}$, then $\lambda_{xgx^{-1}}(v) = \sigma(xg^{n}x^{-1}, xg^{-n}x^{-1}(v))F$ and $\lambda_g(x^{-1}(v)) = \sigma(g^n,g^{-n}x^{-1}(v))F$. Choosing $n$ larger if necessary, since $S(x^{-1})$ is finite and $xg^{-1}x^{-1}$ is hyperbolic, we may assume that $\sigma(x^{-1},xg^{-n}x^{-1}(v))\in F$. Then
	\begin{align*}
	\lambda_{xgx^{-1}}(v) &= \sigma(xg^nx^{-1},xg^{-n}x^{-1}(v))F\\&
	= \sigma(x,x^{-1}(v))\sigma(g^n,g^{-n}x^{-1}(v))\sigma(x^{-1},xg^{-n}x^{-1}(v))F\\
	& = \sigma(x,x^{-1}(v))\lambda_g(x^{-1}(v)).
	\end{align*}
	Thus $\lambda_g(v) = \lambda_{xgx^{-1}}(v)$ if and only if $\sigma(x,x^{-1}(v))\lambda_g(x^{-1}(v)) = \lambda_g(v)$. This shows our second claim. 
	
	For the first, note that $x\in G(F,F')_{\partial g}$ if and only if $x^{-1}\in G(F,F')_{\partial g}$. The previous argument shows that this is equivalent to $\sigma(x^{-1},x(v))\lambda_g(x(v)) = \lambda_g(v)$
	for all $v\in V(T)$. Rearranging gives the first claim. 
\end{proof}
For $g\in G(F,F')$ moving towards infinity, we let \[G(F,F')_{\partial g}(1) := \{x\in G(F,F')_{\partial g}\mid s(x) = 1\}\]
and $G(F,F')_{+\partial g} := G(F,F')_{\partial g}(1)\cup \partial g$. 
\begin{lemma}\label{lem:uni_and_asymp}
	Suppose $F$ is $2$-transitive and $g\in G(F,F')$ is hyperbolic. Then: 
	\begin{enumerate}[label = (\roman{*})]
		\item $G(F,F')_{\partial g}(1)$ is a subgroup of $G(F,F')$;
		\item\label{itm:stab_dir_2} If $x\in G(F,F')_{\partial g}(1)$, then $xg\asymp g$ and $s(xg) = s(g)$.
		\item\label{itm:stab_dir_3} If $x\in G(F,F')_{\partial g}$ is hyperbolic with $\omega_+(x) = \omega_+(g)$, we have $x\asymp g$.
	\end{enumerate}
\end{lemma}
\begin{proof}
	To see $G(F,F')_{\partial g}(1)$ is a subgroup, it suffices to show that if $x,y\in G(F,F')_{\partial g}(1)$, then $s(xy) = s(x^{-1}) = 1$. If $x\in G(F,F')_{\partial g}(1)$, then  $x\omega_+(g) = \omega_+(xgx^{-1}) = \omega_+(g)$. Also, since $F$ is $2$-transitive, Corollary \ref{cor:classification_of_uniscalar} shows that $x$ is elliptic. We must have $x^{-1}$ elliptic and so $s(x^{-1}) = 1$ by Proposition \ref{prop:elliptic_scale}. Thus, $G(F,F')_{\partial g}(1)$ is closed under taking inverses. If $x,y\in G(F,F')_{\partial g}(1)$, then $x,y$ both elliptic and therefore must eventually fix any infinite ray with endpoint $\omega_+(g)$. This shows $xy$ is elliptic and hence uniscalar by Proposition \ref{prop:elliptic_scale}. 
	
	To show \ref{itm:stab_dir_3} suppose $x\in G(F,F')_{\partial_g}$ is hyperbolic with $\omega_+(x) = \omega_{+}(g)$ and choose $v\in V(T)$. We show that $\lambda_x(v) = \lambda_g(v)$ which implies $x\asymp g$ by Proposition \ref{prop:classification_in_terms_of_sing_traj}. Note that $x$ is translation compatible with $g$. Lemma \ref{lem:random_natural} and Lemma \ref{lem:basic_asym_sing_prop} gives $k\in\bN$ such that $\lambda_g(x^{-k}(v)) = F$ and $\sigma(x^k, x^{-k}(v)) = \lambda_{x}(v)$. Multiple applications of Lemma \ref{lem:stab_comp} and Lemma \ref{lem:basic_asym_sing_prop} give
	\begin{align*}
	\lambda_g(v) &= \sigma(x,x^{-1}(v))\lambda_{g}(x^{-1}(v))\\
	& = \sigma(x,x^{-1}(v))\sigma(x, x^{-2}(v))\lambda_{g}(x^{-2}(v))\\
	& = \sigma(x^2,x^{-2}(v))\lambda_{g}(x^{-2}(v))\\
	&\hspace{0.2cm}\vdots\\
	& = \sigma(x^k,x^{-k}(v))\lambda_{g}(x^{-k}(v))\\
	& = \lambda_{x}(v).
	\end{align*}
	
	For \ref{itm:stab_dir_2}, suppose $x\in G(F,F')_{\partial g}(1)$. Then $xg\in G(F,F')_{\partial g}$ and $\omega_+(xg) = \omega_{+}(g)$ as $x(\omega_+(g)) = \omega_+(g)$. Hence, $xg\asymp g$ by the previous paragraph. Since $l(g) = l(xg)$, Lemma \ref{lem:length_to_scale} shows $s(g) = s(xg)$.
\end{proof}
\begin{lemma}\label{lem:ends_and_stab_s_mult}
	Suppose $F$ is $2$-transitive and $g\in G(F,F)$ is hyperbolic. Then $G(F,F')_{+\partial g}$ is a scale-multiplicative semigroup which is not open.
\end{lemma}
\begin{proof}
	That $G(F,F')_{+\partial g}$ is a semigroup follows from  Lemma \ref{lem:uni_and_asymp} and Proposition \ref{prop:wk_asymp_semigroup}. That it is scale multiplicative follows from Lemma \ref{lem:uni_and_asymp} and Proposition \ref{prop:asymp_is_s_mult}. To see that it is not open, observe that any open set in $G(F,F')$ contains the translate of a stabiliser in $U(F)$ of a finite set of vertices. Such a set does not fix any boundary points. It follows that since $G(F,F')_{+\partial g}\subset G(F,F')_{\omega_+(g)}$,  the former cannot be open.
\end{proof}

\subsection{Maximality}
\label{ssec:asymp_maximal}
We show that the scale-multiplicative semigroups given in Lemma \ref{lem:ends_and_stab_s_mult} are maximal. As part of our proof, we show that there is no single vertex contained in the axis of every hyperbolic element of $G(F,F')_{+\partial g}$. To prove this result, we construct multiple hyperbolic elements of $G(F,F')_{+\partial g}$. We record details about the singularities of these constructed elements to assist with calculations of the scale function required in the proof of Lemma \ref{lem:same_attracting_end_not_smult}. We require some preparatory lemmas and definitions. Lemma \ref{lem:mult_to_conj_stab} is used to show that no more uniscalar elements can be added.

\begin{lemma}\label{lem:mult_to_conj_stab}
	Suppose $g\in G(F,F')$ is hyperbolic and $x\in G(F,F')$ is elliptic such that $xg^n\asymp g$ for all $n\in\bN$. Then $x\in G(F,F')_{\partial g}(1)$.
\end{lemma}
\begin{proof}
	We must have $\omega_+(xg) = \omega_+(g)$ and hence $x\omega_+(g) = \omega_+(g)$. Choose $v\in V(T)$. We show that $\lambda_g(x(v)) = \sigma(x,v)\lambda_{g}(v)$. This combined with Lemma \ref{lem:stab_comp} completes the result. Since $x$ is elliptic and fixes $\omega_+(g)$, there exists $u\in \axis(g)$ such that $x(u) = u$. Since $g$ is hyperbolic, there exists $N\in\bN$ such that $n\ge N$ implies $d(g^{-n}(v),u)> d(u',u)$ for all $u'\in S(g)\cup S(x)$. Then $\lambda_{g^{n}}(v) = \sigma(g^n,g^{-n}(v))F$ for all $n\ge N$. Since $x$ fixes $u$ and $\omega_+(g)$, our choice of $N$ shows that for $k\ge 2$ and $n\ge N$, both $(xg^{n})^{-k}x(v)$ and $x^{-1}(xg^{n})^{-k}x(v)$ are not vertices in $S(g)\cup S(x)$.
	Hence, 
	\[\lambda_{xg^n}(x(v)) = \sigma(xg^{n}, g^{-n}(v)) = \sigma(x,v)\sigma(g^{n},g^{-n}(v)) = \sigma(x,v)\lambda_{g^n}(v).\]
	But $xg^n\asymp g\asymp g^n$ by assumption. Thus,
	\[\lambda_g(x(v)) = \sigma(x,v)\lambda_{g}(v)\]
	as required. 
\end{proof}

Lemma \ref{lem:proj_onto_int} is a consequence of the fact that $T$ is a tree.
\begin{lemma}\label{lem:proj_onto_int}
	Suppose $g,h\in G(F,F')$ with $v_0,v_1\in \axis(g)\cap \axis(h)$ and $u\in V(T)$. Then:
	\begin{enumerate}[label = (\roman{*})]
		\item If  $v_0<_g \pi_g(u)<_g v_1$, then $\pi_g(u) = \pi_h(u)$.
		\item If $v_0 = \min_{\le_g}(\axis(h)\cap \axis(g))$ and $\pi_g(u)<_g v_0$, then $\pi_h(u) = v_0$.
	\end{enumerate}
\end{lemma}

\begin{lemma}\label{lem:simple_sing_traj_obs}
	Suppose $g\in G(F,F')$ is hyperbolic and $v\in \axis(g)$. Then the set 
	\[A = \{u\in V(T)\mid \pi_{g}(u)\le v\hbox{ and }\lambda_{g}(u)\neq F\}\] is finite.
\end{lemma}
\begin{proof}
	Since $S(g)$ is finite and $g$ is hyperbolic, there exists $N\in\bN$ such that for $n> N$ and $u'\in S(g)$ we have $\pi_{g}g^n(u')>_g v$. It follows from Definition \ref{def:asymp_func} that if $\lambda_g(u)\neq F$, then $u = g^k(u')$ for some $k\in\bN$ and $u'\in S(g)$. Hence, $A$ is contained in the finite set $\bigcup_{0\le n\le N}g^{n}S(g)$.
\end{proof}

Definition \ref{def:2_trans_set} extends the definition of $2$-transitive permutation group to a $2$-transitive subset of a permutation group. According to the definition, if $F$ is $2$-transitive, then any coset or double coset of $F$ in $\Sym(\Omega)$ is $2$-transitive as a set.
\begin{definition}\label{def:2_trans_set}
	A subset $S$ of a permutation group $G$ acting on $X$ is $2$\emph{-transitive} if for any two pairs of distinct elements $(a,b),(c,d)\in X^2$, there exists $g\in S$ such that $(g(a),g(b)) = (c,d)$.
\end{definition}

Recall that for a subtree $A\subset T$, a set of permutations $\{\sigma_v\in F'\mid v\in V(A)\}$ is consistent if $\sigma_v(c(u,v)) = \sigma_u(c(u,v))$ whenever $(u,v)$ forms an edge. Lemma \ref{lem:autExtension} associates automorphisms to consistent sets of permutations.
\begin{lemma}\label{lem:end_semigroup_large}
	Suppose $F$ is $2$-transitive and $g\in G(F,F')$ is hyperbolic with $l(g)> D_g$. Choose $v\in\axis(g)$ such that $u\in S(g)$ implies $\pi_g(u)<_g v$. There exists $g_v\in G(F,F')$ asymptotic to $g$ such that $l(g) = l(g_v)$ and $\min_{\le_g}\pi_g(\axis(g_v)) = v$. Furthermore, if $u\in S(g_v)$, then:
	\begin{enumerate}[label = (\roman{*})]
		\item \label{itm:sing_g_l_1}$\pi_{g}(u)\le_g v$;
		\item \label{itm:sing_g_l_2}$\pi_{g_v}(u)\le_{g_v} v$;
		\item \label{itm:sing_g_l_3}either $\lambda_g(u)\neq F$ or $\lambda_{g}(g_v(u))\neq F$; and
		\item \label{itm:sing_g_l_4}either $d(u,\axis(g))\le D_g$ or $d(g_v(u),\axis(g))\le D_g$.
	\end{enumerate} 
\end{lemma}
\begin{proof}
	Choose a bi-infinite path $P$ such that $u\in \axis(g)\cap P$ if and only if $u\ge_g v$. We define $g_v$ asymptotic to $g$ such that $\axis(g_v) = P$ and $l(g_v) = l(g)$. To do so, choose any $h\in U(F)$ such that $\axis(h) = P$ and $l(h) = l(g)$. Such a $h$ exists via Corollary \ref{cor:2-trans_translations}. Our automorphism $g_v$ will act as $h$ far enough away from $\axis(g)$, as $g$ far enough along $\axis(g)$ and transition between the two in between. This results in $g_v\asymp g$ and so we verify that $\lambda_{g_v}(g_v(u)) = \lambda_{g}(g_v(u))$ at each step of the definition. We are informally extending the definition of $\lambda_{g_v}(u)$ to cases where $g_v$ may only partially defined but $\sigma(g_v^{k}, g_v^{-k}(u))$ is defined for all $k\in\bN_0$. A schematic for $g_v$ is given in Figure \ref{fig:schem}.
	\begin{figure}
		\caption{A schematic for $g_v$.}
		\label{fig:schem}
		\begin{center}
			\begin{overpic}[ scale = 1.05]{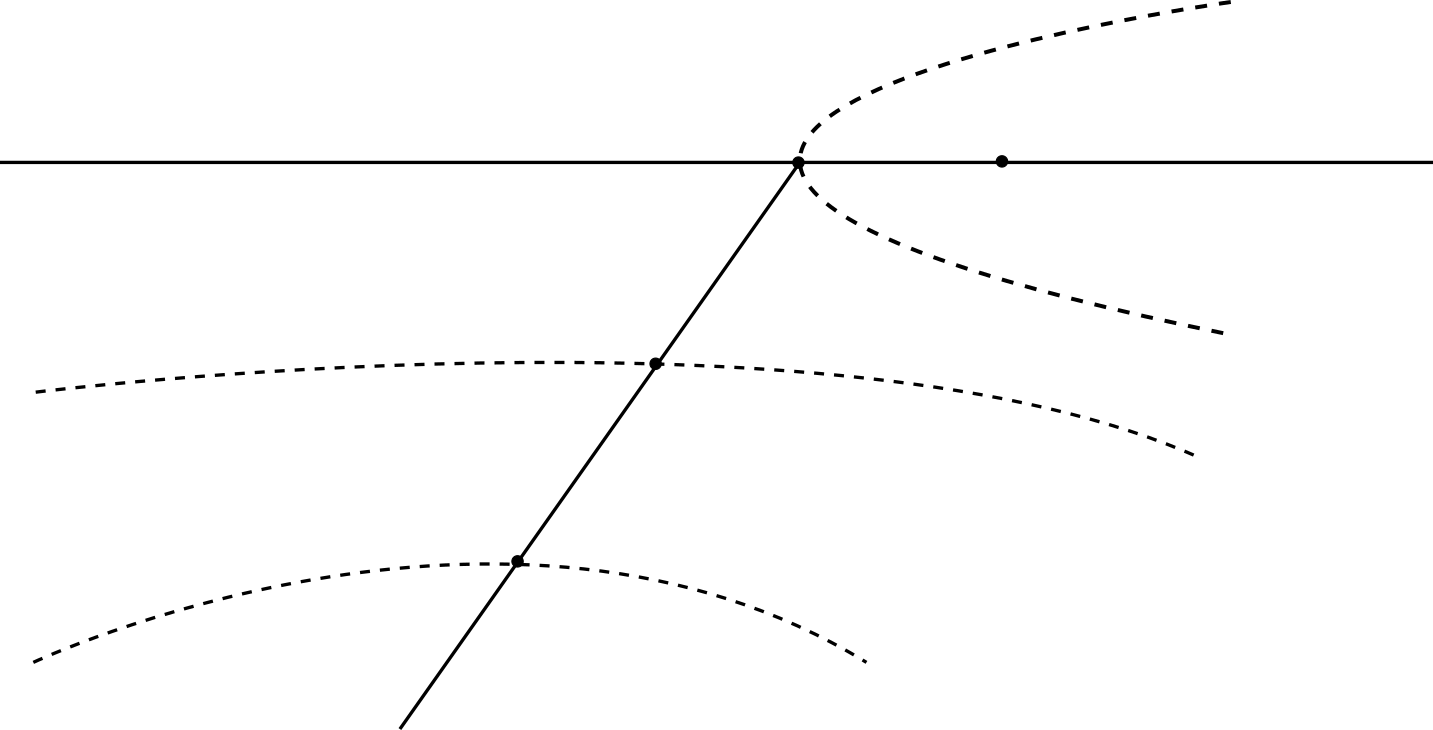}
				\put(51,37){$v$}
				\put(0,42){$\axis(g)$}
				\put(30,0){$\axis(g_v) = \axis(h)$}
				\put(10,5){$g_v(u) = h(u)$}
				\put(10,17){$\sigma(g_v,u) \in \lambda_g(g_v(u))$}
				\put(10,30){$\sigma(g_v,u) \in \lambda_g(g_v(u))\lambda_g(u)^{-1}$}
				\put(70,45){$g_v(u) = g(u)$}
				\put(45,22){$h^{-1}(v)$}
				\put(35,7.5){$h^{-2}(v)$}
				\put(69,36){$g(v)$}
			\end{overpic}
		\end{center}
	\end{figure}
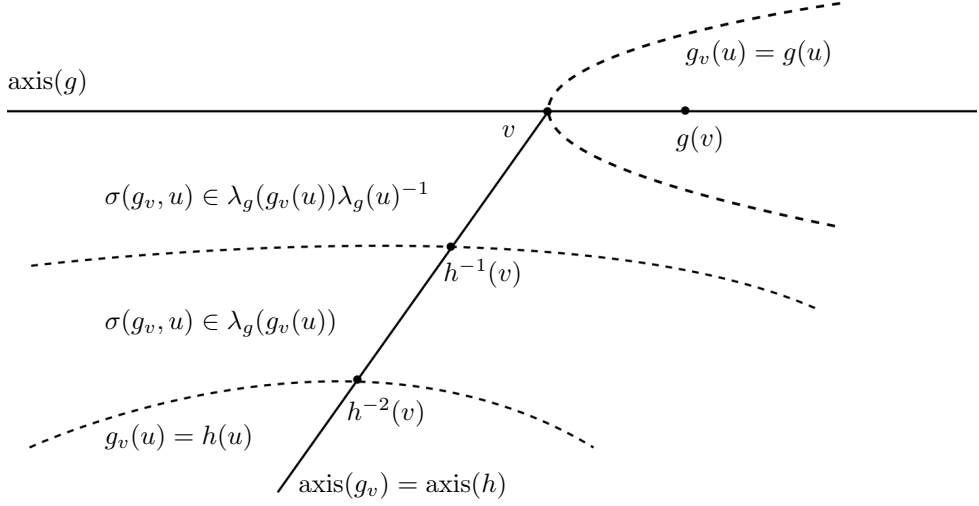 
	For $u\in V(T)$ we define $g_v(u)$ via the following steps:
	\begin{enumerate}[label = {Step \arabic{*}}, itemindent = .7cm, leftmargin = *, itemsep = 0.5cm, labelsep = 0cm]
		\item\label{itm:def_step_1}. If $\pi_{h}(u)\le_{h}h^{-2}(v)$, then set $g_v(u) = h(u)$ and $\sigma(g_v,u) = \sigma(h,u)$. Note that $\pi_{h}h^{-n}(u)\le_{h}h^{-2}(v)$ for all $n\in\bN_0$. In particular, \[\sigma((g_v)^n, g_v^{-n+1}(u)) = \sigma(g_v,u)\sigma(g_v,g_v^{-1}(u)) \cdots \sigma(g_v, g_v^{-n+1}(u)) = \sigma(h^n,h^{-n+1}(u))\]
		for all $n\in\bN_0$. Hence, 
		\[\lambda_{g_v}(g_v(u)) = \lim_{n\to\infty}\sigma(g_v^n, g_v^{-n+1}(u))F = \lim_{n\to\infty}\sigma(h^{n},h^{-n+1}(u))F = F.\]
		Note that
		\[d(g_v(u),\axis(g))\ge d(g_v^{-1}(v),\axis(g))\ge l(h)= l(g) > D_g.\] 
		Lemma \ref{lem:basic_asym_sing_prop}  shows that $\lambda_g(g_v(u)) = F = \lambda_{g_v}(g_v(u))$.
		\item\label{itm:def_step_2}. Suppose $h^{-2}(v)<_{h}\pi_{h}(u)\le_{h}v$. We suppose $h^{-2}(v)<_{h}\pi_{h}(u)\le_{h}h^{-1}(v)$ and define $g_v(u)$, $g_v^2(u)$, $\sigma(g_v,u)$ and $\sigma(g_v,g_v(u))$. This suffices since if $u\in V(T)$ with $h^{-1}(v)<_{h}\pi_{h}(u)\le_{h}v$, then $u = g_v(u')$ where $h^{-2}(v)<_{h}\pi_{h}(u')\le_{h}h^{-1}(v)$. We proceed by induction on $d(u,\axis(h))$. If $d(u,\axis(h)) = 0$, set $g_v(u) = h(u)$ and $g_v^2(u) = h^2(u)$. Since $\lambda_g(g_v(u))$ and $ \lambda_g(g_v^2(u))\lambda_g(g_v(u))^{-1}$ are $2$-transitive sets, we may consistently choose $\sigma(g_v,u)\in \lambda_g(g_v(u))$ and ${\sigma(g_v,g_v(u))\in \lambda_g(g_v^2(u))\lambda_g(g_v(u))^{-1}}$.
		
		Suppose we have consistently defined $\sigma(g_v,u)$ and $g_v(u)$ if $d(u,\axis(h))\le n$. Suppose now that $d(u,\axis(h)) = n+1$. Choose $u'$ distance one away from $u$ on the path between $u$ and $\axis(h)$. Then $g_v(u)$  and $g_v^2(u)$ are determined by $\sigma(g_v,u')$ and $\sigma(g_v,g_v(u'))$. Since $F$ is $2$-transitive, we may make a consistent choice of $\sigma(g_v,u)\in \lambda_g(g_v(u))$ and $\sigma(g_v,g_v(u))\in \lambda_g(g_v^2(u))\lambda_g(g_v(u))^{-1}$. This completes our induction. 
		
		Since $\pi_{h}g_v^{-1}(u) \le_h h^{-2}(v)$, \ref{itm:def_step_1} that shows that $\lambda_{g_v}(u) = F$. Applying Lemma \ref{lem:basic_asym_sing_prop} we have 
		\[\lambda_{g_v}(g_v(u)) = \sigma(g_v,u)\lambda_{g_v}(u) = \lambda_g(g_v(u)),\]
		and
		\[\lambda_{g_v}(g_v^2(u)) = \sigma(g_v,g_v(u))\lambda_{g_v}(g_v(u)) = \lambda_g(g_v^2(u))\lambda_g(g_v(u))^{-1}\lambda_{g}(g_v(u)) = \lambda_g(g_v^2(u)).\]
		\item\label{itm:def_step_3}. Finally, for $\pi_{h}(u)>_{h}v$ set $g_v(u) = g(u)$. This defines $g_v(u)$ as an automorphism since $\pi_g(u) = \pi_{h}(u)$, see Lemma \ref{lem:proj_onto_int}, and $l(h) = l(g)$.  There exists $k> 0$ such that \[h^{-1}(v)<_{h}\pi_{g_v}g_{v}^{-k}(u)\le_{h}v.\]
		We have $\pi_{h}g_v^{1-k}(u) >_{h} v$ and so 
		\[u = g_v^{k-1}g_v^{1-k}(u) = g^{k-1}g_v^{1-k}(u).\]
		Thus, $g_v^{1-k}(u) = g^{1-k}(u)$.
		We saw in \ref{itm:def_step_2} that $\lambda_{g_v}(g_v^{-k}(u)) = \lambda_{g}(g_{v}^{-k}(u))$. Lemma \ref{lem:basic_asym_sing_prop} and our definition of $g_v$ show that
		\begin{align*}
		\lambda_{g_v}(g_v(u)) &= \sigma(g_v^{k - 1},g_v^{-k+1}(u)) \lambda_{g_v}(g_v^{-k+1}(u))\\
		& = \sigma(g^{k - 1},{g}^{-k+1}(u)) \lambda_{g}(g^{-k+1}(u))\\
		& = \lambda_g(g_v(u)).
		\end{align*}
		This completes our definition of $g_v$.
	\end{enumerate}
	\vspace{0.2cm}
	
	That $\min_{\le_g}\pi_g(\axis(g_v)) = v$ follows since $\axis(g_v) = \axis(h) = P$. We now prove the statements concerning $u\in S(g_v)$.
	
	Since $v\in \axis(g)\cap \axis(h)\cap \axis(g_v)$ and $\omega_+(g) = \omega_+(g_v) = \omega_+(h)$, if $\pi_{g}(u)>_g v$ or $\pi_{g_v}(u)> v$, then $\pi_{h}(u)>_{h} v$.  \ref{itm:def_step_3} of our construction shows that $\sigma(g_v,u) = \sigma(g,u)$. Our assumptions on $v$ imply that $\sigma(g_v,u)\in F$. This gives \ref{itm:sing_g_l_1} and \ref{itm:sing_g_l_2}. 
	
	If $u\in S(g_v)$, then  $g_v^{-2}(v)<_{g_v}\pi_{g_v}(u)\le_{g_v}v$, since otherwise ${\sigma(g_v,u) = \sigma(h,u)\in F}$ or $\sigma(g_v,u) = \sigma(g,u)\in F$, see \ref{itm:def_step_1} and \ref{itm:def_step_3}. \ref{itm:def_step_2} of our construction shows that $\sigma(g_v,u) \in \lambda_{g}(g_v(u))$ or $\sigma(g_v,u)\in\lambda_g(g_v(u))\lambda_g(u)^{-1}$. We must have either ${\lambda_g(u) \neq F}$ or $\lambda_{g}(g_v(u))\neq F$, hence \ref{itm:sing_g_l_3} holds. But, if $\lambda_g(u')\neq F$ for some $u'\in V(T)$, then $d(u',\axis(g))\le D_g$ by Lemma \ref{lem:basic_asym_sing_prop}. This gives \ref{itm:sing_g_l_4}.
	
	It remains to show that $g_v\in G(F,F')$ as then $g_v\asymp g$ by construction. It suffices to show that $S(g_v)$ is finite. This follows since the number of vertices satisfying \ref{itm:sing_g_l_1} and \ref{itm:sing_g_l_3} is finite by Lemma \ref{lem:simple_sing_traj_obs}. 
\end{proof}
\begin{lemma}\label{lem:same_attracting_end_not_smult}
	Suppose $F$ is $2$-transitive and $g,h\in G(F,F')$ are hyperbolic such that $\omega_{+}(h) = \omega_{+}(g)$. If the semigroup generated by $h$ and $\partial g$ is scale-multiplicative, then $h\asymp g$.
\end{lemma}
\begin{proof}
	We show the contrapositive. Suppose $h\not\asymp g$. Since $g\asymp g^n$ for all $n\in\bN$, Lemma \ref{lem:depth_power} allows us to assume $l(g)> D_g$. Since $h\not\asymp g$, Proposition \ref{prop:classification_in_terms_of_sing_traj} gives $v\in V(T)$ such that $\lambda_g(v)\neq \lambda_h(v)$. Replacing $h$ with $h^n$ for some $n\in\bN$ if necessary, Lemma \ref{lem:basic_asym_sing_prop} allows us to assume that $\lambda_h(v) = \sigma(h,h^{-1}(v))F$ and $\lambda_g(h^{-1}(v)) = F$. 
	
	Since $\omega_+(h) = \omega_+(g)$, we may label 
	\[\axis(g) = (\ldots, v_{-1},v_0,v_{1},\ldots)\]
	such that:
	\begin{enumerate}[label = (\roman{*})]
		\item $v_i\le_g v_j$ if and only if $i\le j$;
		\item \label{itm:axis_assumpt}$v_k\in \axis(h)$ for all $k\ge 0$; and
		\item if $u\in S(h)\cup hS(g)\cup S(g)$ or $u = v$, then $\pi_{g}(v)\le_g v_0$ and $\pi_h(v)\le_h v_0$.\label{itm:assumpt_1}
	\end{enumerate}
	Choose $l > D_g+\log_{|\Omega| - 1}(s(g)s(h))+l(h)$. We show that $s(g_lh)\neq s(g_l)s(h)$ where $g_l:= g_{v_l}$ is given in Lemma \ref{lem:end_semigroup_large}.
	
	It follows from \cite[Lemma 4.7]{Baumgartner15} that $v_n\in\axis(g_lh)$ if and only if $n \ge l - l(h)$. Indeed, $\axis(g_lh)$ is of the form
	\begin{equation}\label{eq:axis}
	(\ldots, (g_lh)^{-n}(v_{l - l(h)}),\ldots,(g_lh)^{-n}(v_{l}),\ldots, (g_lh)^{-1}(v_{l - l(h)}),\ldots, v_{l - l(h)},\ldots, v_l,v_{l+1},\ldots).
	\end{equation}
	We investigate the singularities of $g_lh$ in order bound $s(g_lh)$ from below. We use fact that $S(g_lh)\subset S(h)\cup h^{-1}S(g_l)$ to reduce the problem to $S(h)$ and $S(g_l)$. Restrictions on $S(h)$ are given by Claim \ref{claim:dist} \ref{itm:dist_3} and \ref{itm:assumpt_1}. Restrictions on $S(g_l)$ are given by Claim \ref{claim:dist} \ref{itm:dist_1}, Claim \ref{claim:dist} \ref{itm:dist_2} and Lemma \ref{lem:end_semigroup_large} \ref{itm:sing_g_l_4}.
	
	\begin{claim}\label{claim:dist}
		Suppose $u\in V(T)$ such that $\pi_{g}(u) = g_l(v_l)$. Then for all $n\in\bN$
		\begin{enumerate}[label = (\alph{*})]
			\item\label{itm:dist_1} $d(g_l^{-1}(g_lh)^{-n}(u),\axis(g))> D_g$;
			\item \label{itm:dist_2}$d((g_lh)^{-n - 1},\axis(g))> D_g$; and
			\item \label{itm:dist_3}$\pi_{g}(g_lh)^{-n - 1}(u) = v_{l - l(h)}$.
		\end{enumerate}
	\end{claim}
	\begin{proof}
		Observe that $\pi_{g_lh}(u) = \pi_g(u) = \pi_h(u)$ by Lemma \ref{lem:proj_onto_int}. For \ref{itm:dist_1}, note that since \[\pi_{g_lh}(g_lh)^{-n}(u) \le_{g_lh} (g_lh)^{-1}g_l(v_{l}) = v_{l - l(h)}<_{g_lh} v_l,\] Lemma \ref{lem:proj_onto_int} shows that $\pi_{g_l}(g_lh)^{-n}(u) = v_l$. Thus, $\pi_{g_l}g_l^{-1}(g_lh)^{-n}(u) = g_l^{-1}(v_l)<_{g_l} v_l$. Lemma \ref{lem:proj_onto_int} shows that $\pi_g g_l^{-1}(g_lh)^{-n}(u) = v_l$. Since $\pi_{g_l}(g_lh)^{-n}(u) = v_l$, we have $d(g_l^{-1}(g_lh)^{-n}(u),v_l) = d((g_lh)^{-n}(u),v_l)+ l(g_l)$. By construction, $l(g_l) = l(g)$ which is assumed to be strictly larger than $D_g$. Hence,
		\[d(g_l^{-1}(g_lh)^{-n}(u), \axis(g)) = d(g_l^{-1}(g_lh)^{-n}(u), \pi_gg_l^{-1}(g_lh)^{-n}(u))= d(g_l^{-1}(g_lh)^{-n}(u),v_l)> D_g.\]
		
		Noting that  
		\[\pi_{g_lh}(g_lh)^{-n - 1}(u) = (g_lh)^{-n - 1}g_l(v_l) <_{g_lh} v_{l - l(h)},\]
		\ref{itm:dist_3} follows from Lemma \ref{lem:proj_onto_int}.
		
		Finally, for \ref{itm:dist_2} observe that $\pi_{g_lh}(g_lh)^{-1}(u) = (g_lh)^{-1}g_l(v_l) = v_{l - l(h)}$. Since $g_lh$ acts by translation along its axis 
		\[d((g_lh)^{-n-1}(u), v_{l - l(h)}) = nl(g_lh) + d((g_lh)^{-1}(u), v_{l - l(h)}).\]
		It follows from \cite[Lemma 4.7]{Baumgartner15} that $nl(g_lh) = nl(g_l)+nl(h)$. This is strictly larger that $D_g$ since $l(g_l) = l(g)$. But \ref{itm:dist_3} shows than $\pi_{g} (g_lh)^{-n - 1}(u) = v_{l - l(h)}$. Thus,
		\[d((g_lh)^{-n - 1}(u), \axis(g)) = d((g_lh)^{-n -1}(u), v_{l - l(h)}) > D_g.\qedhere\]
	\end{proof}
	
	\begin{claim}\label{claim:large_powers_not_sing}
		Suppose $u\in V(T)$ with $\pi_{g_lh}(u) = g_l(v_l)$. If $n> 2$, then $(g_lh)^{-n}(u)\not\in S(g_lh)$.
	\end{claim}
	
	\begin{proof}
		We show that $(g_lh)^{-n}(u)\not\in S(h)$ and $h(g_lh)^{-n}(u)\not\in S(g_l)$. Combined, this shows $(g_lh)^{-n}\not\in S(h)\cup h^{-1}S(g_l)$. 
		
		Claim \ref{claim:dist} \ref{itm:dist_3} and \ref{itm:assumpt_1} shows $(g_lh)^{-n}(u)\not\in S(h)$. Now $h(g_lh)^{-n}(u) = g_l^{-1}(g_lh)^{-n+1}(u)$. Since $n> 2$,  applying Claim \ref{claim:dist} \ref{itm:dist_1} to $(g_lh)^{-n+1}(u)$ and Claim \ref{claim:dist} \ref{itm:dist_2} to $g_l^{-1}(g_lh)^{-n+1}(u)$ shows that 
		\[d((g_lh)^{-n+1}(u),\axis(g))> D_g\hbox{ and }d(g_l^{-1}(g_lh)^{-n+1}(u),\axis(g))> D_g.\]
		Lemma \ref{lem:end_semigroup_large} \ref{itm:sing_g_l_4} shows that $g_l^{-1}(g_lh)^{-n+1}(u)\not\in S(g_l)$.
	\end{proof}
	\begin{claim}\label{claim:restricted_sing}
		Suppose $u\in V(T)$ such that:
		\begin{enumerate}[ label = (\Roman{*})]
			\item\label{itm:sing_ass_1}  $\pi_{g_lh}(u) = g_l(v_l)$; 
			\item\label{itm:sing_ass_2} $\pi_g(g_l^{-1}(u))>_gv_{D_g+ l(h)}$; and
			\item\label{itm:sing_ass_3}$d(u,\axis(g_lh)) \ge l$. 
		\end{enumerate}
		Then $(g_lh)^{-n}(u)\not\in S(g_lh)$ for all $n\in\bN$.
	\end{claim}
	\begin{proof}
		It suffices to test when $n = 1$ and $n = 2$ since Claim \ref{claim:large_powers_not_sing} then completes the proof. Observe that Lemma \ref{lem:proj_onto_int}, \ref{itm:sing_ass_1} and \ref{itm:sing_ass_2} show that
		\[\pi_{g_lh}(u) = \pi_{g}(u) = \pi_h(u) = \pi_{g_l}(u) = g_l(v_l),\]
		and
		\[\pi_g(g_l^{-1}(u)) = \pi_h(g_l^{-1}(u))>_gv_{D_g+l(h)}.\]
		In particular, 
		\[d(u,\axis(g)) = d(u,\axis(h)) = d(u,\axis(g_lh))\ge l.\]
		
		Suppose $n = 1$. Since $\pi_{h}(g_l^{-1}(u))>_h v_{D_g+l(h)}$, we have $\pi_h h^{-1}g_l^{-1}(u)>_h v_{D_g}>_h v_0$. It follows from from \ref{itm:assumpt_1} that $(g_lh)^{-1}(u)\not\in S(h)$. It suffices to show $g_l^{-1}(u)\not\in S(g_l)$. Observe that 
		\[d(g_l^{-1}(u),v_l) = d(u,g(v_l)) = d(u,\axis(g)) = d(u,\axis(g_lh))\ge l>D_g.\]
		But $\pi_{g}(g_l^{-1}(u))>_g v_{D_g}$. Thus, \[d(g_l^{-1}(u),\axis(g)) > d(g_l^{-1}(u),v_l) - d(v_l,v_{D_g}) \ge l - l +D_g = D_g.\]
		We have already seen that $d(u,\axis(g))\ge l > D_g$. Hence, $g_l^{-1}(u)\not\in S(g_l)$ by Lemma \ref{lem:end_semigroup_large}  \ref{itm:sing_g_l_4}.
		
		We use a similar argument for $n = 2$. That $(g_lh)^{-2}(u)\not\in S(h)$ follows from Claim \ref{claim:dist} \ref{itm:dist_3} and \ref{itm:assumpt_1}. We show that $g_l^{-1}(g_lh)^{-1}(u)\not\in S(g_l)$ via Lemma \ref{lem:end_semigroup_large} \ref{itm:sing_g_l_4}. Lemma \ref{lem:proj_onto_int} shows that $\pi_g(g_lh)^{-1}(u) = \pi_h(g_lh)^{-1}(u)$ and $\pi_h(g_l^{-1}(u)) = \pi_g(g_l^{-1}(u))$. In particular, 
		\begin{align*}
		d((g_lh)^{-1}(u),\axis(g)) &= d((g_lh)^{-1}(u),\axis(h))= d(g_l^{-1}(u),\axis(h))\\ & =  d(g_l^{-1}(u),\axis(g)) > D_g.
		\end{align*}
		To see that $d(g_l^{-1}(g_lh)^{-1}(u),\axis(g))> D_g$, note that 
		\[\pi_{g_l} (g_lh)^{-1}(u) = v_{l}= \pi_g g_l^{-1}(g_lh)^{-1}(u)\]
		by Lemma \ref{lem:proj_onto_int}. Hence, \begin{align*}
		d(g_l^{-1}(g_lh)^{-1}(u),\axis(g))&=d(g_l^{-1}(g_lh)^{-1}(u),v_l)\\ &= l(g_l) + d((g_lh)^{-1}(u),v_l)\ge l(g)> D_g.\qedhere
		\end{align*}
	\end{proof}
	\begin{claim}\label{claim:one_sing}
		We have $(g_lh)^{-1}g_l(v) = h^{-1}(v)\in S(g_lh)$ but $(g_lh)^{-2}g_l(v)\not\in S(g_lh)$. In particular, $\lambda_{g_lh}(g_l(v))\neq F$.
	\end{claim}
	\begin{proof}
		By assumptions on $v$ and $h$, we have $\lambda_g(v)\neq \lambda_h(v) = \sigma(h,h^{-1}(v))F$. Lemma \ref{lem:proj_onto_int} and \ref{itm:assumpt_1} shows $\pi_{g_l}(v) = v_l$ and so $\pi_{g_l}g_l(v) = g_l(v_l)$. Lemma \ref{lem:proj_onto_int} shows  $\pi_{g}g_l(v) = g_l(v_l)$. Hence
		\[d(g_l(v),\axis(g)) = d(g_l(v),\axis(g_l)) = d(v,v_l).\]
		But $d(v,v_l)> l > D_g$. Lemma \ref{lem:basic_asym_sing_prop} shows that $\lambda_g(g_l(v)) = F$. Construction of $g_l$ gives $\sigma(g_l,v) \in \lambda_{g}(g_l(v))\lambda_{g}(v)^{-1} =\lambda_g(v)^{-1}$, see \ref{itm:def_step_2} of the proof of Lemma \ref{lem:end_semigroup_large}. Thus, 
		\[\sigma(g_lh,h^{-1}(v)) = \sigma(g_l,v)\sigma(h,h^{-1}(v)) \in  \lambda_{g}(v)^{-1}\lambda_{h}(v).\]
		Since $\lambda_g(v)\neq \lambda_h(v)$, the set $\lambda_{g}(v)^{-1}\lambda_{h}(v)$ is of the form $FaF$ where $a\in F'\setminus F$. This intersects trivially with $F$. Hence, $h^{-1}(v)\in S(g_lh)$.  
		
		We now show $(g_lh)^{-2}g_l(v)\not\in S(h)$ and $h(g_lh)^{-2}g_l(v) = (hg_l)^{-1}(v)\not\in S(g_l)$. This shows $(g_lh)^{-2}g_l(v)\not\in S(h)\cup h^{-1}S(g_l)$. That $(g_lh)^{-2}g_l(v)\not\in S(h)$ follows from Claim \ref{claim:dist} \ref{itm:dist_3} and \ref{itm:assumpt_1}. To show $g_l^{-1}h^{-1}(v)\not\in S(g_l)$, it suffices to show that $\lambda_g(g_l^{-1}h^{-1}(v)) = F$ by Lemma \ref{lem:end_semigroup_large} \ref{itm:sing_g_l_3} since it is assumed that $h$ is such that $\lambda_g(h^{-1}(v))= F$. Note that $g_l^{-1}h^{-1}(v) = g_l^{-1}(g_lh)^{-1}g_l(v)$. Also $\pi_{g_l}(v) = v_l$ and so $\pi_{g_l}g_l(v) = \pi_gg_l(v) = g_l(v)$ by Lemma \ref{lem:proj_onto_int}. It follows from Claim \ref{claim:dist} \ref{itm:dist_1} that $d(g_l^{-1}(g_lh)^{-1}g_l(v), \axis(g))> D_g$. Lemma \ref{lem:basic_asym_sing_prop} shows that 
		\[\lambda_g(g_l^{-1}h^{-1}(v)) = \lambda_g(g_l^{-1}(g_lh)^{-1}g_l(v)) = F.\]
		
		Finally, Claim \ref{claim:large_powers_not_sing} and Lemma \ref{lem:basic_asym_sing_prop} show that $\lambda_{g_lh}(g_l(v))=\sigma(g_lh,h^{-1}(v))F\neq F$.
	\end{proof}
	\begin{claim}\label{claim:no_further_sing}
		Suppose $u\in V(T)$ with $\pi_{g_lh}(u)>_{g_lh}v_l$. Then $u\not\in S(g_lh)$.
	\end{claim}
	\begin{proof}
		If $\pi_{g_lh}(u) >_{g_lh}v_l$, then Lemma \ref{lem:proj_onto_int} shows that $\pi_h(u)>_h v_l$. It follows that $u\not\in S(h)$ by \ref{itm:assumpt_1}. It suffices to show that $h(u)\not\in S(g_l)$. Since \[(v_l,v_{l+1},\ldots)\subset \axis(h)\cap \axis(g_lh)\cap \axis(g),\] 
		Lemma \ref{lem:proj_onto_int} shows $\pi_{g_lh}h(u) = \pi_{g_l}h(u) = \pi_hh(u)>_{g_l}v_l$. Lemma \ref{lem:end_semigroup_large} \ref{itm:sing_g_l_2} shows that $h(u) \not\in S(g_l)$.
	\end{proof}
	Let $\cT$ the minimal complete subtree of $T$ satisfying:
	\begin{enumerate}[label = (\Alph*)]
		\item $v_{l+1}\in\Int(\cT)$ and $g_lh(v_{l+1})\in \cT$; and
		\item If $u\in V(T)$ with $\pi_{g_lh}(u)\in\Int(\cT)$ and $d(u,\axis(g))< D_{g_lh}$, then $u\in \cT$.\label{itm:minimal_subtree_prop_2}
	\end{enumerate} 
	Using the description of $\axis(g_lh)$ given in equation \eqref{eq:axis} and the fact that $v_{l - 1}\not\in\axis(g_l)$, we see that the path 
	\[P = (g_l(v_l), g_l(v_{l-1}),\ldots, g_l(v_{l - D_{g_lh}}))\]
	is contained in $\cT$ and meets $\axis(g_lh)$ at the vertex $g_l(v_l)$. Claim \ref{claim:one_sing} and \ref{itm:minimal_subtree_prop_2} show that $g_l(v)\in \Int(\cT)$. It follows from \ref{itm:assumpt_1} and Claim \ref{claim:one_sing} that $D_{g_lh}>l$ and so $g_l(v_0)\in P$. Since $l> D_g + l(h)>0$, we then have $g_l(v_{D_g+l(h)})\in P$.
	
	\begin{claim}\label{claim:aut_fixed_point}
		Suppose $\varphi\in M_{g_lh,\cT}$. Then $\varphi g_l(v_{D_g+l(h)}) = g_l(v_{D_g+l(h)})$.
	\end{claim}
	\begin{proof}
		We proceed by showing the contrapositive. Suppose $\varphi\in \Aut(\cT)$ such that $\varphi$ fixes $\axis(g_lh)\cap \cT$ but ${\varphi g_l(v_{D_g+l(h)})\neq g_l(v_{D_g+l(h)})}$. It follows from \ref{itm:assumpt_1} that the path from $g_l(v)$ to $g_l(v_l)$ passes through $g_l(v_0)$. In particular, it passes through $g_l(v_{D_g+l(h)})$. Since $\varphi g_l(v_{D_g+l(h)})\neq g_l(v_{D_g+l(h)})$, the path from $\varphi g_l(v)$ to $g_l(v_l)$ does not pass through $g_l(v_{D_g+l(h)})$. We must have $\pi_g(g_l^{-1}\varphi g_l(v))>_gv_{D_g+l(h)}$.
		Since $\varphi$ fixes $\axis(g_lh)\cap \cT$, we have
		\[\pi_{g_lh}(\varphi g_l(v)) = \pi_{g_lh}(g_l(v)) = g_l(v)\]
		and 
		\[d(\varphi g_l(v),\axis(g_lh)) = d(g_l(v),\axis(g_lh))= d(g_l(v),g_l(v_{l})) \ge l.\]
		Applying Claim \ref{claim:restricted_sing} to $\varphi g_l(v)$ shows that $\lambda_{g_lh}(\varphi g_l(v)) = F$. Claim \ref{claim:one_sing} shows that 
		\[F\cap \lambda_{g_lh}(\varphi g_l(v))\lambda_{g_lh}(g_l(v))^{-1} = F\cap \lambda_{g_lh}(g_l(v))^{-1} = \varnothing.\]
		Hence, $\sigma(\varphi,g_l(v))\not\in F\cap \lambda_{g_lh}(\varphi g_l(v))\lambda_{g_lh}(g_l(v))^{-1}$.
		The second part of Lemma \ref{lem:scale_flex} shows that $\varphi\not\in M_{g_lh}$.
	\end{proof}
	Claim \ref{claim:aut_fixed_point} and repeated applications of the Orbit-Stabiliser Theorem show that 
	\[|M_{g_lh,\cT_1}|\le (|\Omega| -1)^{|\Int(\cT)| - d(g(v_{l}),g(v_{D_g+l(h)}))}.\]
	Now $d(g(v_{l}),g(v_{D_g+l(h)})) = l - D_g-l(h)$. But  $l > D_g+ \log_{|\Omega| - 1}(s(g)s(h)) + l(h)$ by assumption. Hence, $d(g(v_{l}),g(v_{D_g}))>  \log_{|\Omega| - 1}(s(g)s(h))$. Lemmas \ref{lem:scale_flex} and \ref{lem:length_to_scale} show that $s(g_lh)> s(g)s(h) = s(g_l)s(h)$ as required. 
\end{proof}

\begin{theorem}\label{thm:maximal}
	For $F$ $2$-transitive and $g\in G(F,F')$  hyperbolic, $G(F,F')_{+\partial g}$ is a maximal scale-multiplicative semigroup which is not open.
\end{theorem}

\begin{proof}
	It suffices to show maximality by Lemma \ref{lem:ends_and_stab_s_mult}.	Suppose $h\in G(F,F')\setminus G(F,F')_{+\partial g}$ is hyperbolic. Then $s(h)> 1$ by Corollary \ref{cor:classification_of_uniscalar}. For various cases of $h$, we find $h'\in G(F,F')_{+\partial g}$ and $n\in\bN$ such that $s(h^nh')\neq s(h^n)s(h')$, thus any scale-multiplicative semigroup containing $G(F,F')_{+\partial g}$ cannot contain $h$. We immediately discount the case when $\omega_{-}(h) = \omega_{+}(g)$, since then there exists $n,m\in\bN$ such that $h^{n}g^m$ is elliptic and therefore uniscalar by Proposition \ref{prop:elliptic_scale}. Similarly, we also may also suppose $\omega_{-}(g)\neq \omega_{+}(h)$. Lemma \ref{lem:same_attracting_end_not_smult} discounts the case when $\omega_+(h) = \omega_+(g)$.
	
	Our final case is if the sets of ends fixed by $g$ and $h$ respectively are disjoint. Equivalently, $\axis(h)\cap \axis(g)$ is finite. Therefore, $\pi_g(\axis(h))$ is finite.
	Replacing $g$ with $g^n$ if necessary, we assume that $l(g)> D_g$. 
	Choose $v\in \axis(g)$ such that $d(v,\axis(h))>\log_{|\Omega|-1}(s(g)s(h))$ and $v >_g \pi_g(u)$ for all ${u\in \axis(h)}$. Define $g_v$ as in Lemma \ref{lem:end_semigroup_large}. Then \[d(\axis(g_v),\axis(h)) > \log_{|\Omega|-1}(s(g)s(h)).\]
	It follows from \cite[Lemma 4.6 Part 1]{Baumgartner15} that $l(g_vh)> \log_{|\Omega| - 1}(s(g)s(h))$. Proposition \ref{prop:U(F)_small_scale} and Lemma \ref{lem:scale_weak_asyp} show that $s(g_vh)> s(g)s(h)$. But $l(g_v) = l(g)$, hence $s(g) = s(g_v)$ by Lemma \ref{lem:length_to_scale}. It follows that $s(g_vh)> s(g_v)s(h)$ as required.
	
	We have shown that if $S$ is a scale multiplicative semigroup containing $G(F,F')_{+\partial g}$ and $h\in S$ is hyperbolic, then $h\in G(F,F')_{+\partial g}$. Suppose $x\in S$ is elliptic. We are left to show that $x\in G(F,F')_{+\partial g}$. For all $h\in G(F,F')_{+\partial g}$ hyperbolic, we must have $xh$ hyperbolic as otherwise $s(xh) = 1 < s(h)$ by Proposition \ref{prop:elliptic_scale} and Corollary \ref{cor:classification_of_uniscalar}. We must have $xh\asymp h$ as $xh\in S$. In particular, $xg^n\asymp g^n$ for all $n\in\bN$. Lemma \ref{lem:mult_to_conj_stab} shows that $x\in G(F,F')_{\partial g}(1)\le G(F,F')_{+\partial g}$ as required. 
\end{proof}
	\bibliography{../../../../BibFiles/PhDThesisRef}

\begin{thebibliography}{{Byw}19}

\bibitem[BM00]{Burger00}
M.~Burger and S.~Mozes.
\newblock Groups acting on trees: from local to global structure.
\newblock {\em Inst. Hautes \'Etudes Sci. Publ. Math.}, (92):113--150 (2001),
  2000.

\bibitem[BMW12]{Baumgartner12}
U.~{Baumgartner}, R.~G. M{\"o}ller, and G.~A. Willis.
\newblock Hyperbolic groups have flat-rank at most 1.
\newblock {\em Israel J. Math.}, 190:365--388, 2012.

\bibitem[Bou98]{Bourbaki98}
N.~Bourbaki.
\newblock General topology. chapters 1--4,.
\newblock {\em Springer}, 6:1--3, 1998.

\bibitem[BRW16]{Baumgartner15}
U.~{Baumgartner}, J.~Ramagge, and G.~A. Willis.
\newblock Scale-multiplicative semigroups and geometry: automorphism groups of
  trees.
\newblock {\em Groups Geom. Dyn.}, 10(3):1051--1075, 2016.

\bibitem[BW06]{Baumgartner06}
U.~{Baumgartner} and G.~A. Willis.
\newblock The direction of an automorphism of a totally disconnected locally
  compact group.
\newblock {\em Math. Z.}, 252(2):393--428, 2006.

\bibitem[{Byw}19]{Bywaters19}
T.~P. {Bywaters}.
\newblock {The space of directions for hyperbolic totally disconnected locally
  compact groups}.
\newblock {\em arXiv e-prints}, Oct 2019.

\bibitem[CH12]{Chat12}
Z.~Chatzidakis and E.~Hrushovski.
\newblock An invariant for difference field extensions.
\newblock {\em Ann. Fac. Sci. Toulouse Math. (6)}, 21(2):217--234, 2012.

\bibitem[Cor13]{corn13}
Y.~Cornulier.
\newblock {Group actions with commensurated subsets, wallings and cubings}.
\newblock {\em ArXiv e-prints}, February 2013.

\bibitem[CRW17]{Caprace17}
P.-E. {Caprace}, C.~D. {Reid}, and P.~{Wesolek}.
\newblock {Approximating simple locally compact groups by their dense locally
  compact subgroups}.
\newblock {\em ArXiv e-prints}, June 2017.

\bibitem[DSW06]{Dani06}
S.~Dani, N.~A. Shah, and G.~A. Willis.
\newblock Locally compact groups with dense orbits under
  $\mathbb{Z}^{d}$-actions by automorphisms.
\newblock {\em Ergodic Theory and Dynamical Systems}, 26(5):1443--1465, 2006.

\bibitem[GGT18]{garrido2018}
A.~Garrido, Y.~Glasner, and S.~Tornier.
\newblock {\em Automorphism groups of trees: generalities and prescribed local
  actions}, page 92–116.
\newblock London Mathematical Society Lecture Note Series. Cambridge University
  Press, 2018.

\bibitem[HM81]{Hofmann81}
K.~H. Hofmann and A.~Mukherjea.
\newblock Concentration functions and a class of noncompact groups.
\newblock {\em Math. Ann.}, 256(4):535--548, 1981.

\bibitem[JRW96]{Jaworski96}
W.~Jaworski, J.~Rosenblatt, and G.~A. Willis.
\newblock Concentration functions in locally compact groups.
\newblock {\em Math. Ann.}, 305(4):673--691, 1996.

\bibitem[LB16]{Boudec15}
A.~Le~Boudec.
\newblock Groups acting on trees with almost prescribed local action.
\newblock {\em Comment. Math. Helv.}, 91(2):253--293, 2016.

\bibitem[M{\"o}l02]{Moe02}
R.~G. M{\"o}ller.
\newblock Structure theory of totally disconnected locally compact groups via
  graphs and permutations.
\newblock {\em Canadian Journal of Mathematics}, 54(4):795--827, 2002.

\bibitem[PRW17]{Praeger17}
C.~E. {Praeger}, J.~{Ramagge}, and G.~{Willis}.
\newblock {A graph-theoretic description of scale-multiplicative semigroups of
  automorphisms}.
\newblock {\em ArXiv e-prints}, October 2017.

\bibitem[PW03]{Previts03}
W.~H Previts and T.~S. Wu.
\newblock Dense orbits and compactness of groups.
\newblock {\em Bulletin of the Australian Mathematical Society},
  68(1):155--159, 2003.

\bibitem[Rei13]{Reid13}
C.~D. Reid.
\newblock Local {S}ylow theory of totally disconnected, locally compact groups.
\newblock {\em J. Group Theory}, 16(4):535--555, 2013.

\bibitem[SW13]{Shalom13}
Y.~Shalom and G.~A. Willis.
\newblock Commensurated subgroups of arithmetic groups, totally disconnected
  groups and adelic rigidity.
\newblock {\em Geom. Funct. Anal.}, 23(5):1631--1683, 2013.

\bibitem[vD31]{vDa31}
D.~van Dantzig.
\newblock {\em Studien over topologische algebra}.
\newblock University of Groningen, 1931.

\bibitem[Wil]{WillisNeretin}
G.~A. Willis.
\newblock A canonical form and scale for almost automorphisms of the tree.
\newblock In Preparation.

\bibitem[Wil94]{Willis94}
G.~A. Willis.
\newblock The structure of totally disconnected, locally compact groups.
\newblock {\em Math. Ann.}, 300(2):341--363, 1994.

\bibitem[Wil01]{Willis01}
G.~A. Willis.
\newblock Further properties of the scale function on a totally disconnected
  group.
\newblock {\em J. Algebra}, 237(1):142--164, 2001.

\bibitem[Wil04]{Willis04}
G.~A. Willis.
\newblock Tidy subgroups for commuting automorphisms of totally disconnected
  groups: an analogue of simultaneous triangularisation of matrices.
\newblock {\em New York J. Math.}, 10:1--35 (electronic), 2004.

\bibitem[Wil14]{Willis14}
G.~A. Willis.
\newblock The nub of an automorphism of a totally disconnected, locally compact
  group.
\newblock {\em Ergodic Theory Dynam. Systems}, 34(4):1365--1394, 2014.

\bibitem[Wil15]{Willis15}
G.~A. Willis.
\newblock The scale and tidy subgroups for endomorphisms of totally
  disconnected locally compact groups.
\newblock {\em Math. Ann.}, 361(1-2):403--442, 2015.

\end{thebibliography}
\bibliographystyle{alpha}

\end{document}